\documentclass[a4paper]{article}
\usepackage{authblk}
\usepackage{maths_symbols}
\numberwithin{equation}{section}
\begin{document}
\title{Generalised Rank-Constrained Approximations of Hilbert--Schmidt Operators on Separable Hilbert Spaces and Applications}

\author[]{Giuseppe Carere\thanks{~giuseppe.carere@uni-potsdam.de, ORCID ID: 0000-0001-9955-4115} }
\author[]{Han Cheng Lie\thanks{~han.lie@uni-potsdam.de, ORCID ID: 0000-0002-6905-9903}}
\affil[]{~Institut f\"ur Mathematik, Universit\"at Potsdam, Potsdam OT Golm 14476, Germany}
\renewcommand\Affilfont{\small}

\date{}
\maketitle

\begin{abstract}
	In this work we solve, for given bounded operators $B,C$ and Hilbert--Schmidt operator $M$ acting on potentially infinite-dimensional separable Hilbert spaces, the reduced rank approximation problem,
\begin{align*}
	\min\{\norm{M-BXC}_{L_2}:\ \dim\ran{X}\leq r\}.
\end{align*}
This extends the result of Sondermann (\textit{Statistische Hefte}, 1986) and Friedland and Torokhti (\textit{SIAM J. Matrix Analysis and Applications}, 2007), which studies this problem in the case of matrices $M$, $B$, $C$, $X$, and the analysis involves the Moore--Penrose inverse. In classical approximation problems that can be solved by the singular value decomposition or Moore--Penrose inverse, the solution satisfies a minimal norm property. Friedland and Torokhti state such a minimal norm property of the solution. We show that this minimal norm property does not hold in general and give a modified minimality property that does hold. We show that the solution may be discontinuous in infinite-dimensional settings. We give conditions for continuity of the solutions and construct continuous approximations when such conditions are not met. Finally, we study problems from signal processing, reduced rank regression and linear operator learning under a rank constraint. Our theoretical results enable us to explicitly find solutions to these problems and to characterise their existence, uniqueness and minimality property.
\end{abstract}


\section{Introduction}
\label{sec:intro}
We study the following reduced rank approximation in Hilbert--Schmidt norm $\norm{\cdot}_{L_2}$,
\begin{align}
	\min\{\norm{M-BXC}_{L_2}:\ \dim\ran{X}\leq r\},
	\label{eqn:informal_problem}
\end{align}
where $r\in\N$ and $M$, $B$, $X$, $C$ are suitable linear operators. That is, we want to find the best approximation to $M$ of the form $BXC$ for given $B$ and $C$ in Hilbert--Schmidt norm. Such approximations have also been called `generalised approximations', see e.g. \cite{friedland_generalized_2007,li2023generalized}.  In the finite-dimensional case, the operators can be represented as matrices by expressing them in orthonormal bases. The Hilbert--Schmidt norm then becomes the Frobenius norm in the matrix formulation of the problem. 

In the finite-dimensional matrix case, the generalised approximation problem is well studied, cf. \cite{sondermann1986best,friedland_generalized_2007,soto2021error}. In \cite[Theorem 2.1]{friedland_generalized_2007} and \cite{sondermann1986best} a solution of the finite-dimensional matrix problem is formulated. Furthermore, conditions for uniqueness of the solution are given. A minimality property of the solution is also stated in \cite[Theorem 2.1]{friedland_generalized_2007}, namely that it should have minimal norm among all solutions of the approximation problem. In \cite{soto2021error} an error analysis of the solution proposed by \cite[Theorem 2.1]{friedland_generalized_2007} is added.

Next, we state the main contributions of this work, review some of the literature on applications of generalised rank-constrained approximations---with or without the minimal norm property---introduce notation and outline the rest of the paper.

\subsection{Contributions of this work}
\label{subsec:contributions}
Firstly, the theoretical contributions of this work are as follows:
\begin{itemize}
	\item 
		We show in \Cref{sec:finite_dim_case} that a conclusion in \cite[Theorem 2.1]{friedland_generalized_2007}---namely, that the solution of \eqref{eqn:informal_problem} in a finite-dimensional matrix formulation satisfies a minimal norm property---is not in general correct. In many applications, the minimal norm property is often seen as a desirable property of the solution.
	\item
		We propose a refinement of \cite[Theorem 2.1]{friedland_generalized_2007} by modifying the minimal norm property in \Cref{sec:finite_dim_case}. This modified property is always satisfied by the solution given in \cite[Theorem 2.1]{friedland_generalized_2007} and can still be interpreted as a minimality property. 
	\item
		We generalise the refined statement to an infinite-dimensional setting \Cref{thm:main}, and note that the infinite-dimensional approximation problem does not always admit an optimal solution by giving a necessary and sufficient existence condition.
	\item
		In the case where the generalised approximation problem \eqref{eqn:informal_problem} admits one or more optimal solutions, we characterise the solutions that satisfy the modified minimality property in \Cref{thm:main}, as well as the solutions that do not satisfy the modified minimality property in \Cref{cor:all_solutions}. In the case where \eqref{eqn:informal_problem} does not admit optimal solutions, we construct a sequence of approximate minimisers of the objective $\norm{M-BXC}_{L_2}$ in \Cref{subsec:approximation}.
	\item
		We calculate the minimal approximation error that can be attained for the problem \eqref{eqn:informal_problem} in \Cref{subsec:optimal_error}.
	\item
		We show that solutions of \eqref{eqn:informal_problem} are not bounded in general in \Cref{sec:unbounded_solution}, due to the Moore--Penrose inverse of $C$ that appears in the solution expression. It is even possible that some solutions are continuous, i.e. bounded,  and have finite Hilbert--Schmidt norm, while others are discontinuous, i.e. unbounded, and as such do not even have finite operator norm.
	\item
		We give several necessary or sufficient conditions for all solutions of \eqref{eqn:informal_problem} to be bounded in \Cref{subsec:boundedness_conditions}. 
	\item
		In the case where the problem \eqref{eqn:informal_problem} does not admit any bounded solutions, we propose a procedure to construct bounded approximations in \Cref{subsec:bounded:approximations}. This approximation procedure also allows us to formulate an adjoint problem to \eqref{eqn:informal_problem}, which in certain cases can be useful to recover bounded approximations, see \Cref{rmk:adjoint_formulation}.
\end{itemize}
While the Moore--Penrose inverse is defined for certain kinds of unbounded operators such as closed range operators, we restrict our analysis to the setting where $B$ and $C$ not only have closed range, but are bounded linear operators.

The second aim of this work is to demonstrate how our results can be used in relevant applications of current research interest such as operator learning \cite{kovachki_neural_2024} and in particular linear operator learning \cite{de_hoop_convergence_2023,mollenhauer_learning_2023}, reduced rank regression \cite{izenman_reduced-rank_1975,kacham2021reduced,turri_randomized_2023} and signal processing \cite{torokhti2009towards,torokhti2009filtering}. 
In such applications, there is a direct interest in an infinite-dimensional formulation. For example, as mentioned in \cite{mollenhauer_learning_2023}, the setting of infinite-dimensional linear operator learning  has direct applications to functional regression with functional response and conditional mean embeddings. In such an operator learning setting, given samples from the distribution of Hilbert space-valued random variables $x$ and $y$, the goal is to learn a linear operator $A$ that satisfies
\begin{align}
	\min\{\mathbb{E}\norm{y-Ax}^2: A\in\mathcal{A}\},
	\label{eqn:informal_application}
\end{align}
where $\mathcal{A}$ can be for example the bounded linear or Hilbert--Schmidt operators. This is also studied as a linear signal processing problem in \cite{torokhti2009towards}, where $x$ and $y$ are $\R^m$ or $\C^m$ valued. In \cite{torokhti2009towards} $x$, $y$, and $A$ are also allowed to be reweighted and the rank of $A$ can be restricted. Under such a rank restriction, the above problem can in fact also be seen as a reduced rank regression problem similar to \cite{izenman_reduced-rank_1975,turri_randomized_2023}. We show that these problems can be rewritten in formulation \eqref{eqn:informal_problem}. In \cite{kacham2021reduced}, a reduced rank regression problem of the form \eqref{eqn:informal_problem} is considered where the Hilbert--Schmidt norm is replaced by the operator norm. Applications from the perspective of linear operator learning, signal processing and reduced rank regression thus all motivate the study of a problem of form \eqref{eqn:informal_application}. Using the obtained solution of problem \eqref{eqn:informal_problem} in its adjoint formulation in the infinite-dimensional setting, we therefore analyse problem \eqref{eqn:informal_application}, where $\mathcal{A}$ is the space of finite-rank operators between separable Hilbert spaces. We also allow for reweighting of $x,y$ and $A$ as in \cite{torokhti2009towards}, and give necessary and sufficient conditions for a solution to exist, constructing also an explicit solution that satisfies a minimality property. In this application, we show that our proposed minimality property of the solutions corresponds to a maximal kernel property that naturally occurs in linear operator learning, c.f.\ \cite[Remark 3.12(b)]{mollenhauer_learning_2023}.

Even in finite-dimensional settings, the study of an infinite-dimensional analogue can be beneficial. Infinite-dimensional problems are often discretised and transformed to finite-dimensional generalised rank-constrained approximation problems. This happens for example in the study of inverse problems governed by partial differential equations, e.g. \cite{chung2017optimal}. An infinite-dimensional version of the generalised approximation problem would then enable the analysis of such problems directly in their infinite-dimensional formulation, allowing us to draw conclusions that are independent of the discretisation dimension.

\subsection{Related literature}
\label{subsec:literature}
In finite-dimensional settings, the reduced rank approximation problem \eqref{eqn:informal_problem} has been applied in the fields of signal processing \cite{torokhti2009towards,torokhti2007towards,torokhti2023optimal,torokhti2019improvement,soto2019improvement,torokhti2009filtering,soto2021data,soto2016generalized,torokhti2009generic,soto2021second,soto2022regularized,torokhti2021new,soto2019improvement,soto2017optimal,grant2014efficient,torokhti2012optimal,torokhti2018best}, data transmission systems \cite{torokhti2023matrix,soto2023fast}, prescribed eigenvalue approximation \cite{li2023generalized}, regularized inverse problems \cite{chung2017optimal}, subspace estimation in computer vision \cite{larsson2015simple}, neuroimaging \cite{langs2014decoupling}, classification \cite{mineiro2015fast}, recommender systems \cite{zhao2016predictive,chen2018visual}, image processing \cite{chung2013computing} and in distributed memory computing \cite{boutsidis2016optimal}. 
It is also used for tensor approximation \cite{saibaba2016hoid,begovic2022hybrid}, low-rank posterior approximation for linear-Gaussian inverse problems \cite{spantini_optimal_2015,spantini2017goal,bousserez2018optimal}, randomised matrix approximation algorithms \cite{song2017low,bakshi2020robust}, regression \cite{xiang2012optimal}, Bayes-risk minimisation in linear inverse problems \cite{chung2014efficient}, generative modeling \cite{langs2011learning} and fast least squares problems in statistical learning \cite{yang2020towards}. The analogue of \eqref{eqn:informal_problem} without the rank constraint is also used in other numerical linear algebra tasks \cite{holodnak2015topics}. Further applications are mentioned in \cite{chavarria2022effective}. 

For treatment of a similar problem to \eqref{eqn:informal_problem} in finite dimensions, in which the operator norm is used instead of the Hilbert--Schmidt norm, see \cite{sou2012generalized,kacham2021reduced}.

In \cite[Theorem 2.1]{friedland_generalized_2007} a solution is proposed for the reduced rank approximation problem, and a minimal norm property is stated for this solution.
This minimal norm property is reiterated in other works, \cite{boutsidis2016optimal,torokhti2007computational,li2023generalized,torokhti2009filtering,soto2021data,chung2017optimal,friedland2016matrices}, and used in \cite{howlett2022multilinear,yu2012rank,torokhti2009towards,torokhti2023matrix,soto2023fast,torokhti2018best,larsson2015simple}. Note that \cite{soto2021error}, which adds an error analysis, does not state the minimal norm property. 
The minimal norm property also has direct use in applications, as for example in the context of dynamic scene reconstruction from image measurements of \cite{larsson2015simple}. There, it is mentioned that the choice of solutions has a large effect on the quality of the reconstruction, and a minimal norm solution favours points close to the camera plane.

Imposing a minimality property has the consequence of mitigating nonuniqueness of the solution. This can also be achieved by imposing a regularizing term $\alpha^2\norm{X}_{L_2}^2$. This is the approach of \cite{chung2015optimal} for a similar reduced rank approximation problem in the context of inverse problems.

\subsection{Notation}
\label{subsec:notation}

For a matrix $Y\in\C^{s\times t}$, we denote by $(Y)_r$ a rank-$r$ truncated singular value decomposition (SVD) of $Y$. Thus if $Y=U\Sigma V^*$ is an SVD of $Y$, then $(Y)_r=\sum_{i=1}^{r}\sigma_i u_i v_i^*$ where $\sigma_i=\Sigma_{ii}$ are nonincreasing and $u_i$ and $v_i$ are the $i$-th column of $U$ and $V$ respectively. The Frobenius norm of $Y$ is denoted as $\norm{Y}_F$. By $Y^\dagger\in\C^{t\times s}$ we denote the Moore--Penrose inverse, also known as the pseudo-inverse or generalised inverse. 
We define $P_{Y,L}\coloneqq \sum_{i=1}^{\rank{Y}}u_iu_i^*$ and $P_{Y,R}\coloneqq  \sum_{i=1}^{\rank{Y}}v_iv_i^*$. Thus, $P_{Y,L}$ and $P_{Y,R}$ are the matrix representations in the standard bases of the projectors onto $\ran{Y}$ and $\ker{Y}^\perp$ respectively. It holds that $P_{Y,L}=YY^\dagger$ and $P_{Y,R}=Y^\dagger Y$. We note that the equality $P_{Y,L}=YY^\dagger$ is no longer true in general in a infinite-dimensional setting.

We let $\mathcal{H}$ and $\mathcal{K}$ be possibly infinite-dimensional separable Hilbert spaces over $\C$ or $\R$. By $\B(\mathcal{H}$, $\mathcal{K})$, $\B_0(\mathcal{H}$, $\mathcal{K})$, $\B_{00}(\mathcal{H}$,  $\mathcal{K})$ and $L_2(\mathcal{H},\mathcal{K})$, we denote the linear spaces of bounded, compact, finite-rank and Hilbert--Schmidt operators between Hilbert spaces $\mathcal{H}$ and $\mathcal{K}$. The Hilbert--Schmidt norm of a linear operator between finite-dimensional spaces is equal to the Frobenius norm of any matrix representation of this operator in an orthonormal basis (ONB). By $\B_{00,r}(\mathcal{H},\mathcal{K})$ we denote the set of finite-rank operators, i.e. operators that are bounded, have finite-dimensional range, and have rank at most $r\in\N$. Further details of these classes of operators is given in \Cref{sec:elements_fa}.

\subsection{Outline}
\label{subsec:outline}
In \Cref{sec:finite_dim_case}, we recall the statement of \cite[Theorem 2.1]{friedland_generalized_2007}. We then apply it to an elementary example in order to motivate the need for an alternative minimality property. In \Cref{sec:main_result} we then formulate and prove a refinement of the statement of \cite[Theorem 2.1]{friedland_generalized_2007} in the infinite-dimensional setting, with the modified minimality property. There, we also recover the modified finite-dimensional result of \cite{friedland_generalized_2007}, characterise solutions that do not satisfy the minimality property, and construct approximate minimisers. In \Cref{sec:unbounded_solution}, it is shown that solutions may be unbounded. In \Cref{sec:bounded_solutions}, we discuss the cases in which solutions are bounded, and approximate unbounded solutions by bounded ones. We apply our results in a more concrete setting in \Cref{sec:signal_processing_reduced_rank_regression}, which includes examples from signal processing, linear operator learning and reduced rank regression. In \Cref{sec:elements_fa} we recall certain concepts from functional analysis that we need in order to study the generalised rank-constrained approximation problem \eqref{eqn:informal_problem} in infinite dimensions. Finally, \Cref{sec:proofs} contains some proofs of minor results used in the text.

\section{Optimal approximation for matrices}
\label{sec:finite_dim_case}
In this section we recall the statement of \cite[Theorem 2.1]{friedland_generalized_2007} and discuss the minimal norm property stated therein. 

The statement of \cite[Theorem 2.1]{friedland_generalized_2007}, which provides an expression \eqref{eqn:optimal_matrix} for a solution of a rank-constrained approximation problem \eqref{eqn:matrix_problem} and asserts that any such solution has a certain minimal norm property \eqref{eqn:matrix_constraint}, is:
\newline
\newline
\textbf{Claim} of \cite[Theorem 2.1]{friedland_generalized_2007}: Let $M\in\C^{m\times n}$, $B\in\C^{m\times p}$ and $C\in\C^{q\times n}$ be given.
Then,
\begin{align}
	\hat{X}=B^\dagger(P_{B,L}MP_{C,R})_r C^\dagger
	\label{eqn:optimal_matrix}
\end{align}
solves the low-rank approximation problem
\begin{align}
	\min{\left\{\norm{M-BXC}_{F}:\ X\in\C^{p\times q},\ \rank{X}\leq r\right\} },
	\label{eqn:matrix_problem}
\end{align}
and satisfies the minimal norm property
\begin{align}
	\norm{X}_{F}=\min\left\{ \norm{\tilde{X}}_{F}:\ \tilde{X}\text{ solves }\eqref{eqn:matrix_problem} \right\}.
	\label{eqn:matrix_constraint}
\end{align}

A few remarks are in order. Firstly, in general not all solutions of \eqref{eqn:matrix_problem} are of the form \eqref{eqn:optimal_matrix}. Secondly, as $P_{C,R}=C^\dagger C$ and as in finite dimensions it holds that $P_{B,L}=BB^\dagger$, one may write \eqref{eqn:optimal_matrix} as $B^\dagger (BB^\dagger XC^\dagger C)_rC^\dagger$.
Thirdly, also a necessary and sufficient condition for uniqueness of $\hat{X}$ is given in \cite[Theorem 2.1]{friedland_generalized_2007}.
Finally, if $(P_{B,L}MP_{C,R})_r$ is not uniquely specified, then the matrix $\hat{X}$ in \eqref{eqn:optimal_matrix} is not uniquely determined. Hence \cref{eqn:optimal_matrix} should be interpreted as a set of solutions. Furthermore, this lack of uniqueness extends also to $\norm{\hat{X}}_{F}$: if there exists more than one solution of the form \eqref{eqn:optimal_matrix}, then these solutions might not all have the same norm. In other words, whether $\hat{X}$ in \eqref{eqn:optimal_matrix} is of minimal norm or not in general depends on the choice of truncated SVD $(P_{B,L}MP_{C,R})_r$. This fact indicates that the claim above is not sufficiently precise.

The following example demonstrates that if there is more than one solution to \eqref{eqn:matrix_problem} of the form \eqref{eqn:optimal_matrix}, then these solutions may have different Frobenius norm.

\begin{example}
	\label{ex:instructive_example}
Let us take $n=m=2,\ p=q=3,\ r=1$, and
\begin{align}
	M= 
	\begin{pmatrix}
		1 & 0 \\
		0 & 1
	\end{pmatrix},\quad
	B=C^\top=
	\begin{pmatrix}
		1 & 0 & 0\\
		0 & \frac{1}{2} & 0 \\
	\end{pmatrix}. 
	\label{eqn:example_definitions}
\end{align}
Then,
\begin{align*}
	C^\dagger = (B^\dagger)^\top=
	\begin{pmatrix}
		1 & 0 & 0\\
		0 & 2 & 0\\
	\end{pmatrix},\quad
	P_{B,L}=P_{C,R}=
	\begin{pmatrix}
		1 & 0\\
		0 & 1
	\end{pmatrix}.
\end{align*}
Hence $M=P_{B,L}MP_{C,R}$ and its rank-1 truncated SVDs are
\begin{align*}
	M_{1,a}\coloneqq
	\begin{pmatrix}
		1 & 0 \\
		0 & 0
	\end{pmatrix},\quad
	M_{1,b}\coloneqq
	\begin{pmatrix}
		0 & 0 \\
		0 & 1 \\
	\end{pmatrix}.
\end{align*}
The two minimizing matrices of the form \eqref{eqn:optimal_matrix} then are
\begin{align*}
	\hat{X}_{a}\coloneqq
	\begin{pmatrix}
		1 & 0 & 0\\
		0 & 0 & 0\\
		0 & 0 & 0\\
	\end{pmatrix},\quad
	\hat{X}_b\coloneqq
	\begin{pmatrix}
		0 & 0 & 0\\
		0 & 4 & 0\\
		0 & 0 & 0\\
	\end{pmatrix}.
\end{align*}
However, the Frobenius norms of $\hat{X}_a,\hat{X}_b$ differ, and thus cannot both be minimal. In particular, \eqref{eqn:matrix_constraint} does not hold.
\end{example}
The preceding example suggests that \cite[Theorem 2.1]{friedland_generalized_2007} can be made more precise, either by removing the minimum norm condition \eqref{eqn:matrix_constraint}, or by replacing this condition by a weaker condition which all solutions $\hat{X}$ given by \eqref{eqn:optimal_matrix} satisfy. Such a weaker condition must thus be independent of the choice of truncated SVD $(P_{B,L}MP_{C,R})_r$. Furthermore, if $C$ is a linear operator mapping between infinite-dimensional spaces, then ${C}^\dagger$ will in general not be bounded (cf. \Cref{prop:moore_penrose}), and thus cannot be Hilbert--Schmidt.  It can then happen that also solutions $\hat{X}$ of the form \eqref{eqn:optimal_matrix} are not Hilbert--Schmidt, in which case $\norm{\hat{X}}_{L_2}$ may be ill-defined, and we need a minimality condition that does not involve the Frobenius norm in order to generalise this condition to the infinite-dimensional setting.  For these two reasons we motivate an alternative condition, by continuing our Example \ref{ex:instructive_example}.

\addtocounter{theorem}{-1}
\begin{example}	[Continued]
	Let us put 
	\begin{align*}
		\hat{X}_{a}(\alpha) \coloneqq
		\begin{pmatrix}
			\multicolumn{2}{c}{\multirow{2}{*}{$X_{a,11}$}} & \alpha_1 \\
			 & & \alpha_2  \\ 
			\alpha_3 & \alpha_4 & \alpha_5
		\end{pmatrix},\quad
		\hat{X}_{b}(\alpha) \coloneqq
		\begin{pmatrix}
			\multicolumn{2}{c}{\multirow{2}{*}{$X_{b,11}$}} & \alpha_1 \\
			 & & \alpha_2  \\ 
			\alpha_3 & \alpha_4 & \alpha_5
		\end{pmatrix},
	\end{align*}
	with $X_{a,11}=\big(\begin{smallmatrix}
	  1 & 0\\
	  0 & 0
	\end{smallmatrix}\big)$ 
	and $X_{b,11}=\big(\begin{smallmatrix}
	  0 & 0\\
	  0 & 4
	\end{smallmatrix}\big)$,
	so $\hat{X}_a(0)=\hat{X}_a$ and $\hat{X}_b(0)=\hat{X}_b$. For $c\in\{a,b\}$, we notice that $\hat{X}_c(\alpha)$ also solves the minimisation problem \eqref{eqn:matrix_problem}, for any $\alpha=(\alpha_i)_{i=1}^5\in\C^5$, as $B\hat{X}_{a}(\alpha)C=\left( \begin{smallmatrix}1 & 0 \\ 0 & 0\end{smallmatrix} \right)$ and $B\hat{X}_{b}(\alpha)C=\left( \begin{smallmatrix}0 & 0 \\ 0 & 1\end{smallmatrix} \right)$.
	Hence, $B\hat{X}_{c}(\alpha)C=B\hat{X}_c(0)C$ for all $\alpha$. Furthermore, any solution to \eqref{eqn:matrix_problem} can be written as $\hat{X}_c(\alpha)$ for some $\alpha\in\mathbb{C}^5$ and $c\in\{a,b\}$. To see this, let us write an element of $\C^{3\times 3}$ as, for $x_{11},x_{12},x_{21},x_{22}\in\C$,
	\begin{align*}
	X(\alpha)\coloneqq 
	\begin{pmatrix}
		x_{11} & x_{12} & \alpha_1 \\
		x_{21} & x_{22} & \alpha_2 \\
		\alpha_3 & \alpha_4 & \alpha_5
	\end{pmatrix}.
	\end{align*}
	Like $B\hat{X}_c(\alpha)C$, also $BX(\alpha)C$ does not depend on $\alpha$. By the Eckart--Young Theorem (c.f.\ \Cref{prop:eckart_young}), for $X(\alpha)$ to be a solution of \eqref{eqn:matrix_problem}, we need $BX(\alpha)C$ to be a rank-1 SVD of $M$. That is, we need either $x_{11}=1$ or $x_{22}=4$, and all other entries to be equal to zero. 

	Thus, among all the solutions of \eqref{eqn:matrix_problem}, i.e. among all $\hat{X}_c(\alpha)$, only $\hat{X}_a(0)$ and $\hat{X}_b(0)$ are of the form \eqref{eqn:optimal_matrix}, and only $\hat{X}_a(0)$ satisfies \eqref{eqn:matrix_constraint}. 
	For this example, we now construct an alternative to \eqref{eqn:matrix_constraint}, which both $\hat{X}_a(0)$ and $\hat{X}_b(0)$ satisfy and which can still be interpreted as a minimality condition.

	Let $e_i$ denote the $i$-th canonical basis vector of $\C^{3}$ and $V\coloneqq\Span{e_1,e_2}\subset\C^{3}$. Let $P_V$ the matrix representation of the projector onto $V$ in the canonical basis of $\mathbb{C}^3$. Then for $B,C$ defined in \eqref{eqn:example_definitions}, $P_{B,R}=P_V=P_{C,L}$. We observe $P_{B,L}\hat{X}_c(\alpha)P_{C,R} = P_V\hat{X}_{c}(\alpha)P_V=\hat{X}_{c}(0)$ for any $\alpha$, and compute,
	\begin{align}
		\norm{\hat{X}_{c}(\alpha)}_{F}^2 = \norm{X_{c,11}}_{F}^2 + \sum_{i=1}^{5}\abs{\alpha_i}^2\geq \norm{X_{c,11}}_{F}^2 = \norm{\hat{X}_c(0)}_{F}^2.
		\label{eqn:minimiality_of_example}
	\end{align}
	Hence, for fixed $c\in\{a,b\}$, among the solutions $\hat{X}_c(\alpha)$, only $\hat{X}_c(0)$ has minimal Hilbert--Schmidt norm. We recall that the $\hat{X}_c(\alpha)$ satisfy $P_V\hat{X}_c(\alpha)P_V=\hat{X}_c(\alpha)$ if and only if $\alpha=0$. That is, it is precisely the solutions $\hat{X}$ of \eqref{eqn:matrix_problem} which satisfy 
	\begin{align}
		\hat{X} = P_{B,R}\hat{X}P_{C,L},
		\label{eqn:minimality_condition}
	\end{align}
	that are the minimisers in \eqref{eqn:minimiality_of_example} for some $c\in\{a,b\}$.
	Hence \eqref{eqn:minimality_condition} provides an alternative for the minimal norm property \eqref{eqn:matrix_constraint}. Thus, while $X_a(0)$ and $X_b(0)$ might not have the same norm, $X_c(0)$ has minimal norm among all $X_c(\alpha),\alpha\in\C^5$, for fixed $c\in\{a,b\}$, and $X_a(0),X_b(0)$ are in turn those solutions of \eqref{eqn:matrix_problem} which satisfy \eqref{eqn:minimality_condition}.
\end{example}

To generalise the reasoning of Example \ref{ex:instructive_example} as to why \eqref{eqn:minimality_condition} can be interpreted as a minimality condition, notice that \eqref{eqn:minimality_condition} is equivalent to $(I-P_{B,R})\hat{X}P_{C,L}=0$ and $\hat{X}(I-P_{C,L})=0$, since any $X\in\C^{p\times q}$ can be decomposed as
\begin{align*}
	X = P_{B,R}XP_{C,L}+(I-P_{B,R})XP_{C,L}+X(I-P_{C,L}).
\end{align*}
Now, $(I-P_{B,R})\hat{X}P_{C,L} = 0$ is equivalent to the property that $\hat{X}$ maps $\ran{C}$ into $\ker{B}^\perp$ and $\hat{X}(I-P_{C,L})=0$ is equivalent to the property that $\hat{X}$ maps $\ran{C}^\perp$ to $\{0\}$. In other words, if we express $\hat{X}\in\C^{p\times q}$ relative to the decompositions $\ran{C}\oplus\ran{C}^\perp$ of its domain and $\ker{B}^\perp\oplus\ker{B}$ of its range,
\begin{align*}
	\hat{X} =
	\begin{pmatrix}
		X_{11} & X_{12} \\
		X_{21} & X_{22}
	\end{pmatrix},
\end{align*}
with  $X_{11}\in\C^{s\times t}$, $X_{12}\in\C^{s\times (q-t)}$, $X_{21}\in\C^{(p-s)\times t}$, $X_{22}\in\C^{(p-s)\times(q-t)}$, $t\coloneqq\dim\ran{C}$, $s\coloneqq\dim\ker{B}^\perp$, then $X_{21}$, $X_{12}$, $X_{22}$ are all zero. If \eqref{eqn:minimality_condition} fails to hold, then at least one of the blocks $X_{12},X_{21},X_{22}$ must be nonzero, and the Hilbert--Schmidt norm of $\hat{X}$ is larger than that of $X_{11}$. Thus, we can view \eqref{eqn:minimality_condition} as expressing a minimality property, without referring to a norm.

If $\hat{X}$ in \eqref{eqn:optimal_matrix} is unique, then $\norm{\hat{X}}_{L_2}$ is indeed minimal among all minimisers $X$ of \eqref{eqn:matrix_problem}, in which case the statement of \cite[Theorem 2.1]{friedland_generalized_2007} is still valid.
Of course, this holds only when Frobenius norms can in fact be computed, which is no longer true in general for infinite-dimensional spaces.
In the next section we introduce a statement similar to \cite[Theorem 2.1]{friedland_generalized_2007} with the modified minimality property \eqref{eqn:minimality_condition} in an infinite-dimensional setting.

\section{Optimal approximations for Hilbert--Schmidt operators on separable Hilbert spaces}
\label{sec:main_result}
We are now ready to formulate \eqref{eqn:optimal_matrix} and \eqref{eqn:matrix_problem} with the minimality property \eqref{eqn:minimality_condition} in infinite dimensions. In \Cref{subsec:inf_dim_problem}, we give necessary and sufficient conditions for a solution to exist and describe the solution. An interpretation of the minimality property is also given, as well as the finite-dimensional modification of \cite[Theorem 2.1]{friedland_generalized_2007}.  When no solution exists, i.e. when the minimum Hilbert--Schmidt approximation error is not attained, one can still explicitly construct arbitrarily good approximations, as is shown in \Cref{subsec:approximation}. When a solution does exist, the corresponding optimal Hilbert--Schmidt approximation error is computed in \Cref{subsec:optimal_error}. The theory below involves the use of unbounded linear operators, SVDs of compact operators and Moore--Penrose inverses, which are recalled in \Cref{sec:elements_fa}.

\subsection{Problem formulation and solution}
\label{subsec:inf_dim_problem}
As in the finite-dimensional setting of \Cref{sec:finite_dim_case}, we again need the orthogonal projectors $P_{\overline{\ran{A}}}$ and $P_{\ker{A}^\perp}$ onto the closure of the range and orthogonal complement of the kernel of a bounded linear operator $A$. We note by \eqref{eqn:third_moore_penrose} that $P_{\ker{A}^\perp}=A^\dagger A$ as in the finite-dimensional setting. However, unlike the finite-dimensional setting, it holds by \eqref{eqn:fourth_moore_penrose} that $P_{\overline{\ran{A}}}$ merely extends $AA^\dagger$, as $AA^\dagger$ is only defined on the dense subspace $\dom{A^\dagger}$, whereas $P_{\overline{\ran{A}}}$ is defined everywhere. \begin{setting}
	Let $\mathcal{H}_1,\mathcal{H}_2,\mathcal{H}_3,\mathcal{H}_4$ be separable Hilbert spaces, let $r\in\N$ and let $M\in L_2(\mathcal{H}_1,\mathcal{H}_4)$, $B\in\B(\mathcal{H}_3,\mathcal{H}_4)$ and $C\in\B(\mathcal{H}_1,\mathcal{H}_2)$ be given. We define
	\begin{align*}
		\mathcal{X}_r &\coloneqq \left\{X:\mathcal{H}_2\rightarrow\mathcal{H}_3:\ \dom{X}\supset\dom{C^\dagger},\dim\ran{X}\leq r, BXC\in L_2(\mathcal{H}_1,\mathcal{H}_4)\right\},\\
		\mathcal{Y}_{r}&\coloneqq\left\{ Y\in\B_{00,r}(\mathcal{H}_1,\mathcal{H}_4):\ \ran{Y}\subset\dom{B^\dagger},\right.\\
		&\left.\qquad\forall\tilde{Y}\in\B_{00,r}(\mathcal{H}_1,\mathcal{H}_4),\ \norm{P_{\overline{\ran{B}}}MP_{\ker{C}^\perp}-Y}_{L_2}\leq \norm{P_{\overline{\ran{B}}}MP_{\ker{C}^\perp}-\tilde{Y}}_{L_2} \right\}.
	\end{align*}
	\label{setting:main}
\end{setting}
Thus in \Cref{setting:main}, $\mathcal{X}_r$ contains in particular all operators of rank at most $r$, which are bounded by definition. We stress that $\mathcal{X}_r$ in general also contains unbounded operators whose range has dimension at most $r$. Furthermore, $\mathcal{Y}_{r}$ is the set of solutions of the reduced rank problem 
\begin{align}
	\min\{\norm{P_{\overline{\ran{B}}} MP_{\ker{C}^\perp}-Y}_{L_2}:\ Y\in\mathcal{B}_{00,r}(\mathcal{H}_1,\mathcal{H}_4)\},
	\label{eqn:equivalent_problem}
\end{align}
that also map into $\dom{B^\dagger}$.
Hence by \Cref{prop:eckart_young}, $\mathcal{Y}_{r}$ consists of all rank-$r$ truncated SVDs of $P_{\overline{\ran{B}}}MP_{\ker{C}^\perp}$, as defined in \eqref{eqn:definition_svd}, that map into $\dom{B^\dagger}$.

We can now state the infinite-dimensional problem and its solution.
\begin{theorem}
	\label{thm:main}
	Assume \Cref{setting:main}. There exists a solution to the problem
	\begin{align}
		\label{eqn:reduced_rank_approximation_problem}
		\min\{\norm{M-BXC}_{L_2}:\ X\in\mathcal{X}_r\}
	\end{align}
	if and only if $\mathcal{Y}_{r}\not=\emptyset$.
	In that case, for each $(P_{\overline{\ran{B}}}MP_{\ker{C}^\perp})_r\in\mathcal{Y}_{r}$ a solution is given by
	\begin{align}
		\label{eqn:hat_X}
		\hat{X}:\dom{C^\dagger}\subset\mathcal{H}_2\rightarrow \mathcal{H}_3,\quad \hat{X}= B^\dagger (P_{\overline{\ran{B}}}MP_{\ker{C}^\perp})_rC^\dagger.
	\end{align}
	Furthermore, $\hat{X}$ satisfies
	\begin{align}
		\label{eqn:minimality_constraint}
		X = P_{\ker{B}^\perp} X P_{\overline{\ran{C}}}\text{ on }\dom{X}.
	\end{align}
	This solution is unique, in the sense that $X=\tilde{X}$ on $\dom{C^\dagger}$ for solutions $X,\tilde{X}$ of \eqref{eqn:reduced_rank_approximation_problem} satisfying \eqref{eqn:minimality_constraint}, if and only if one of the following conditions holds:
	\begin{enumerate}[label=(\alph*)]
		\item
			$P_{\overline{\ran{B}}}MP_{\ker{C}^\perp}\in\B_{00,r}(\mathcal{H}_1,\mathcal{H}_4),$
			\label{eqn:uniqueness_condition_a}

		\item
			$\sigma_r({P_{\overline{\ran{B}}}MP_{\ker{C}^\perp}})>\sigma_{r+1}({P_{\overline{\ran{B}}}MP_{\ker{C}^\perp}}).$
			\label{eqn:uniqueness_condition_b}
	\end{enumerate}
\end{theorem}

The conditions \ref{eqn:uniqueness_condition_a} or \ref{eqn:uniqueness_condition_b} ensure the rank-$r$ reduced SVD in \eqref{eqn:hat_X} is unique, c.f.\ \Cref{subsec:svd}.
Furthermore, if $\ran{M}\subset\dom{B}^\dagger$, then $P_{\overline{\ran{B}}}M = BB^\dagger M$ by \eqref{eqn:fourth_moore_penrose} and
\begin{align*}
	\ran{(P_{\overline{\ran{B}}}MP_{\ker{C}^\perp})_r}\subset\ran{P_{\overline{\ran{B}}}MP_{\ker{C}^\perp}}\subset\ran{B}.
\end{align*}
The condition $\ran{(P_{\overline{\ran{B}}}MP_{\ker{C}^\perp})_r}\subset\dom{B^\dagger}$ is then satisfied.
\begin{remark}[Solution domain]
	\label{rmk:solution_extension}
	In general, not all solutions of \eqref{eqn:reduced_rank_approximation_problem} that satisfy \eqref{eqn:minimality_constraint} are given by \eqref{eqn:hat_X}, as they might have a domain larger than $\dom{C^\dagger}=\dom{X^\dagger}$, but they agree with $\hat{X}$ on $\dom{C^\dagger}$, i.e. they extend $\hat{X}$.
\end{remark}
\begin{remark}[Equivalent domain condition]
	\label{rmk:equivalent_range_condition}
	Since $\ran{(P_{\overline{\ran{B}}}MP_{\ker{C}^\perp})_r}\subset\ran{P_{\overline{\ran{B}}}MP_{\ker{C}^\perp}}\subset \overline{\ran{B}}$ and $\dom{B}^\dagger=\ran{B}\oplus\ran{B}^\perp$, the requirement $\ran{(P_{\overline{\ran{B}}}MP_{\ker{C}^\perp})_r}\subset \dom{B}^\dagger$ in the definition of $\mathcal{Y}_{r}$ is equivalent to $\ran{(P_{\overline{\ran{B}}}MP_{\ker{C}^\perp})_r}\subset\ran{B}$.
\end{remark}
\begin{remark}[Case with $B$, $C$ having closed range]
	\label{rmk:closed_range_case}
	By \Cref{prop:moore_penrose} \cref{item:mp_inv_boundedness}, the following hold:
	\begin{enumerate}
		\item
			\label{item:C_closed_range}
			If $C$ has closed range, then $\mathcal{X}_r=B_{00,r}(\mathcal{H}_2,\mathcal{H}_3)$ and $P_{\overline{\ran{C}}}=CC^\dagger$. In this case, all solutions of \eqref{eqn:reduced_rank_approximation_problem} that satisfy \eqref{eqn:minimality_constraint} are given by \eqref{eqn:hat_X}.
	
		\item
			If $B$ has closed range, then $\dom{B^\dagger}=\mathcal{H}_4$, so $\ran{(P_{\overline{\ran{B}}}MP_{\ker{C}^\perp})_r}\subset \dom{B^\dagger}$ trivially, hence $\mathcal{Y}_{r}=\{(P_{\overline{\ran{B}}}MP_{\ker{C}^\perp})_r\}$ and so $\mathcal{Y}_{r}$ is nonempty. 
	\end{enumerate}
\end{remark}

In order to prove the theorem, we define the following solution spaces:
\begin{align*}
	\mathcal{S}_r&\coloneqq\left\{ X\in\mathcal{X}_{r}:\ \forall\tilde{X}\in\mathcal{X}_r,\ \norm{M-BXC}_{L_2}\leq \norm{M-B\tilde{X}C}_{L_2}\right\},\\
	\mathcal{S}_{r,*}&\coloneqq\left\{ X\in\mathcal{S}_r:\ X = P_{\ker{B}^\perp} X P_{\overline{\ran{C}}} \text{ on } \dom{X}\right\}.
\end{align*}
Thus, $\mathcal{S}_{r}$ is the set of solutions to the reduced rank problem \eqref{eqn:reduced_rank_approximation_problem}. Further, $\mathcal{S}_{r,*}$ is the set of solutions to \eqref{eqn:reduced_rank_approximation_problem} that also satisfy \eqref{eqn:minimality_constraint}. 
We must show, under the existence condition $\mathcal{Y}_{r}\not=\emptyset$, that $\hat{X}$ defined in \eqref{eqn:hat_X} lies in $\mathcal{S}_{r,*}$, and that $X=\tilde{X}$ on $\dom{C^\dagger}$ for any $X,\tilde{X}\in\mathcal{S}_{r,*}$ if and only if the uniqueness condition \ref{eqn:uniqueness_condition_a} or \ref{eqn:uniqueness_condition_b} holds.
To this end, we show that solutions to \eqref{eqn:reduced_rank_approximation_problem} and \eqref{eqn:minimality_constraint} can be constructed by a solution to \eqref{eqn:equivalent_problem} which maps into $\dom{B^\dagger}$ and vice versa, thereby drawing a connection between the sets $\mathcal{S}_{r,*}$ and $\mathcal{Y}_{r}$. The following lemma is crucial for this purpose. 

\begin{lemma}
	\label{lemma:projected_problem}
	Let $M$, $B$, $C$ and $\mathcal{X}_r$ be as in \Cref{thm:main} and $c\coloneqq \norm{M}_{L_2}^2-\norm{P_{\overline{\ran{B}}}MP_{\ker{C}^\perp}}_{L_2}^2$. For any $X\in\mathcal{X}_r$,
	\begin{align}
		\norm{M-BXC}_{L_2}^2 &= \norm{P_{\overline{\ran{B}}}MP_{\ker{C}^\perp}-BXC}_{L_2}^2 + c.
		\label{eqn:error_fried_torok_2}
	\end{align}
\end{lemma}

\begin{proof}[Proof of \Cref{lemma:projected_problem}]
	For any separable Hilbert space $\mathcal{K}$ with ONB $(e_i)_i$, we have, 
	\begin{align*}
		\langle P_{\overline{\ran{B}}}S,(I-P_{\overline{\ran{B}}})T\rangle_{L_2} = \sum_{i}^{}\langle P_{\overline{\ran{B}}}Se_i,(I-P_{\overline{\ran{B}}})Te_i\rangle = 0,\quad T,S\in  L_2(\mathcal{K},\mathcal{H}_4),\\
		\langle SP_{\ker{C}^\perp},T(I-P_{\ker{C}^\perp})\rangle_{L_2} = \langle (I-P_{\ker{C}^\perp})T^*,P_{\ker{C}^\perp}S^*\rangle_{L_2}=0,\quad T,S\in L_2(\mathcal{H}_1,\mathcal{K}).
	\end{align*}
	Furthermore, the Moore--Penrose equations \eqref{eqn:first_moore_penrose}, \eqref{eqn:third_moore_penrose} and \eqref{eqn:fourth_moore_penrose} in \Cref{prop:moore_penrose} state that $P_{\overline{\ran{B}}}B=BB^\dagger B=B$ and $CP_{\ker{C}^\perp}=CC^\dagger C = C$.
	Therefore, for any $X\in\mathcal{X}_r$, we have $P_{\overline{\ran{B}}}BXCP_{\ker{C}^\perp}=BXC\in L_2(\mathcal{H}_1,\mathcal{H}_4)$, and
	\begin{alignat*}{3}
		\norm{M-BXC}_{L_2}^2 &=&& \norm{P_{\overline{\ran{B}}}MP_{\ker{C}^\perp}-BXC}_{L_2}^2 + \norm{(I-P_{\overline{\ran{B}}})MP_{\ker{C}^\perp}}_{L_2}^2  \\
		&&&+ \norm{P_{\overline{\ran{B}}}M(I-P_{\ker{C}^\perp})}_{L_2}^2 + \norm{(I-P_{\overline{\ran{B}}})M(I-P_{\ker{C}^\perp})}_{L_2}^2\\
		&=&& \norm{P_{\overline{\ran{B}}}MP_{\ker{C}^\perp}-BXC}_{L_2}^2 + c.
	\end{alignat*}
\end{proof}

Using \Cref{lemma:projected_problem}, we can now relate the sets $\mathcal{S}_r$ and $\mathcal{Y}_{r}$.

\begin{lemma}
	\label{lemma:solution_characterisation}
	Let $M$, $B$, $C$ and $\mathcal{X}_r$ be as in \Cref{thm:main}. Then the following statements hold.
	\begin{enumerate}
		\item
			For $Y\in\mathcal{Y}_{r}$ it holds that ${B^\dagger YC^\dagger}$ is a well-defined linear operator on $\dom{C^\dagger}$, which satisfies $B^\dagger YC^\dagger\in\mathcal{S}_r$ and $BB^\dagger YC^\dagger C=Y$.
			\label{item:surjectivity}
		\item 
			For $X\in\mathcal{X}_r$, $BXC\in\mathcal{Y}_{r}$ if and only if $X\in\mathcal{S}_r$.
			\label{item:well-defined_and_reducing}
		\item 
			For $Y\in\mathcal{Y}_{r}$ it holds that $B^\dagger YC^\dagger\in\mathcal{S}_{r,*}$.
			\label{item:inverse_range}
		\item
			For $X\in\mathcal{S}_{r,*}$ there exists $Y\in\mathcal{Y}_{r}$ such that $X=B^\dagger YC^\dagger$ on $\dom{C^\dagger}$.
			\label{item:inverse_range_description}

	\end{enumerate}
\end{lemma}

If we define the map $j:\mathcal{X}_r\rightarrow \B_{00,r}(\mathcal{H}_1,\mathcal{H}_4)$, $X\mapsto BXC$, then the ``if'' part of \ref{item:well-defined_and_reducing} in \Cref{lemma:solution_characterisation} implies that $j(\mathcal{S}_r)\subset\mathcal{Y}_{r}$ and the ``only if'' part implies that $j(\mathcal{X}_r\backslash\mathcal{S}_r)\subset \B_{00,r}(\mathcal{H}_1,\mathcal{H}_4)\backslash\mathcal{Y}_{r}$. \Cref{item:surjectivity} states that $\restr{j}{\mathcal{S}_r}$ maps onto $\mathcal{Y}_{r}$, hence has a right inverse, and that a right inverse is given by the map $Y\mapsto B^\dagger Y C^\dagger$ on $\dom{C^\dagger}$. By \cref{item:inverse_range}, the image of this right inverse maps into $\mathcal{S}_{r,*}$. By \cref{item:inverse_range_description}, every $X\in\mathcal{S}_{r,*}$ agrees with $B^\dagger Y C^\dagger$ on $\dom{C^\dagger}$, for some $Y\in\mathcal{Y}_{r}$. 	

\begin{proof}[Proof of \Cref{lemma:solution_characterisation}]
	\ref{item:surjectivity}\quad Let $Y\in\mathcal{Y}_{r}$. That is, let $Y$ be a rank-$r$ truncated SVD of $P_{\overline{\ran{B}}}MP_{\ker{C}^\perp}$ satisfying $\ran{Y}\subset\dom{B^\dagger}$. 
	It follows that $B^\dagger Y$ is well-defined on all of $\mathcal{H}_2$. By \Cref{lemma:unbounded_after_finite_rank}, $B^\dagger Y$ is in fact bounded and of rank at most $r$. Thus, by the definition of compositions of linear maps in \Cref{subsec:linear_operators},
	\begin{align*}
		\dom{B^\dagger YC^\dagger} = (C^\dagger)^{-1}\left(\dom{B^\dagger Y}\right)= \dom{C^\dagger},
	\end{align*}
	and 
	\begin{align*}
		\dim\ran{B^\dagger YC^\dagger} \leq \rank{Y}\leq r.
	\end{align*}
	To show that $B^\dagger Y C^\dagger\in\mathcal{X}_r$, it remains to show $B(B^\dagger YC^\dagger)C\in L_2(\mathcal{H}_1,\mathcal{H}_4)$. Now, $BB^\dagger$ is the identity on $\ran{B}$ by \eqref{eqn:fourth_moore_penrose}. As $\ran{Y}\subset\ran{B}$, it holds that $BB^\dagger Y = Y$. Further,
	\begin{align*}
		\ker{Y}\supset\ker{P_{\overline{\ran{B}}}MP_{\ker{C}^\perp}}\supset\ker{P_{\ker{C}^\perp}} = \ker{C},
	\end{align*}
	where the first inclusion holds because $Y$ is a truncated SVD of $P_{\overline{\ran{B}}}MP_{\ker{C}^\perp}$. The inclusion $\ker{Y}\supset\ker{C}$ and \eqref{eqn:third_moore_penrose} imply $Y = YP_{\ker{Y}^\perp} = YP_{\ker{C}^\perp} = Y C^\dagger C$. We conclude,
	\begin{align}
		B(B^\dagger YC^\dagger)C = YC^\dagger C = Y.
		\label{eqn:truncated_svd_identity}
	\end{align}
	In particular, $B(B^\dagger YC^\dagger)C$ has rank at most $r$ and is hence Hilbert--Schmidt. Thus, $B^\dagger YC^\dagger\in\mathcal{X}_r$. 	
	Finally, $B^\dagger YC^\dagger\in\mathcal{S}_r$, since for any $\tilde{X}\in\mathcal{X}_r$, applying \eqref{eqn:error_fried_torok_2} twice, \eqref{eqn:truncated_svd_identity} and the definition of $Y$ yields
	\begin{align*}
		\norm{M-B\tilde{X}C}_{L_2}^2 &=  \norm{P_{\overline{\ran{B}}}MP_{\ker{C}^\perp}-B\tilde{X}C}_{L_2}^2+c 
		\geq \norm{P_{\overline{\ran{B}}}MP_{\ker{C}^\perp}-Y}_{L_2}^2+c \\
		&= \norm{P_{\overline{\ran{B}}}MP_{\ker{C}^\perp}-BB^\dagger YC^\dagger C}_{L_2}^2 +c = \norm{M-B(B^\dagger Y C^\dagger)C}_{L_2}^2.
	\end{align*}
	\ref{item:well-defined_and_reducing}\quad Let $X\in\mathcal{X}_r$, and let us assume that $BXC\in\mathcal{Y}_{r}$. By applying \eqref{eqn:error_fried_torok_2} twice and by the definition of $\mathcal{Y}_{r}$, we have for all $\tilde{X}\in\mathcal{X}_r$,
	\begin{align*}
		\norm{M-B\tilde{X}C}_{L_2}^2 &=  \norm{P_{\overline{\ran{B}}}MP_{\ker{C}^\perp} - B\tilde{X}C}_{L_2}^2+c \\
		&\geq\norm{P_{\overline{\ran{B}}}MP_{\ker{C}^\perp} - B{X}C}_{L_2}^2 + c = \norm{M - B{X}C}_{L_2}^2.
	\end{align*}
	Hence, $X\in\mathcal{S}_r$.
	Conversely, let us assume $X\in\mathcal{S}_r$. Let $Y\in\mathcal{Y}_{r}$. By \cref{item:surjectivity}, $B^\dagger YC^\dagger\in\mathcal{S}_r$. Using \eqref{eqn:error_fried_torok_2} twice and \cref{item:surjectivity}, we find for all $\tilde{Y}\in\B_{00,r}(\mathcal{H}_1,\mathcal{H}_4)$,
	\begin{align*}
		\norm{P_{\overline{\ran{B}}}MP_{\ker{C}^\perp}-\tilde{Y}}_{L_2}^2 
		&\geq \norm{P_{\overline{\ran{B}}}MP_{\ker{C}^\perp}-Y}_{L_2}^2 
		= \norm{M-B(B^\dagger YC^\dagger)C}_{L_2}^2 - c \\
		&\geq  \norm{M-BXC}_{L_2}^2 - c 
		= \norm{P_{\overline{\ran{B}}}MP_{\ker{C}^\perp}-BXC}_{L_2},
	\end{align*}
	where we used $Y\in\mathcal{Y}_{r}$ for the first inequality and $X\in\mathcal{S}_r$ for the second inequality.  Thus $BXC\in\mathcal{Y}_{r}$.
	\newline
	\ref{item:inverse_range}\quad Let $Y\in\mathcal{Y}_{r}$. By \cref{item:surjectivity}, $B^\dagger YC^\dagger\in\mathcal{S}_r$. By \eqref{eqn:second_moore_penrose}, \eqref{eqn:third_moore_penrose} and \eqref{eqn:fourth_moore_penrose},
	\begin{align*}
		P_{\ker{B}^\perp}B^\dagger YC^\dagger P_{\overline{\ran{C}}}=B^\dagger BB^\dagger YC^\dagger CC^\dagger=B^\dagger YC^\dagger\quad\text{on }\dom{C^\dagger}.
	\end{align*}
	Thus, $B^\dagger YC^\dagger\in\mathcal{S}_{r,*}$.
	\newline
	\ref{item:inverse_range_description}\quad Let $X\in\mathcal{S}_{r,*}\subset\mathcal{S}_r$, so $Y\coloneqq BXC\in\mathcal{Y}_{r}$ by \cref{item:well-defined_and_reducing}. By definition of $\mathcal{S}_{r,*}$  it holds that $X=P_{\ker{B}^\perp}XP_{\overline{\ran{C}}}$ on $\dom{C^\dagger}\subset\dom{X}$. Hence, using \eqref{eqn:third_moore_penrose} and \eqref{eqn:fourth_moore_penrose}, $X=B^\dagger (BXC)C^\dagger=B^\dagger YC^\dagger$ on $\dom{C^\dagger}$.

\end{proof}

\begin{proof}[Proof of \Cref{thm:main}]
	Suppose that $\mathcal{Y}_{r}\not=\emptyset$ and let $Y\in\mathcal{Y}_{r}$. Then \Cref{lemma:solution_characterisation} \cref{item:inverse_range} gives that $B^\dagger Y C^\dagger\in\mathcal{S}_{r,*}$, i.e. \eqref{eqn:hat_X} solves \eqref{eqn:reduced_rank_approximation_problem} and satisfies \eqref{eqn:minimality_constraint}. Conversely, suppose that \eqref{eqn:reduced_rank_approximation_problem} admits a solution $\hat{X}\in\mathcal{X}_r$, so $\hat{X}\in\mathcal{S}_r$. By \Cref{lemma:solution_characterisation} \cref{item:well-defined_and_reducing} and \Cref{prop:eckart_young}, $Y\coloneqq B\hat{X}C$ is a truncated rank-$r$ SVD of $P_{\overline{\ran{B}}}MP_{\ker{C}^\perp}$. But $\ran{Y}\subset\ran{B}\subset\dom{B^\dagger}$, so $Y\in\mathcal{Y}_{r}$.

	It remains to show that all elements of $S_{r,*}$ agree on $\dom{C^\dagger}$ if and only if condition \ref{eqn:uniqueness_condition_a} or \ref{eqn:uniqueness_condition_b} holds. By the construction of the SVD in \Cref{subsec:svd}, $\mathcal{Y}_{r}$ contains at most one element if and only if \ref{eqn:uniqueness_condition_a} or \ref{eqn:uniqueness_condition_b} holds. If $\mathcal{Y}_r$ contains at most one element, then by \cref{item:inverse_range_description} of \Cref{lemma:solution_characterisation}, all elements of $\mathcal{S}_{r,*}$ agree on $\dom{C^\dagger}$. Conversely, if all elements of $S_{r,*}$ agree on $\dom{C^\dagger}$, and if $Y_1,Y_2\in \mathcal{Y}_r$, then $B^\dagger Y_1 C^\dagger$ and $B^\dagger Y_2 C^\dagger$ agree on $\dom{C}^\dagger$ by \Cref{lemma:solution_characterisation}\ref{item:inverse_range}. Hence by \Cref{lemma:solution_characterisation}\ref{item:surjectivity}, $Y_1 = B(B^\dagger Y_1C^\dagger) C = B(B^\dagger Y_2C^\dagger) C=Y_2$, so that $\mathcal{Y}_r$ has at most one element. This concludes the proof of uniqueness.
\end{proof}

As a corollary to \Cref{thm:main} we obtain the result of \cite[Theorem 2.1]{friedland_generalized_2007} but with the modified minimality property. We recall that $P_{A,L}$ and $P_{A,R}$ are defined in \Cref{sec:finite_dim_case} and that $P_{A,L}= A A^\dagger$ and $P_{A,R}=  A^\dagger A$.
\begin{corollary}
	\label{cor:friedland_torokthi_case}
	Let $B\in\C^{m\times p}$, $C\in\C^{q\times n}$ and $M\in\C^{m\times  n}.$ Then 
	\begin{align*}
		\hat{X}=B^\dagger(P_{B,L}MP_{C,R})_rC^\dagger
	\end{align*}
	is a solution to the problem
	\begin{align*}
		\min\{\norm{M-BXC}_{F}:\ X\in\C^{p\times q},\rank{X}\leq r\}.
	\end{align*}
	Furthermore, $\hat{X}$ satisfies
	\begin{align}
		\label{eqn:corrected_miniminality_constraint}
		X=P_{B,R}XP_{C,L}.
	\end{align}
	This solution is the unique solution satisfying \eqref{eqn:corrected_miniminality_constraint} if and only if one of the following conditions holds:
	\begin{enumerate}[label=(\alph*)]
		\item
			$r\geq \rank{P_{B,L}MP_{C,R}},$
		\item
			$\sigma_r({P_{B,L}MP_{C,R}})>\sigma_{r+1}({P_{B,L}MP_{C,R}}).$
	\end{enumerate}
\end{corollary}

\begin{proof}
	We associate a matrix $A\in\C^{t\times s}$ with the linear operator $L_A\in\B(\C^s,\C^t)$ that left multiplies $x\in\C^s$ with $A$, i.e $L_Ax=Ax$. Hence $A$ is the expression of $L_A$ in the orthogonal standard bases of $\C^s$ and $\C^t$, so we have the equality of norms $\norm{L_A}_{L_2}=\norm{A}_F$. Furthermore, one may check that $L_{A}^\dagger=L_{A^\dagger}$, $(L_A)_r=L_{(A)_r}$, $P_{\overline{\ran{L_A}}}=P_{\ran{L_A}}=P_{\ran{A}}=L_{P_{A,L}}$ and $P_{\ker{L_A}^\perp}=P_{\ker{A}^\perp}=L_{P_{A,R}}$. As all finite-dimensional linear maps are Hilbert--Schmidt operators and have closed range, Remark \ref{rmk:closed_range_case} implies that we can apply \Cref{thm:main} to $L_M,L_B,L_C$ with $\mathcal{X}_r=\B_{00,r}(\C^q,\C^p)\simeq\{X\in\C^{p\times q},\rank{X}\leq r\}$ and where $\mathcal{Y}_{r}$ consists of the truncated rank-$r$ SVDs of $P_{{\ran{B}}}L_MP_{\ker{C}^\perp}$. The result now follows from the conclusions of \Cref{thm:main} by expressing the linear operators again as matrices.
\end{proof}

\Cref{thm:main} generalises the statement in \cite{friedland_generalized_2007} about existence and uniqueness of $\hat{X}$ of the optimal low-rank approximation problem \eqref{eqn:matrix_problem}--\eqref{eqn:matrix_constraint} in the finite-dimensional setting to the setting of possibly infinite-dimensional separable Hilbert spaces. The minimality property \eqref{eqn:matrix_constraint} in \cite[Theorem 2.1]{friedland_generalized_2007} has been replaced by \eqref{eqn:minimality_constraint} in \Cref{thm:main}. 

We now present an interpretation of \eqref{eqn:minimality_constraint} that parallels the interpretation of property \eqref{eqn:minimality_condition} in the finite-dimensional setting given at the end of \Cref{sec:finite_dim_case}.
Let us take $X\in\mathcal{X}_r$. Then $\dom{X}\supset\ran{C}$, so $\mathcal{D}\coloneqq\dom{X}\cap\overline{\ran{C}}$ contains $\ran{C}$. Therefore, $\mathcal{D}\subset\overline{\ran{C}}$ densely and $\mathcal{D}^\perp=\ran{C}^\perp$. Also, $\dom{X} = \mathcal{D}\oplus \ran{C}^\perp$. Indeed, when $h\in \dom{X}\subset\mathcal{H}_2=\overline{\ran{C}}\oplus \ran{C}^\perp$, then $h=h_1+h_2$ for some $h_1\in\overline{\ran{C}},h_2\in\ran{C}^\perp\subset\dom{C^\dagger}\subset\dom{X}$, hence $h_1=h-h_2\in \dom{X}$.

Let us decompose $\mathcal{H}_3 = \ker{B}^\perp \oplus \ker{B}$ and $\dom{X} = \mathcal{D}\oplus \ran{C}^\perp$ . We define
\begin{align*}
	X_{11}&\coloneqq \restr{\left(P_{\ker{B}^\perp} X P_{\overline{\ran{C}}}\right)}{\mathcal{D}}:\mathcal{D}\subset\overline{\ran{C}}\rightarrow \ker{B}^\perp,\\
	X_{12}&\coloneqq \restr{\left(P_{\ker{B}^\perp}XP_{\ran{C}^\perp}\right)}{\ran{C}^\perp}:\ran{C}^\perp\rightarrow \ker{B}^\perp,\\
	X_{21}&\coloneqq \restr{\left(P_{\ker{B}} X P_{\overline{\ran{C}}}\right)}{\mathcal{D}}:\mathcal{D}\subset\overline{\ran{C}}\rightarrow \ker{B},\\
X_{22}&\coloneqq \restr{\left(P_{\ker{B}}XP_{\ran{C}^\perp}\right)}{\ran{C}^\perp}:\ran{C}^\perp\rightarrow \ker{B}.
\end{align*}
While $X_{12},X_{22}$ are defined everywhere on $\ran{C}^\perp$, $X_{11},X_{21}$ are densely defined in $\overline{\ran{C}}$, as $\mathcal{D}\subset\overline{\ran{C}}$ densely. By \Cref{prop:moore_penrose} \cref{item:mp_inv_boundedness}, $X_{11},X_{21}$ are everywhere defined when $\ran{C}$ is closed. We can now write $X$ relative to these decompositions of $\mathcal{H}_2$ and $\mathcal{H}_3$ as
\begin{align}
	X = 
	\begin{pmatrix}
		X_{11} & X_{12} \\
		X_{21} & X_{22}
	\end{pmatrix},
	\label{eqn:inf_dim_decomposition}
\end{align}
with the understanding that for $h_1+h_2\in \mathcal{D}\oplus\ran{C}^\perp= \dom{X}$,
\begin{align*}
	X(h_1,h_2) &= 
	\begin{pmatrix}
		X_{11} & X_{12} \\
		X_{21} & X_{22}
	\end{pmatrix}
	\begin{pmatrix}
		h_1 \\
		h_2
	\end{pmatrix}\\
	&=
	(X_{11}h_1+X_{12}h_2,X_{21}h_1+X_{22}h_2)
	\in
	\ker{B}^\perp\oplus \ker{B}.
\end{align*}
Since
\begin{align}
	\label{eqn:operator_decomposition}
	X = P_{\ker{B}^\perp}XP_{\overline{\ran{C}}}+(I-P_{\ker{B}^\perp})XP_{\overline{\ran{C}}}+X(I-P_{\overline{\ran{C}}}),
\end{align}
the minimality condition \eqref{eqn:minimality_constraint} is equivalent to both $X(I-P_{\overline{\ran{C}}})=0$ and $(I-P_{\ker{B}^\perp})XP_{\overline{\ran{C}}}=0$ on $\dom{X}$. It follows, respectively, that $X_{12}=X_{22}=0$ and that $X_{21}=0$. This generalises the interpretation of \eqref{eqn:minimality_condition} in \Cref{sec:finite_dim_case} to infinite-dimensional separable Hilbert spaces: the minimality of $\hat{X}$ of the form \eqref{eqn:hat_X} is in the sense that all blocks in \eqref{eqn:inf_dim_decomposition} except for the upper left block $X_{11}$ are zero.

As noted in \Cref{rmk:solution_extension}, we have characterised $\mathcal{S}_{r,*}\cap\{\dom{X}=\dom{C}^\dagger\}$. We can now characterise $\mathcal{S}_r\cap\{\dom{X}=\dom{C}^\dagger \}$, that is, we characterise \textit{all} solutions of \eqref{eqn:reduced_rank_approximation_problem} on $\dom{C^\dagger}$, that need not satisfy the minimality property \eqref{eqn:minimality_constraint}.
\begin{corollary}
	\label{cor:all_solutions}
	The problem \eqref{eqn:reduced_rank_approximation_problem} admits a solution if and only if $\mathcal{Y}_r\not=\emptyset$. In that case,
	\begin{align*}
	&\mathcal{S}_r\cap\left\{ \dom{X}=\dom{C}^\dagger \right\}
	\\
	=&\left\{ \hat{X} + P_{\ker{B}}T + SP_{\ran{C}^\perp}:\ \hat{X}\in\mathcal{S}_{r,*},\ T:\dom{C}^\dagger\subset \mathcal{H}_2\rightarrow \mathcal{H}_3,\ S:\ran{C}^\perp\subset\mathcal{H}_2\rightarrow\mathcal{H}_3 \right\}.
	\end{align*}
\end{corollary}
\begin{proof}
	It follows directly from \Cref{thm:main} that \eqref{eqn:reduced_rank_approximation_problem} has a solution if and only if $\mathcal{Y}_r\not=\emptyset$.
	Now let $X=\hat{X}+SP_{\ran{C}^\perp}+P_{\ker{B}}T$ with $\hat{X}$, $S$ and $T$ as specified. Then $\dom{X}=\dom{\hat{X}}=\dom{C}^\dagger$ and $BXC=B\hat{X}C$. But $\hat{X}\in\mathcal{S}_r$, so $X\in\mathcal{S}_r$. This shows the ``$\supset$'' inclusion holds. Conversely, let $X\in\mathcal{S}_r$ with $\dom{X}=\dom{C^\dagger}$. We define $\hat{X}=B^\dagger BXCC^\dagger$, so $\hat{X}\in\mathcal{S}_{r,*}$ by \Cref{lemma:solution_characterisation} \Cref{item:well-defined_and_reducing,item:inverse_range}. In particular, using also \eqref{eqn:third_moore_penrose} and \eqref{eqn:fourth_moore_penrose}, we have on $\dom{C}^\dagger=\dom{X}$ that $P_{\ker{B}^\perp}XP_{\overline{\ran{C}}}=B^\dagger BXCC^\dagger =\hat{X}$. We also define 
	$S\coloneqq XP_{\ran{C}^\perp}$ and $T\coloneqq \restr{P_{\ker{B}}XP_{\overline{\ran{C}}}}{\dom{C}^\dagger}$. 
	Thus, $S$ and $T$ are well-defined, as $P_{\ran{C}^\perp}$ and $\restr{P_{\overline{\ran{C}}}}{\dom{C}^\dagger}$ both map into $\dom{C}^\dagger=\dom{X}$. Furthermore, $S=SP_{\ran{C}^\perp}$ and $T=P_{\ker{B}}T$. Hence using \eqref{eqn:operator_decomposition}:
	\begin{align*}
		X &= P_{\ker{B}^\perp}XP_{\overline{\ran{C}}}+P_{\ker{B}}XP_{\overline{\ran{C}}}+XP_{\ran{C}^\perp} \\
		&= \hat{X}+P_{\ker{B}}T+SP_{\ran{C}^\perp},
	\end{align*}
	where the second equality only holds on $\dom{X}=\dom{C}^\dagger$.
\end{proof}

We end this section by noting that the solutions of \eqref{eqn:reduced_rank_approximation_problem} are all unbounded when only the solutions that also satisfy the minimality property \eqref{eqn:minimality_constraint} are unbounded.

\begin{corollary}
	If $\mathcal{S}_{r,*}$ contains only unbounded solutions, then all solutions of problem \eqref{eqn:reduced_rank_approximation_problem} are unbounded.
	\label{cor:all_solutions_unbounded_when_minimal_solutions_unbounded}
\end{corollary}

\begin{proof}
	Let $\tilde{X}$ be a solution of \eqref{eqn:reduced_rank_approximation_problem}. By \Cref{thm:main} we have $\mathcal{Y}_r\not=\emptyset$. Thus, by \Cref{lemma:solution_characterisation}\ref{item:inverse_range}, $\mathcal{S}_{r,*}\not=\emptyset$. By \Cref{cor:all_solutions}, it then holds on $\dom{C^\dagger}$ that $\tilde{X}=\hat{X} + P_{\ker{B}}T + SP_{\ran{C}^\perp}$ for some $\hat{X}\in\mathcal{S}_{r,*}$, $T:\dom{C}^\dagger\subset \mathcal{H}_2\rightarrow \mathcal{H}_3$ and $S:\ran{C}^\perp\subset\mathcal{H}_2\rightarrow\mathcal{H}_3$. We have $\hat{X}=P_{\ker{B}^\perp}\hat{X}$ by \eqref{eqn:minimality_constraint}. For $h\in\ran{C}$ it holds that $P_{\ran{C}^\perp}h=0$, hence,
	\begin{align*}
		\norm{\tilde{X}h}^2 = \norm{\hat{X}h + P_{\ker{B}}Th}^2 = \norm{\hat{X}h}^2+\norm{P_{\ker{B}}Th}^2\geq \norm{\hat{X}h}^2.
	\end{align*}
	Let us assume $\hat{X}$ is unbounded. Then there exists a bounded sequence $(h_n)_n$ in $\dom{C^\dagger}$ satisfying $\norm{\hat{X}h_n}\rightarrow\infty$. Since $\hat{X}$ satisfies \eqref{eqn:minimality_constraint}, $\hat{X}h_n=\hat{X}P_{\ran{C}}h_n$. Hence $k_n\coloneqq P_{\ran{C}}h_n\in\ran{C}$ is a bounded sequence and $\norm{\hat{X}k_n}$ diverges. But then also $\norm{\tilde{X}k_n}$ diverges, showing that $\tilde{X}$ is unbounded.
\end{proof}

\subsection{The case of no solutions and approximate minimisers}
\label{subsec:approximation}
Even when the condition $\mathcal{Y}_{r}\not=\emptyset$ with $\mathcal{Y}_r$ defined in \Cref{setting:main} fails to hold, that is, even when $\ran{Y}\subset\dom{B}^\dagger$ does not hold for any rank-$r$ reduced SVD $Y$ of $P_{\overline{\ran{B}}}MP_{\ker{C}^\perp}$, one can still construct a minimizing sequence for \eqref{eqn:reduced_rank_approximation_problem} that satisfies \eqref{eqn:minimality_constraint}. To do so, suppose $Y\coloneqq(P_{\overline{\ran{B}}}MP_{\ker{C}^\perp})_r$ has SVD $\sum_{i=1}^{r}\lambda_i f_i\otimes e_i$, with orthonormal sequences $(e_i)_{i=1}^{r}$ and $(f_i)_{i=1}^{r}$ in $\mathcal{H}_3$ and $\mathcal{H}_2$ respectively. We recall the notation `$\otimes$' from \Cref{sec:elements_fa}. Since $f_i\in\ran{Y}\subset\overline{\ran{B}}$, we can find $f_i^m\in{\ran{B}}$ such that $\norm{f_i-f_i^m}<1/n$ whenever $m\geq n$ and $n\in\N$. Let us define $Y^{(n)}\coloneqq \sum_{i=1}^{r}\lambda_i f_i^n\otimes e_i$. Then, 
\begin{align*}
	\norm{Y-Y^{(n)}}_{L_2}^2 = \sum_{i=1}^{r}\lambda_i^2\norm{f_i^n-f_i}^2\leq r\lambda_1^2n^{-2}\rightarrow 0,\quad n\rightarrow\infty.
\end{align*}
As in the proof of \Cref{lemma:solution_characterisation} \cref{item:surjectivity}, we see that $\hat{X}^{(n)}\coloneqq B^\dagger Y^{(n)}C^\dagger$ is a well-defined element of $\mathcal{X}_r$ satisfying \eqref{eqn:minimality_constraint} and $B\hat{X}^{(n)}C=Y^{(n)}$. In particular, $B\hat{X}^{(n)}C\rightarrow Y$ in $L_2(\mathcal{H}_1,\mathcal{H}_4)$. Using \Cref{lemma:projected_problem} twice and \Cref{prop:eckart_young}, it follows that for $X\in\mathcal{X}_r$,
\begin{align*}
	\inf\{\norm{M-BXC}_{L_2}^2:\ X\in\mathcal{X}_r\}&\leq \lim_n\norm{M-B\hat{X}^{(n)}C}_{L_2}^2\\
	&=  \lim_n\norm{P_{\overline{\ran{B}}}MP_{\ker{C}^\perp}-B\hat{X}^{(n)}C}_{L_2}^2+c\\
	&=\norm{P_{\overline{\ran{B}}}MP_{\ker{C}^\perp}-Y}_{L_2}^2+c\\
	&\leq\norm{P_{\overline{\ran{B}}}MP_{\ker{C}^\perp}-BXC}_{L_2}^2+c \\
	&= \norm{M-BXC}_{L_2}^2.
\end{align*}
By taking the infimum over $X\in\mathcal{X}_r$ we see that $\norm{M-B\hat{X}^{(n)}C}_{L_2}$ converges to $\inf\{\norm{M-BXC}_{L_2}:\ X\in\mathcal{X}_r\}$. We conclude that $(\hat{X}^{(n)})_n$ is a sequence of approximate minimisers, i.e. $\norm{M-B\hat{X}^{(n)}C}_{L_2}\rightarrow \inf\{\norm{M-BXC}_{L_2},\ X\in\mathcal{X}_r\}$, and each element in this sequence satisfies the minimality property $\hat{X}^{(n)}=P_{\ker{B}^\perp}\hat{X}^{(n)}P_{\overline{\ran{C}}}$ on $\dom{\hat{X}^{(n)}}=\dom{C}^\dagger.$ 

\subsection{Optimal approximation error}
\label{subsec:optimal_error}
The quality of the optimal approximation \eqref{eqn:hat_X} is addressed in the following result. The result also implies that the optimal approximation error converges to 0 as $r\rightarrow\infty$.
\begin{proposition}
	\label{prop:optimal_error}
	Assume \Cref{setting:main} and $\mathcal{Y}_{r}\not=\emptyset$. For any $\hat{X}$ given by \eqref{eqn:hat_X} it holds that
	\begin{align*}
		\norm{M-B\hat{X}C}_{L_2}^2 =  \norm{M}_{L_2}^2 - \Delta,
	\end{align*}
	with $\Delta\coloneqq \sum_{i\leq r}^{}\abs{\sigma_i(P_{\overline{\ran{B}}}MP_{\ker{C}^\perp})}^2$. We have
	\begin{align*}
		\Delta  =   \sum_{i\leq r}^{}\sigma_i(P_{\overline{\ran{B}}}MC^\dagger CM^*P_{\overline{\ran{B}}}) 
		= \sum_{i\leq r}^{}\sigma_i(C^\dagger CM^*P_{\overline{\ran{B}}}MC^\dagger C).
	\end{align*}
	If additionally $\ran{M}\subset\dom{B^\dagger}$, then 
	$B^\dagger MC^\dagger CM^*B$ has only nonnegative eigenvalues and
		\begin{align*}
			\Delta =  \sum_{i\leq r}^{}\lambda_i(B^\dagger MC^\dagger CM^*B), 
		\end{align*}
	where $\lambda_i(\cdot)$ denotes the $i$-th largest eigenvalue.
\end{proposition}
\begin{proof}
	By \Cref{lemma:solution_characterisation} \cref{item:well-defined_and_reducing}, $B\hat{X}C$ is a rank-$r$ reduced SVD of $P_{\overline{\ran{B}}}MP_{\ker{C}^\perp}$.
	Thus, by \Cref{lemma:projected_problem},
	\begin{align*}
		\norm{M-B\hat{X}C}_{L_2}^2 &=  \norm{M}_{L_2}^2 + \norm{P_{\overline{\ran{B}}}MP_{\ker{C}^\perp}-B\hat{X}C}_{L_2}^2 -\norm{P_{\overline{\ran{B}}}MP_{\ker{C}^\perp}}_{L_2}^2 \\
		&= \norm{M}_{L_2}^2 + \sum_{i>r}^{}\abs{\sigma_i(P_{\overline{\ran{B}}}MP_{\ker{C}^\perp})}^2 - \sum_{i}^{}\abs{\sigma_i(P_{\overline{\ran{B}}}MP_{\ker{C}^\perp})}^2\\
		&= \norm{M}_{L_2}^2 - \Delta.
	\end{align*}
	By \Cref{prop:hs_calc_from_svd} and the fact that $P_{\overline{\ran{B}}}$ and $P_{\ker{C}^\perp}$ are orthogonal projectors,
	\begin{align*}
		\Delta= \sum_{i\leq r}^{}\sigma_i(P_{\overline{\ran{B}}}MP_{\ker{C}^\perp}(P_{\overline{\ran{B}}}MP_{\ker{C}^\perp})^*)  = \sum_{i\leq r}^{}\sigma_i(P_{\overline{\ran{B}}}MP_{\ker{C}^\perp}M^*P_{\overline{\ran{B}}}),
	\end{align*}
	and similarly
	\begin{align*}
		\Delta= \sum_{i\leq r}^{}\sigma_i((P_{\overline{\ran{B}}}MP_{\ker{C}^\perp})^*P_{\overline{\ran{B}}}MP_{\ker{C}^\perp})  = \sum_{i\leq r}^{}\sigma_i(P_{\ker{C}^\perp}M^*P_{\overline{\ran{B}}}MP_{\ker{C}^\perp}).
	\end{align*}
	Now we can use $P_{\ker{C}^\perp}=C^\dagger C$, c.f.\ \eqref{eqn:third_moore_penrose}. Finally, the assumption $\ran{M}\subset\dom{B^\dagger}$ and \eqref{eqn:fourth_moore_penrose} imply $P_{\overline{\ran{B}}}M = BB^\dagger M$. Hence, 
	\begin{align*}
		T\coloneqq P_{\overline{\ran{B}}}MP_{\ker{C}^\perp}M^*P_{\overline{\ran{B}}} = BB^\dagger M C^\dagger CM^*P_{\overline{\ran{B}}},
	\end{align*}
	and $T\in L_2(\mathcal{H}_2)$, $T=T^*$, $T\geq 0$. In particular, the singular values of $T$ equal the eigenvalues of $T$. To compute $\Delta$ under the assumption $\ran{M}\subset\dom{B}^\dagger$, we may thus use the eigenvalues of $T$.
	Suppose that there exists some nonzero $v$ such that $Tv=\lambda v$. Then $v\in\ran{B}\subset\dom{B^\dagger}$. With $w\coloneqq B^\dagger v$ it holds that $w\not=0$. By \eqref{eqn:second_moore_penrose} and \eqref{eqn:fourth_moore_penrose}, $B^\dagger MC^\dagger CM^*Bw= B^\dagger Tv =\lambda w$. Thus $\lambda$ is an eigenvalue of $B^\dagger MC^\dagger CM^*B$. Conversely, suppose there exists some nonzero $w$ such that $B^\dagger MC^\dagger CM^*Bw =\lambda w$. Then $w\not\in\ker{B}$. With $v\coloneqq Bw$ it holds that $v\not=0$. By \eqref{eqn:first_moore_penrose} and \eqref{eqn:fourth_moore_penrose}, $Tv=BB^\dagger MC^\dagger CM^* B w=B\lambda w=\lambda v$. Thus $\lambda$ is an eigenvalue of $T$. We conclude that
	\begin{align*}
		\Delta=\sum_{i\leq r}^{}\sigma_i(T)  = \sum_{i\leq r}^{}\lambda_i(T)= \sum_{i\leq r}^{}\lambda_i(B^\dagger MC^\dagger CM^*B).
	\end{align*}
\end{proof}

\section{Unbounded solution}
\label{sec:unbounded_solution}
We consider an example that shows that for a given approximation problem, all solutions $\hat{X}$ of \eqref{eqn:reduced_rank_approximation_problem} of the form \eqref{eqn:hat_X}---and in fact all solutions of \eqref{eqn:reduced_rank_approximation_problem}---are unbounded. This example also shows it is possible that some solutions of the form \eqref{eqn:hat_X} are bounded, while others are unbounded.

Let $\mathcal{H}$ be a separable Hilbert space and take $\mathcal{H}_i\coloneqq\mathcal{H}$ for $1\leq i\leq 4$. 
Let $(e_n)_n$ be an orthonormal basis of $\mathcal{H}$. Let $(\gamma_n)_n,(\mu_n)_n\in\C\backslash\{0\}$ converge to zero. We assume $(\mu_n)_n$ is nonincreasing. Let $(\alpha_n)_n\in\C$ be such that $\abs{\alpha_n}\rightarrow\infty$ and $(\alpha_n \gamma_n)_n$ is square summable. For example, $\gamma_n=n^{-2}$ and $\alpha_n=n$ satisfy these conditions.

Define $w\coloneqq \sum_{n\geq 2}^{}\alpha_n\gamma_n e_n\in\mathcal{H}$. Put $f_1\coloneqq w\norm{w}^{-1}$ and $f_2\coloneqq e_1$ and extend $\{f_1,f_2\}$ to an ONB $(f_n)_n$ of $\mathcal{H}$. We then set, for $h\in\mathcal{H}$,
\begin{align*}
	Mh = \sum_{n}^{}\mu_n\langle h,f_n\rangle f_n,\\
	Ch = \sum_{n}^{}\gamma_n\langle h,e_n\rangle e_n.
\end{align*}
Then $C$ is injective and $\ran{C}=\{h\in\mathcal{H}:\ \sum_{n}^{}\gamma_n^2\langle h,e_n\rangle^2<\infty\}$. As $\ran{C}$ contains $\{e_n\}_n$, it is dense in $\mathcal{H}$, thus $\ran{C}^\perp=\{0\}$. Therefore, $\dom{C^\dagger}=\ran{C}$ and the Moore--Penrose inverse of $C$ is given, as in \Cref{ex:moore_penrose_inverse}, by
\begin{align*}
	C^\dagger h = \sum_{n}^{}\gamma_n^{-1}\langle h,e_n\rangle e_n,\quad h\in\dom{C^\dagger}.
\end{align*}
Let us take $B=I$ for simplicity, so $P_{\overline{\ran{B}}}=I$. As $C$ is injective, $P_{\overline{\ran{B}}}MP_{\ker{C}^\perp}=M$. Since $(\mu_n)_n$ is nonincreasing, for any $r\geq 1$, one of the rank-$r$ truncated SVDs of $P_{\overline{\ran{B}}}MP_{\ker{C}^\perp}$ is given by
\begin{align}
	(P_{\overline{\ran{B}}}MP_{\ker{C}^\perp})_rh=\sum_{n=1}^{r}\mu_n \langle h,f_n\rangle f_n.
	\label{eqn:example_svd}
\end{align}
With this SVD for given $r\geq 1$, we now compute for $h\in\dom{C}^\dagger$,
\begin{align*}
	B^\dagger (P_{\overline{\ran{B}}}MP_{\ker{C}^\perp})_rC^\dagger h = \sum_{n=1}^{r}\sum_{m=1}^{\infty}\mu_n\gamma_m^{-1}\langle h,e_m\rangle\langle e_m,f_n\rangle f_n.
\end{align*}
As $f_1=w\norm{w}^{-1}$, it follows for $h=e_m\in\dom{C^\dagger}$, $m\geq 2$,
\begin{align*}
	\norm{B^\dagger (P_{\overline{\ran{B}}}MP_{\ker{C}^\perp})_rC^\dagger e_m}^2 &= \sum_{n=1}^{r}\Abs{\mu_n\gamma^{-1}_m\langle e_m,f_n\rangle}^2\\
	&\geq \Abs{\mu_1\gamma^{-1}_m\langle e_m,f_1\rangle}^2 = \abs{\mu_1}^2\abs{\alpha_m}^2\norm{w}^{-2}\rightarrow\infty
\end{align*}
as $m\rightarrow \infty$, showing that $B^\dagger (P_{\overline{\ran{B}}}MP_{\ker{C}^\perp})_rC^\dagger$ is unbounded. This shows that \eqref{eqn:hat_X} can indeed give unbounded solutions. If $\mu_1>\mu_2$ and $r=1$, then \eqref{eqn:hat_X} gives a unique solution, which shows that \eqref{eqn:hat_X} can also give only unbounded solutions. In this case, $\mathcal{S}_{r,*}$ only contains unbounded solutions, so that in fact all solutions of \eqref{eqn:reduced_rank_approximation_problem} are unbounded by \Cref{cor:all_solutions_unbounded_when_minimal_solutions_unbounded}.

Now let $r=1$ and $\mu_1=\mu_2$, so that $P_{\overline{\ran{B}}}MP_{\ker{C}^\perp}$ has at least two distinct rank-$1$ truncated SVDs. 
The rank-$1$ truncated SVD \eqref{eqn:example_svd}, i.e. $(P_{\overline{\ran{B}}}MP_{\ker{C}^\perp})_1 = \mu_1\langle \cdot,f_1\rangle f_1$,  yields the following solution, which is of the form \eqref{eqn:hat_X}:
\begin{align*}
	\hat{X}_a h\coloneqq\mu_1\langle C^\dagger h,f_1\rangle f_1 = \mu_1\sum_{n}^{}\gamma_n^{-1}\langle h,e_n\rangle \langle e_n,f_1\rangle f_1 = \mu_1\norm{w}^{-1}\sum_{n\geq 2}^{}\alpha_n\langle h,e_n\rangle f_1,
\end{align*}
for $h\in\dom{C^\dagger}$. By again taking $h=e_m$, $\hat{X}_a$ is seen to be unbounded.
As $f_2=e_1$, the rank-$1$ truncated SVD $\mu_2\langle \cdot,e_1\rangle e_1$ yields the solution 
\begin{align*}
	\hat{X}_b h\coloneqq \mu_2\langle C^\dagger h,e_1\rangle e_1 = \mu_2\gamma_1^{-1}\langle h,e_1\rangle e_1,\quad h\in\dom{C^\dagger},
\end{align*}
which is bounded. 

We have thus constructed two solutions of the form \eqref{eqn:hat_X}: one of these is bounded, in particular it has finite rank and has finite Hilbert--Schmidt norm; the other is unbounded.

\section{Bounded solutions and approximations}
\label{sec:bounded_solutions}
\Cref{thm:main} prescribes a set of solutions that are of the form \eqref{eqn:hat_X}. As shown in \Cref{sec:unbounded_solution}, this set can contain unbounded solutions. It is even possible that some solutions of the form \eqref{eqn:hat_X} are bounded and others are unbounded. It is natural to ask when all solutions are bounded. In \Cref{subsec:boundedness_conditions} we give necessary and sufficient conditions for \eqref{eqn:hat_X} to have only bounded solutions. In the case where they are not bounded, we use the concept of outer inverses to construct suitable bounded approximations in \Cref{subsec:bounded:approximations}.  An adjoint formulation of \eqref{eqn:reduced_rank_approximation_problem}-\eqref{eqn:minimality_constraint} is discussed in \Cref{rmk:adjoint_formulation}.

\subsection{Necessary and sufficient conditions for bounded solutions}
\label{subsec:boundedness_conditions}
We recall that \Cref{rmk:closed_range_case}\ref{item:C_closed_range} gives a sufficient condition for \eqref{eqn:hat_X} to consist of only bounded solutions, namely that $\ran{C}$ is closed in $\mathcal{H}_2$. An even stronger condition is given next.

\begin{lemma}
	Let $C\in\B(\mathcal{H}_1,\mathcal{H}_2)$. If $\ker{C}^\perp$ is finite-dimensional, then $\ran{C}$ is closed.
	\label{lemma:ker_perp_inf_dim}
\end{lemma}

\begin{proof}
	Since $\overline{\ran{C^*}}=\ker{C}^\perp$ by \Cref{prop:ker_ran_relation}, it follows that $\ran{C^*}$ is finite-dimensional when $\ker{C}^\perp$ is finite-dimensional. But finite-dimensional spaces are closed. By \Cref{prop:closed_range_thm}, then also $\ran{C}$ is closed.
\end{proof}

It follows that all $\hat{X}$ in \eqref{eqn:hat_X} are bounded when $\mathcal{H}_1$ is finite-dimensional. We now discuss a range condition on $C$ that is not just sufficient, but also necessary for $\hat{X}$ to be bounded.

\begin{proposition}
	\label{prop:bounded_solutions}
	In \Cref{setting:main} suppose $\mathcal{Y}_{r}\not=\emptyset$ and $(P_{\overline{\ran{B}}}MP_{\ker{C}^\perp})_r\in\mathcal{Y}_{r}$. Let $\hat{X}$ be given by \eqref{eqn:hat_X} and define $Z\coloneqq B^\dagger(P_{\overline{\ran{B}}}MP_{\ker{C}^\perp})_r$, so $\hat{X}=ZC^\dagger$. The following are equivalent:
	\begin{enumerate}
		\item $\hat{X}$ is bounded.
			\label{item:bounded}
		\item $\ker{\hat{X}}$ is closed in $\dom{C^\dagger}$.
			\label{item:closed_kernel}
		\item $C(\ker{Z}\cap\ker{C}^\perp)$ is closed in $\ran{C}$,
			\label{item:closed_range}
		\item $\norm{Zh}\leq \eta\norm{Ch}$ for all $h\in\mathcal{H}_1$ and some $\eta>0$,
			\label{item:upper_bound}
		\item
			$\ran{(P_{\overline{\ran{C^*}}}M^*P_{{\ker{B^*}^\perp}})_r}\subset\dom{(C^*)^\dagger}$.
			\label{item:adjoint}
	\end{enumerate}
\end{proposition}
\begin{proof}
	The equivalence of \cref{item:bounded,item:closed_kernel} is proven similarly as one proves the equivalence of continuity and having closed range for linear functionals, i.e. when $r=1$. 
	We outline a proof in the appendix as \Cref{lemma:closed_kernel_and_finite_dim_range}.
	To show the equivalence between \cref{item:closed_kernel,item:closed_range} we compute $\ker{\hat{X}}$.  As $\ran{C^\dagger}=\ker{C}^\perp$, c.f.\ \Cref{prop:moore_penrose} \cref{item:mp_inv_range}, we have
	\begin{align*}
		\ker{\hat{X}} &= \{k\in\dom{C^\dagger}:\ C^\dagger k\in\ker{Z}\}\\
		&= \{k\in\dom{C^\dagger}:\ C^\dagger k\in\ker{Z}\cap\ker{C}^\perp\}\\
		&= \{k_0+k_\perp\in\ran{C}\oplus\ran{C}^\perp:\ C^\dagger k_0\in\ker{Z}\cap\ker{C}^\perp\}.
	\end{align*}
	If $k_0\in\ran{C}$ is such that $C^\dagger k_0\in\ker{Z}\cap\ker{C}^\perp$, then $k_0=P_{\overline{\ran{C}}}k_0=CC^\dagger k_0\in C(\ker{Z}\cap\ker{C}^\perp)$ by \eqref{eqn:fourth_moore_penrose}. Conversely, if $k_0=Cz$ for some $z\in\ker{Z}\cap{\ker{C}}^\perp$, then $z=P_{\ker{C}^\perp}z = C^\dagger C z = C^\dagger k_0$ by \eqref{eqn:third_moore_penrose}. Thus,
	\begin{align*}
		\ker{\hat{X}}= C(\ker{Z}\cap\ker{C}^\perp)\oplus\ran{C}^\perp =  U\oplus \ran{C}^\perp,
	\end{align*}
	with $U\coloneqq C(\ker{C}\cap\ker{C}^\perp)$.
	If $U\oplus \ran{C}^\perp\subset \ran{C}\oplus \ran{C}^\perp$ is closed, then $U\subset\ran{C}$ is closed. Since $\ran{C}^\perp$ is closed, the converse holds as well. This shows that \cref{item:closed_kernel,item:closed_range} are equivalent. To show that \cref{item:bounded,item:upper_bound} are equivalent, we first note that $\ker{C}\subset\ker{B^\dagger(P_{\overline{\ran{B}}}MP_{\ker{C}^\perp})_r}$. Thus, $\hat{X}=ZC^\dagger$ implies $\hat{X}C = ZC^\dagger C = ZP_{\ker{C}^\perp}=ZP_{\ker{Z}^\perp}=Z$ by \eqref{eqn:third_moore_penrose}. We see that \cref{item:bounded} implies \cref{item:upper_bound} with $\eta\coloneqq\norm{\hat{X}}$. Conversely, if \cref{item:upper_bound} holds, then by \eqref{eqn:fourth_moore_penrose} we have $\norm{\hat{X}h}=\norm{ZC^\dagger h}\leq \eta\norm{CC^\dagger h}=\eta\norm{P_{\ker{C}^\perp}h}\leq \eta\norm{h}$ for $h\in\dom{C^\dagger}\subset\dom{\hat{X}}$. The fact $\dom{C^\dagger}\subset\dom{\hat{X}}$ follows by definition of $\mathcal{X}_r$ in \Cref{setting:main}. Since $\dom{C^\dagger}$ is dense in $\mathcal{H}_2$, we conclude that $\hat{X}$ is bounded on $\mathcal{H}_2$, showing \cref{item:bounded} holds. Finally, we show that \cref{item:adjoint} implies \cref{item:upper_bound} and that \cref{item:bounded} implies \cref{item:adjoint}. Suppose \cref{item:adjoint} holds. Let $(P_{\overline{\ran{C^*}}}M^*P_{\ker{B^*}^\perp})_r=\sum_{i=1}^{r}\alpha_i f_i\otimes e_i$ be a truncated rank-$r$ SVD, with $(\alpha_i)_i$ a nonnegative sequence and $(f_i)_i\subset\mathcal{H}_1$ and $(e_i)_i\subset\mathcal{H}_4$ orthonormal. Thus, we have $(P_{\overline{\ran{B}}}MP_{\ker{C}^\perp})_r=\sum_{i=1}^{r}\alpha_i e_i\otimes f_i$ by \Cref{prop:ker_ran_relation}. By \Cref{rmk:equivalent_range_condition}, \cref{item:adjoint} implies $\ran{(P_{\overline{\ran{C^*}}}M^*P_{\ker{B^*}^\perp})_r}\subset\ran{C^*}$, hence $f_i\in\ran{C^*}$ for each $i$. That is, $f_i=C^*h_i$ for some $(h_i)_i\subset\mathcal{H}_2$. Now, for $h\in\mathcal{H}_1$,
	\begin{align*}
		\norm{Zh}^2&= \norm{B(P_{\overline{\ran{B}}}MP_{\ker{C}^\perp})_rh}^2\leq\norm{B}\sum_{i=1}^{r}\alpha_i^2\langle h,f_i\rangle^2 =\norm{B}\sum_{i=1}^{r}\alpha_i^2\langle h,C^*h_i\rangle^2\\
		&=  \norm{B}\sum_{i=1}^{r}\alpha_i^2\langle Ch,h_i\rangle^2\leq \norm{B}\sum_{i=1}^{r}\alpha_i^2\norm{Ch}^2\norm{h_i}^2 = \left(\norm{B}\sum_{i=1}^{r}\alpha_i^2\norm{h_i}^2\right)\norm{Ch}^2,
	\end{align*}
	where we use consecutively the definition of $Z$, the definition of the operator norm and the orthonormality of $(e_i)_i$, the fact that $f_i=C^*h_i$, the definition of the adjoint, the Cauchy--Schwarz inequality, and rearranging. Thus, \cref{item:upper_bound} holds. If \cref{item:bounded} holds, then $(B\hat{X}C)^* = C^*\hat{X}^*B^*$. By \eqref{eqn:third_moore_penrose} and \eqref{eqn:fourth_moore_penrose}, $B\hat{X}C = BB^\dagger(P_{\overline{\ran{B}}}MP_{\ker{C}^\perp})_rC^\dagger C=(P_{\overline{\ran{B}}}MP_{\ker{C}^\perp})_r$. By \Cref{prop:ker_ran_relation},
	\begin{align*}
		\ran{(P_{\overline{\ran{C^*}}}M^*P_{{\ker{B^*}^\perp}})_r} = \ran{(P_{\overline{\ran{B}}}MP_{\ker{C}^\perp})_r^*}  = \ran{(B\hat{X}C)^*} =\ran{C^*\hat{X}^*B^*}\subset \ran{C^*},
	\end{align*}
	We conclude \cref{item:adjoint} holds.
\end{proof}

We give a final sufficient condition, namely a lower bound condition, that is weaker than $C$ having closed range. Since $C$ is bounded, $\ker{C}$ is closed. We recall that the quotient space $\mathcal{H}_1/\ker{C}$ then is a Banach space with norm $\norm{h+\ker{C}}=\inf\{\norm{h+c}:\ c\in\ker{C}\}$, and that the quotient map $Q\in\B(\mathcal{H}_1,\mathcal{H}_1/\ker{C})$, $h\mapsto h+\ker{C}$ is continuous. See for example \cite[Section III.4]{conway_course_2007}.
\begin{corollary}
	\label{cor:lower_bound_condition}
	Suppose that there exists $c>0$ such that $\norm{Ch}\geq c\norm{h+\ker{C}}$ for all $h\in\ker{Z}\cap\ker{C}^\perp$. Then $\hat{X}$ is bounded.
\end{corollary}
\begin{proof}
	Let us denote $W\coloneqq \ker{Z}\cap\ker{C}^\perp$. We show that $C(W)$ is closed in $\mathcal{H}_2$. Then $C(W)$ is also closed in $\ran{C}$ and we may apply \Cref{prop:bounded_solutions}. Now, we can decompose $C=\tilde{C}\circ Q$ for some $\tilde{C}\in\B(\mathcal{H}_1/\ker{C},\mathcal{H}_2)$. Let $(w_n)_n$ be a sequence in $W$ such that $(Cw_n)_n=(\tilde{C}(w_n+\ker{C}))_n$ converges to $k\in\mathcal{H}_4$. We need to show $k=Cw$ for some $w\in W$. By assumption, $\norm{w_n-w_m+\ker{C}}\leq c^{-1}\norm{C(w_n-w_m)}$, so $(w_n+\ker{C})_n$ is a Cauchy sequence in $\mathcal{H}_1/\ker{C}$. 
	But $\mathcal{H}_1/\ker{C}$ is complete, so $w_n+\ker{C}$ converges to some $h+\ker{C}\in\mathcal{H}_1/\ker{C}$. As $W$ and $\ker{C}$ are orthogonal, their direct sum $W+\ker{C}$ is closed, c.f.\ \cite[Section 9.2.(A)]{ben-israel_generalized_2003}. Hence $h+\ker{C}\in W+\ker{C}$, i.e. $h+\ker{C}=w+\ker{C}$ for some $w\in W$. Now, $Cw=Ch=\tilde{C}(h+\ker{C})=\lim_n \tilde{C}(w_n+\ker{C})=k$, by continuity of $\tilde{C}$.
\end{proof}

In case $C$ has closed range, the operator $\tilde{C}$ in the proof of \Cref{cor:lower_bound_condition} also has closed range. But $\tilde{C}$ is also injective, so it has a bounded inverse on its range, by the inverse mapping theorem, c.f.\ \cite[Theorem III.12.5]{conway_course_2007}. Hence 
\begin{align}
	\label{eqn:lower_bound_for_closed_range_operator}
	\norm{h+\ker{C}}=\norm{\tilde{C}^{-1}\tilde{C}h}\leq \norm{\tilde{C}^{-1}}\norm{Ch},\quad\text{for all } h\in\mathcal{H}_1.
\end{align}
In particular, we conclude that the lower bound condition of \Cref{cor:lower_bound_condition} is satisfied with $c\coloneqq \norm{\tilde{C}^{-1}}^{-1}$ in the case where $C$ has closed range. 
The lower bound condition in \Cref{cor:lower_bound_condition} need not hold everywhere on $\mathcal{H}_1$, so it is weaker than \eqref{eqn:lower_bound_for_closed_range_operator}. Hence, the condition in \Cref{cor:lower_bound_condition} is strictly weaker than the condition that $\ran{C}$ is closed.

\subsection{Bounded approximations}
\label{subsec:bounded:approximations}
Even when the equivalent conditions of \Cref{prop:bounded_solutions} do not hold, we can enforce bounded solutions by instead considering the problem
\begin{align}
	\inf\{\norm{M-BXC}_{L_2}:\ X\in\B_{00,r}(\mathcal{H}_2,\mathcal{H}_3)\},
	\label{eqn:bounded_problem}
\end{align}
with the minimality property
\begin{align}
	P_{\ker{B}^\perp}XP_{\overline{\ran{C}}}=X.
	\label{eqn:bounded_minimality}
\end{align}
As $\mathcal{B}_{00,r}(\mathcal{H}_2,\mathcal{H}_3)\subset \mathcal{X}_r$, this infimum is bounded from below by the minimum in \eqref{eqn:reduced_rank_approximation_problem}. In the case that $\ran{C}$ is closed, the problems \eqref{eqn:reduced_rank_approximation_problem}-\eqref{eqn:minimality_constraint} and \eqref{eqn:bounded_problem}-\eqref{eqn:bounded_minimality} coincide by \Cref{rmk:closed_range_case}\ref{item:C_closed_range}. Let us therefore assume that $\ran{C}$ is not closed, so $\ker{C}^\perp$ is infinite-dimensional by \Cref{lemma:ker_perp_inf_dim}. We now show that in this case the infimum \eqref{eqn:bounded_problem} is not attained in general, but is still equal to the minimum in \eqref{eqn:reduced_rank_approximation_problem}. To do so, we construct a minimizing sequence $\hat{X}_n\in\mathcal{B}_{00,r}(\mathcal{H}_2,\mathcal{H}_3)$ to \eqref{eqn:bounded_problem}-\eqref{eqn:bounded_minimality} such that $\hat{X}_n\rightarrow\hat{X}$ pointwise on $\dom{\hat{X}}$. We do this under the assumption that $\mathcal{Y}_{r}\not=\emptyset$ and hence that a solution $\hat{X}$ of the form \eqref{eqn:hat_X} exists.

Given $\hat{X}$ of the form \eqref{eqn:hat_X}, we construct the approximating sequence $\hat{X}_n\in\mathcal{B}_{00,r}(\mathcal{H}_2,\mathcal{H}_3)$ based on the approximation of the unbounded Moore--Penrose inverse $C^\dagger$ by outer inverses as proposed by \cite{huang_approximation_2006}. We briefly repeat the construction and its properties here, but refer to \cite{huang_approximation_2006} for elaboration, see also \cite[Theorem 3.24]{engl_regularization_1996}.

Let $(Y_n)_n$ be an increasing sequence of finite-dimensional subspaces of $\overline{\ran{C}}$ such that $\overline{\cup_n Y_n}=\overline{\ran{C}}$. We define $X_n\coloneqq C^* Y_n$ and let $P_n$, $Q_n$ be the orthogonal projectors onto $Y_n$ and $X_n$ respectively. We then set $C_n\coloneqq P_nC$, so $C_n$ has finite rank. It can be shown, c.f.\ \cite[p114]{huang_approximation_2006}, that $(X_n)_n$ is an increasing sequence of finite-dimensional subspaces, that $\overline{\cup_n X_n}=\ker{C}^\perp$ and that $\ker{C_n}^\perp=X_n$, $\ran{C_n}=Y_n$ so that $\restr{C_n}{X_n}:X_n\rightarrow Y_n$ is a boundedly invertible operator. Hence, we can define the finite-rank operator $C_n^\#:\mathcal{H}_2\rightarrow\mathcal{H}_1$ as
\begin{align*}
	C_n^\#(y)\coloneqq
	\begin{cases}
		(\restr{C_n}{X_n})^{-1}(y),\quad y\in Y_n, \\
		0, \quad y\in Y_n^\perp.
	\end{cases}
\end{align*}
The following result states the approximation property of $C^\dagger$ by $C_n^\#$, see \cite[Theorem 2.1]{huang_approximation_2006} for a proof.
\begin{theorem}
	\label{thm:outer_inverse_approximation}
	The sequence $(C_n)_n\subset\B_{00}(\mathcal{H}_2,\mathcal{H}_1)$ satisfies the following:
	\begin{enumerate}
		\item
			$\dom{C^\dagger} = \{h\in\mathcal{H}_2:\lim_{n}C_n^{\#}h\text{ exists}\},$
		\item 
			\label{item:outer_inverse_approximation_2}
			$C^\dagger h = \lim_n C_n^\#h$ for $h\in\dom{C^\dagger}$,
		\item 
			$Q_nC^\dagger h = C_n^\#h$ for $h\in\dom{C^\dagger}$,
		\item
			$C_n^\#CC_n^\#=C_n^\#$, i.e. $C_n^\#$ is an outer inverse of $C.$
	\end{enumerate}
\end{theorem}
Under the condition $\mathcal{Y}_r\not=\emptyset$, let us now set 
\begin{align}
	\hat{X}_n\coloneqq B^\dagger(P_{\overline{\ran{B}}}MP_{\ker{C}^\perp})_rC_n^\#.
	\label{eqn:bounded_approximation}
\end{align}
The next result directly implies that the infimum in \eqref{eqn:bounded_problem} and minimum in \eqref{eqn:reduced_rank_approximation_problem} indeed agree.
\begin{proposition}
	\label{prop:solution_to_bounded_problem}
	In \Cref{setting:main} suppose $\mathcal{Y}_{r}\not=\emptyset$ and let $\hat{X}$ be given by \eqref{eqn:hat_X}. With $\hat{X}_n$ defined in \eqref{eqn:bounded_approximation}, it holds that $\hat{X}_n\in\B_{00,r}(\mathcal{H}_2,\mathcal{H}_3)$, $\hat{X}_n\rightarrow\hat{X}$ pointwise on $\dom{\hat{X}}$ and $B\hat{X}_nC\rightarrow B\hat{X}C$ in $L_2(\mathcal{H}_1,\mathcal{H}_4)$. Furthermore, $\hat{X}_n$ satisfies \eqref{eqn:bounded_minimality}. It follows that the infimum in \eqref{eqn:bounded_problem} and minimum in \eqref{eqn:reduced_rank_approximation_problem} agree.
\end{proposition}
In \Cref{sec:unbounded_solution} it is shown that in some cases, all solutions $\hat{X}$ of \eqref{eqn:reduced_rank_approximation_problem} can be unbounded. Since the infimum in \eqref{eqn:bounded_problem} and the minimum in \eqref{eqn:reduced_rank_approximation_problem} are equal by \Cref{prop:solution_to_bounded_problem}, it then follows that the infimum in \eqref{eqn:bounded_problem} is not attained in this case. Indeed, if it were attained by some bounded $\tilde{X}\in\B_{00,r}(\mathcal{H}_2,\mathcal{H}_3)\subset \mathcal{X}_r$, then $\tilde{X}$ would be a bounded minimiser of \eqref{eqn:reduced_rank_approximation_problem}, contradicting the fact that in this case all solutions are unbounded.
\begin{proof}
	Let us define $Y\coloneqq (P_{\overline{\ran{B}}}MP_{\ker{C}^\perp})$. As in the proof of \Cref{lemma:solution_characterisation} \cref{item:surjectivity}, we can use the first equation of \eqref{eqn:truncated_svd_identity} and the line preceding it to find that
	\begin{align*}
		B\hat{X}C = B(B^\dagger YC^\dagger) C = YC^\dagger C = YP_{\ker{C}^\perp} = (P_{\overline{\ran{B}}}MP_{\ker{C}^\perp})_rP_{\ker{C}^\perp},
	\end{align*}
and we can use \Cref{lemma:unbounded_after_finite_rank} to find that $B^\dagger (P_{\overline{\ran{B}}}MP_{\ker{C}^\perp})_r$ is bounded.
	Now, being a composition of bounded operators, $\hat{X}_n$ is bounded. Also, $\hat{X}_n\rightarrow \hat{X}$ pointwise on $\dom{\hat{X}}$ since $C_n^\#\rightarrow C^\dagger$ pointwise on $\dom{C^\dagger}=\dom{\hat{X}}$ by \Cref{thm:outer_inverse_approximation}\ref{item:outer_inverse_approximation_2}. 

	Now, $Q_nP_{\ker{C}^\perp}=Q_n$ since $\ker{Q_n}^\perp=X_n\subset\overline{\ran{C}}$. By \Cref{thm:outer_inverse_approximation}, $C_n^\#=Q_nC^\dagger$ on $\dom{C^\dagger}$. Combining this with \eqref{eqn:third_moore_penrose} yields $C_n^\#C = Q_nC^\dagger C = Q_n$. Using also \eqref{eqn:first_moore_penrose}, \eqref{eqn:third_moore_penrose}, and \eqref{eqn:fourth_moore_penrose}, it follows that
	\begin{align}
		\label{eqn:composition_X_hat_n}
		B\hat{X}_n C &= BB^\dagger(P_{\overline{\ran{B}}}MP_{\ker{C}^\perp})_rC_n^\#C = (P_{\overline{\ran{B}}}MP_{\ker{C}^\perp})_rQ_n, \\
		\label{eqn:minmality_X_hat_n}
		P_{\ker{B}^\perp}\hat{X}_nP_{\overline{\ran{C}}}&= B^\dagger B\hat{X}_nCC^\dagger = B^\dagger(P_{\overline{\ran{B}}}MP_{\ker{C}^\perp})_rQ_nC^\dagger = \hat{X}_n\quad\text{on }\dom{C^\dagger}.
	\end{align}
	Equation \eqref{eqn:minmality_X_hat_n} shows that \eqref{eqn:bounded_minimality} holds, as $\dom{C^\dagger}\subset \mathcal{H}_2$ densely and $\hat{X}_n$ is continuous. 
	As the $(X_n)_n$ are increasing, we can take an ONB $(e_i)_i$ of $\ker{C}^\perp$ such that for all $n$, $\{e_1,\ldots,e_n\}$ is an ONB of $X_n$ and $e_i\in X_n^\perp$, $i>n$. 
Thus, we compute with \eqref{eqn:composition_X_hat_n},
	\begin{align*}
		\norm{B\hat{X}C - B\hat{X}_nC}_{L_2}^2 &= \norm{(P_{\overline{\ran{B}}}MP_{\ker{C}^\perp})_r(P_{\ker{C}^\perp}-Q_n)}_{L_2}^2 \\
		&= \sum_{i>n}^{}\norm{(P_{\overline{\ran{B}}}MP_{\ker{C}^\perp})_re_i}^2 \rightarrow 0,
	\end{align*}
	as the tail of a sum that converges because $(P_{\overline{\ran{B}}}MP_{\ker{C}^\perp})_r$ is finite-rank and thus Hilbert--Schmidt. 
\end{proof}

One could also start with an increasing sequence $(X_n)_n$ and define $Y_n=CX_n$ instead of choosing $(Y_n)_n$ and setting $X_n=C^*Y_n$. In this case the orthogonal projector $P_n$ onto $Y_n$ satisfies $\ran{P_n}=C(\ran{Q}_n)$ and we could try to construct bounded approximations $\tilde{C}_n\coloneqq C^\dagger P_n$ that converge pointwise to $C^\dagger$. If $\overline{\cup_n X_n}=\ker{C}^\perp$, then still $\overline{\cup_n Y_n}=\overline{\ran{C}}$ and the pointwise convergence $\tilde{C}_n\rightarrow C^\dagger$ holds on $\dom{C^\dagger}$. But $C^\dagger P_nCe_i$ is not equal to 0 for $i>n$, so $\tilde{C}_nC = C^\dagger P_n C$ is no longer equal to $Q_n$, which is what we used to show the $L_2(\mathcal{H}_1,\mathcal{H}_4)$ convergence statement in \Cref{prop:solution_to_bounded_problem}. In fact, the sequence,
\begin{align*}
	\sum_{i>n}^{}\norm{(P_{\overline{\ran{B}}}MP_{\ker{C}^\perp})_r(I-C^\dagger P_nC)e_i}^2
\end{align*}
need not converge as $n\to\infty$.
Thus the $L_2(\mathcal{H}_1,\mathcal{H}_4)$ convergence statement in \Cref{prop:solution_to_bounded_problem} will not hold in general for this alternative approximation procedure.

\begin{remark}[Bounded approximations of unbounded approximate minimisers]
	\label{rmk:bounded_approximations}
	The approximate minimisers $(\hat{X}^{(m)})_m$ of some SVD $Y$ of $P_{\overline{\ran{B}}}MP_{\ker{C}^\perp}$ constructed in \Cref{subsec:approximation} may be unbounded. The procedure outlined before \Cref{prop:solution_to_bounded_problem} can be used for each $m$ to approximate $\hat{X}^{(m)}$ with bounded approximations $(\hat{X}^{(m)}_n)_n$. Then $B\hat{X}^{(m)}_nC\rightarrow B\hat{X}^{(m)}C$ in $L_2(\mathcal{H}_1,\mathcal{H}_4)$ as $n\rightarrow \infty.$ In particular, for each $m\in\N$ there exists $n_m\in\N$ such that $\norm{B\hat{X}^{(m)}_{n}C - B\hat{X}^{(m)}C}_{L_2}<m^{-1}$ for all $n\geq n_m$. Hence $B\hat{X}^{(m)}_{n_m}C\rightarrow Y$ in $L_2(\mathcal{H}_1,\mathcal{H}_4)$. Thus $(\hat{X}^{(m)}_{n_m})_m$ is a sequence of bounded approximate minimisers.  Since $B(P_{\ker{B}^\perp}\hat{X}_{n_m}^{(m)}P_{\overline{\ran{C}}})C = B\hat{X}_{n_m}^{(m)}C$ by \eqref{eqn:first_moore_penrose}, \eqref{eqn:third_moore_penrose} and \eqref{eqn:fourth_moore_penrose}, and since $P_{\ker{B}^\perp}\hat{X}_{n_m}^{(m)}P_{\overline{\ran{C}}}$ satisfies \eqref{eqn:bounded_minimality}, we can thus construct a sequence of bounded minimisers which also satisfy the minimality property \eqref{eqn:bounded_minimality}.
\end{remark}

\subsection{Adjoint formulation}
	\label{rmk:adjoint_formulation}
	\Cref{prop:solution_to_bounded_problem} also allows us to pass to an adjoint formulation of \eqref{eqn:reduced_rank_approximation_problem}.
	Let us write $\mathcal{X}_r(B,C)$ and $\mathcal{Y}_r(M,B,C)$ for $\mathcal{X}_r$ and $\mathcal{Y}_r$ defined in \Cref{setting:main} to stress the dependence on $M$, $B$, and $C$. 
	The adjoint problem then is
	\begin{align}
		\label{eqn:adjoint_formulation}
		\min\{\norm{M^*-C^*XB^*}_{L_2}:\ X\in\mathcal{X}_r(C^*,B^*)\},
	\end{align}
	with the minimality property
	\begin{align}
		\label{eqn:adjoint_constraint}
		P_{\overline{\ran{C}}}XP_{\ker{B}^\perp}=X\text{ on }\dom{X}.
	\end{align}
	Note that by \Cref{prop:ker_ran_relation}, the adjoint formulation of the adjoint problem \eqref{eqn:adjoint_formulation}-\eqref{eqn:adjoint_constraint} is the direct problem \eqref{eqn:reduced_rank_approximation_problem}-\eqref{eqn:minimality_constraint}.
	We now discuss why one might want to solve \eqref{eqn:adjoint_formulation} instead of \eqref{eqn:reduced_rank_approximation_problem}. The same reasoning holds for solving \eqref{eqn:adjoint_formulation}-\eqref{eqn:adjoint_constraint} instead of \eqref{eqn:reduced_rank_approximation_problem}-\eqref{eqn:minimality_constraint}. Note that under the assumption $\mathcal{Y}_r(M,B,C)\not=\emptyset$ and $\mathcal{Y}_r(M^*,C^*,B^*)\not=\emptyset$, we have by \Cref{prop:solution_to_bounded_problem}, 
	\begin{align}
		\label{eqn:direct_and_adjoint_problem_relation}
		\begin{split}
			\min\{\norm{M-BXC}_{L_2}^2:\ X\in\mathcal{X}_r(B,C)\} &= \inf\{\norm{M-BXC}_{L_2}^2:\ X\in\mathcal{B}_{00,r}(\mathcal{H}_2,\mathcal{H}_3)\} \\
			&= \inf\{\norm{M^*-C^*X^*B^*}_{L_2}^2:\ X\in\mathcal{B}_{00,r}(\mathcal{H}_2,\mathcal{H}_3)\} \\
			&= \inf\{\norm{M^*-C^*XB^*}_{L_2}^2:\ X\in\mathcal{B}_{00,r}(\mathcal{H}_3,\mathcal{H}_2)\} \\
			&= \min\{\norm{M^*-C^*XB^*}_{L_2}^2:\ X\in\mathcal{X}_r(C^*,B^*)\}.  
		\end{split}
	\end{align}
	Here we used that $(BXC)^*=C^*X^*B^*$ holds true when $B,X,C$ are bounded.  When $\mathcal{Y}_r(M,B,C)=\emptyset$, the minimum of the left hand side becomes an infimum that is not attained. It is, however, still possible that $\mathcal{Y}_r(M^*,C^*,B^*)\not=\emptyset$. In this case, the adjoint problem \eqref{eqn:adjoint_formulation} still has a solution and allows us to recover the minimal Hilbert--Schmidt approximation error, even though \eqref{eqn:reduced_rank_approximation_problem} does not admit a minimal solution. 
	An an example, with $M$, $B$ and $C$ self-adjoint and $B$ and $C$ injective, the conditions $\mathcal{Y}_r(M,B,C)\not=\emptyset$ and $\mathcal{Y}_r(M^*,C^*,B^*)\not=\emptyset$ are equivalent to the conditions $\ran{(M)_r}\subset \dom{B^\dagger}$ and $\ran{(M)_r}\subset\dom{C^\dagger}$ respectively, which need not hold simultaneously. 
	
	If both the direct and adjoint problems admit corresponding solutions in the sense of the following result, then both solutions are bounded. This follows from \Cref{prop:bounded_solutions}, since \cref{item:adjoint} relates to the existence condition $\mathcal{Y}_r(M^*,C^*,B^*)\not=\emptyset$ for the adjoint problem: $\mathcal{Y}_r(M^*,C^*,B^*)\not=\emptyset$ if and only if there exists a truncated rank-$r$ SVD of $P_{\overline{\ran{C^*}}}M^*P_{\ker{B^*}^\perp}$ satisfying \cref{item:adjoint}. More precisely, \Cref{prop:adjoint}\ref{item:direct_solution} can be interpreted as follows. Suppose that the adjoint solution $\hat{X}^a$ of \eqref{eqn:adjoint_formulation}-\eqref{eqn:adjoint_constraint} corresponding to some truncated rank-$r$ SVD $Y^*$ of $P_{\overline{\ran{C^*}}}M^*P_{\ker{B^*}^\perp}$ exists, i.e. $Y^*\in\mathcal{Y}_r(M^*,C^*,B^*)$ and $\hat{X}^a=(C^*)^\dagger Y^*(B^*)^\dagger$. Then \Cref{prop:adjoint}\ref{item:direct_solution} shows that $\hat{X}^a$ is bounded if and only if $Y\in\mathcal{Y}_r(M,B,C)$, i.e.\ if and only if $B^\dagger YC^\dagger$ is well-defined and solves the direct problem \eqref{eqn:reduced_rank_approximation_problem}-\eqref{eqn:minimality_constraint}. Furthermore, since the adjoint formulation of the adjoint problem is the direct problem, we may replace $\hat{X}^a$, $B$, $C$, and $M$ by, respectively, $\hat{X}^d$, $C^*$, $B^*$, and $M^*$. That is, if $Y\in\mathcal{Y}_r(M,B,C)$ so that $\hat{X}^d\coloneqq B^\dagger Y C^\dagger$ is a well-defined solution of \eqref{eqn:reduced_rank_approximation_problem}-\eqref{eqn:minimality_constraint}, then $\hat{X}^d$ is bounded if and only if $Y^*\in\mathcal{Y}_r(M^*,C^*,B^*)$, i.e.\ if and only if $(C^*)^\dagger Y^*(B^*)^\dagger$ is a well-defined solution of \eqref{eqn:adjoint_formulation}-\eqref{eqn:adjoint_constraint}.
		\begin{proposition}
			\label{prop:adjoint}
			Suppose that $Y\in\B(\mathcal{H}_2,\mathcal{H}_3)$ satisfies $Y^*\in\mathcal{Y}_r(M^*,C^*,B^*)$. Then $\hat{X}^a\coloneqq (C^*)^\dagger Y^*(B^*)^\dagger$ is a solution of the adjoint problem \eqref{eqn:adjoint_formulation}-\eqref{eqn:adjoint_constraint}. Furthermore,
			\begin{enumerate}
				\item 
					$\hat{X}^a$ is bounded if and only if $Y\in\mathcal{Y}_r(M,B,C)$,
					\label{item:direct_solution}
				\item
					if $Y\in \mathcal{Y}_r(M,B,C)$, then $\hat{X}^d\coloneqq B^\dagger Y C^\dagger$ defined in \eqref{eqn:hat_X} is a bounded solution of the direct problem \eqref{eqn:reduced_rank_approximation_problem}-\eqref{eqn:minimality_constraint} and $\hat{X}^d=(\hat{X}^a)^*$.
					\label{item:direct_boundedness}
			\end{enumerate}
		\end{proposition}
		\begin{proof}
			By \Cref{thm:main} applied to the adjoint problem, $\hat{X}^a$ solves \eqref{eqn:adjoint_formulation}-\eqref{eqn:adjoint_constraint}. 
			By definition, $Y\in\mathcal{Y}_r(M,B,C)$ is equivalent to $\ran{(P_{\overline{\ran{B}}}MP_{{\ran{C}^\perp}})_r}\subset\dom{B^\dagger}$, so that \Cref{prop:bounded_solutions}\ref{item:bounded}-\ref{item:adjoint} applied to the adjoint problem shows that $\hat{X}^a$ is bounded if and only if $Y\in\mathcal{Y}_r(M,B,C)$. This proves \cref{item:direct_solution}.

			Next, let us assume $Y\in \mathcal{Y}_r(M,B,C)$. By \Cref{thm:main} applied to the direct problem, $\hat{X}^d$ solves \eqref{eqn:reduced_rank_approximation_problem}-\eqref{eqn:minimality_constraint}. 
			The assumption $Y^*\in\mathcal{Y}_r(M^*,C^*,B^*)$ is by definition equivalent to $\ran{(P_{\overline{\ran{C^*}}}M^*P_{{\ker{B^*}^\perp}})_r}\subset\dom{(C^*)^\dagger}$. Thus, by \Cref{prop:bounded_solutions}\ref{item:bounded}-\ref{item:adjoint} applied to the direct problem, $\hat{X}^d$ is bounded.
			It remains to show that $\hat{X}^a=(\hat{X}^d)^*$. Now, $C^*\hat{X}^aB^*=Y^*$ and $B\hat{X}^dC = Y$ by \Cref{lemma:solution_characterisation}\ref{item:surjectivity} applied to the adjoint and direct problems respectively. Then,
			\begin{align*}
				B(\hat{X}^a)^*C = (C^*\hat{X}^aB^*)^* = Y^{**} = Y = B\hat{X}^dC.
			\end{align*}
			Hence it holds on $\dom{C}^\dagger$,
			\begin{align*}
				(\hat{X}^a)^* 
				&=  (P_{\overline{\ran{C}}}\hat{X}^aP_{\ker{B}^\perp})^* 
				= P_{\ker{B}^\perp}(\hat{X}^a)^*P_{\overline{\ran{C}}} \\
				&=  B^\dagger (B(\hat{X}^a)^*C)C^\dagger 
				=  B^\dagger (B\hat{X}^dC)C^\dagger 
				= P_{\ker{B}^\perp}\hat{X}^dP_{\overline{\ran{C}}}
				= \hat{X}^d,
			\end{align*}
			where in the first step we use \eqref{eqn:adjoint_constraint}, in the third and fifth step we use \eqref{eqn:third_moore_penrose} and \eqref{eqn:fourth_moore_penrose}, and in the last step we use \eqref{eqn:minimality_constraint}. Since $\hat{X}^a$ and $\hat{X}^d$ are bounded and $\dom{C^\dagger}\subset\mathcal{H}_2$ densely, it follows that $(\hat{X}^a)^*=\hat{X}^d$ on all of $\mathcal{H}_2$.
		\end{proof}
		By \Cref{prop:adjoint}\ref{item:direct_solution}, an adjoint solution $(C^*)^\dagger Y^*(B^*)^\dagger$ for $Y\in\B(\H_2,\H_3)$, $Y^*\in\mathcal{Y}_r(M^*,C^*,B^*)$, is bounded if and only if $B^\dagger Y C^\dagger$ is a direct solution. In this case, $B^\dagger Y C^\dagger$ is bounded and given by the adjoint of $(C^*)^\dagger Y^* (B^*)^\dagger$, by \Cref{prop:adjoint}\ref{item:direct_boundedness}. Thus, bounded solutions of the adjoint problem \eqref{eqn:adjoint_formulation}-\eqref{eqn:adjoint_constraint} allow one to obtain bounded solutions of the direct problem \eqref{eqn:reduced_rank_approximation_problem}-\eqref{eqn:minimality_constraint} by taking adjoints, and vice versa. In the case that the direct problem has both bounded and unbounded solutions but the adjoint solution has only bounded solutions, this can be used to characterise all bounded solutions of the direct problem as adjoints of the adjoint solutions. An example is given \Cref{sec:unbounded_solution}, where the direct problem has both bounded and unbounded solutions, but the adjoint problem only has bounded solutions by \Cref{rmk:closed_range_case}\ref{item:C_closed_range} with $C$ replaced by $B^*=I$.

\section{Application in Signal Processing, Reduced Rank Regression and Linear Operator Learning}
\label{sec:signal_processing_reduced_rank_regression}
In this section we describe how \Cref{thm:main} can be leveraged to obtain results in infinite-dimensional settings. As example applications, we consider the fields of signal processing, reduced rank regression, and linear operator learning.
We first recall some examples of approximation problems that arise in these applications, before describing a problem that includes these examples as special cases.

In the field of signal processing, one common problem is to reconstruct a random input signal $x\in\R^m$ from a random observation signal $y\in\R^m$ that is related to $x$. One can assume that each observation satisfies $y=Gx+\varepsilon$ for some known linear map $G$ and additive noise $\varepsilon$, or weaker assumptions can be made on the relationship between $x$ and $y$. In \cite{torokhti2009towards}, one considers the problem of reconstructing the signal $x$ as $Ay$, assuming additional structure only in the form of a known covariance matrix of $(x,y)$. To do so, the reconstruction quality is assessed in the expected squared reconstruction error, and thus the objective is to minimise $\mathbb{E}\norm{x-Ay}^2$ over $A$. As the observation signal can be noisy, the rank of $A$ is restricted in order to filter out this noise. For more generality and motivated by applications, one can also allow $x$, $y$, and the reconstruction $Ay$ to be weighted by matrices $W_x,W_y$ and $W_A$. The problem then becomes to find the rank-constrained minimiser $A$ of $\mathbb{E}\norm{W_xx-W_A AW_yy}^2$. 
In \cite{torokhti2009filtering}, one considers an extension of this problem in which not just one signal but sets of signals $\{x(\alpha)\in\R^m:\ \alpha\in K\}$ and $\{y(\alpha)\in\R^m:\ \alpha\in K\}$ are considered, where $K$ is a suitable index set that may be countably or uncountably infinite, e.g. $K=\N$ or $K=[0,1]$, and $x(\alpha)$ and $y(\alpha)$ are finite-dimensional random vectors in $\R^m$ for every $\alpha$.
Below, we consider the settings in \cite{torokhti2009towards,torokhti2009filtering} in the case of infinite-dimensional signals, that is, in case the signals $x(\alpha)$ and $y(\alpha)$ are function-valued for every $\alpha\in K$. 

In the setting of reduced rank regression of \cite{izenman_reduced-rank_1975}, the task is to find a mean vector $\mu$ and rank-restricted matrix $A$ that minimize the mean square error $\mathbb{E}\norm{\Gamma(x-\mu-Ay)}^2$ for random vectors $x$ and $y$ and fixed weighting matrix $\Gamma$. Assuming $x$ and $y$ have zero mean so that $\mu=0$ and writing $W_x=\Gamma=W_A$ and $W_y=I$, we recover the same formulation as in the signal processing setting. This also relates to the regression setting of \cite{turri_randomized_2023}, where $x$ and $y$ are allowed to be infinite-dimensional and are transformed through kernel functions first. Furthermore, in \cite{kacham2021reduced}, one studies the reduced rank regression problem of minimizing $\norm{BX-M}$ where $M$ and $B$ are given matrices, $\norm{\cdot}$ denotes the operator norm, and $X$ is restricted in rank. 
Below, we study the case that $x$ and $y$ in \cite{izenman_reduced-rank_1975} may be infinite-dimensional and that the operator norm in \cite{kacham2021reduced} is replaced with the Hilbert--Schmidt norm. 

A closely related problem that arises in the setting of linear operator learning of \cite{mollenhauer_learning_2023} is to minimise $\mathbb{E}\norm{x-Ay}^2$, where $x$ and $y$ are allowed to take values in a separable Hilbert space, and where the search space for $A$ is taken to be either the space of bounded operators or Hilbert--Schmidt operators between these Hilbert spaces. This has applications in various kinds of regression problems, such as functional regression with functional response; see \cite{mollenhauer_learning_2023}. In this section, we study the operator learning problem in the case that $A$ is bounded and the rank of $A$ is restricted and finite. 

Let $(\Omega,\mathcal{F},\mathbb{P})$ be a countably generated probability space, so that $L^2(\Omega)$ is separable, c.f.\ \cite[Proposition 3.4.5]{cohn_measure_2013}. Let $F,G$ be real separable Hilbert spaces. We let $x\in L^2(\Omega;F)$ and $y\in L^2(\Omega;G)$ be random variables. Here, $L^2(\Omega;F)$ and $L^2(\Omega;G)$ denote Lebesgue-Bochner spaces, so integrability is in the sense of square Bochner integrability, see \cite[Section 1.2]{hytonen_analysis_2016}. These spaces are then separable. This follows for example from \cite[Proposition 1.2.29]{hytonen_analysis_2016} or because $L^2(\Omega;F)\simeq L^2(\Omega)\otimes F$, which is separable as the tensor product of separable Hilbert spaces, c.f \cite[Assumption 2.1.(g)]{klebanov_rigorous_2020}. 
\begin{problem}
	Given weighting operators $W_x\in\B(F,\mathcal{H}_1),W_A\in\B(\mathcal{H}_2,\mathcal{H}_1)$ and $W_y\in\B(G,\mathcal{H}_3)$ for real separable Hilbert spaces $\mathcal{H}_1,\mathcal{H}_2,\mathcal{H}_3$, solve
	\begin{align}
		\min\{\mathbb{E}\norm{W_xx-W_AAW_yy}^2:\ A\in\B_{00,r}(\mathcal{H}_3,\mathcal{H}_2)\}.
		\label{eqn:signal_processing_problem}
	\end{align}
	\label{problem:signal_processing}
\end{problem}
Note that in infinite dimensions it is not a priori clear that the minimum in \eqref{eqn:reduced_rank_approximation_problem} is indeed attained for a bounded operator, c.f.\ \Cref{sec:unbounded_solution}, so \Cref{problem:signal_processing} might not have a solution.
\newline
\newline
If ${F}=G=\R^m$, then we are back in the setting of \cite{torokhti2009towards}. If ${F}=G=L^2(K;\R^m)$ with $(K,\mathcal{K},\lambda)$ a countably generated measure space with $\sigma$-finite measure $\lambda$, then $F$ and $G$ are separable, and we recover the setting of \cite{torokhti2009filtering} by interpreting the infinite set of signals $\{x(\alpha):\ \alpha\in K\}$ as one signal $x\in L^2(\Omega\times K;\R^m)\simeq  L^2(\Omega;L^2(K,\R^m))$. If $W_x,W_y$ and $W_A$ are taken to be the identity, we are in the linear operator learning setting of \cite{mollenhauer_learning_2023} but with an additional rank restriction imposed on $A$. If $W_x=\Gamma=W_A$, $W_y=I$ and $x$ and $y$ are finite-dimensional and have mean 0, then \eqref{eqn:signal_processing_problem} reduces to the reduced rank setting of \cite{izenman_reduced-rank_1975}.

For $z\in L^2(\Omega;E_1),w\in L^2(\Omega;E_2)$ we define the `uncentered covariance' $C_{wz}\in\B(E_1,E_2)$ by 
\begin{align*}
	\langle C_{wz}e,f\rangle \coloneqq \int_{\Omega}^{}\langle w(\omega),f\rangle\langle e,z(\omega)\rangle \d\P(\omega),\quad e\in E_1, f\in E_2.
\end{align*}
We also define $C_z\coloneqq C_{zz}$. Hence with $x$ and $y$ as in \Cref{problem:signal_processing}, $C_x$ and $C_y$ are positive and self-adjoint, while $C_{xy}=C_{yx}^*$. We collect some further facts on covariances in \Cref{lemma:covariance}, which is proven in \Cref{sec:proofs}.
\begin{restatable}{lemma}{covlemma}
	\label{lemma:covariance}
	Let $z\in L^2(\Omega;E_1)$ and $w\in L^2(\Omega;E_2)$. Then the following hold:
	\begin{enumerate}
		\item
			If $E_1=E_2$, then $\tr{C_{wz}}=\mathbb{E}\langle w,z\rangle<\infty$.
			\label{item:covariance_trace}
		\item
			$C_{Qw,z}=QC_{w,z}$ for any separable Hilbert space $E_3$ and $Q\in\B(E_2,E_3)$.
			\label{item:covariance_of_linear_transform}
		\item
			$C_{wz}=C_w^{1/2}UC_z^{1/2}$ for some $U\in\B(E_1,E_2)$ satisfying $\norm{U}\leq 1$ and $U = P_{\overline{\ran{C_w^{1/2}}}}UP_{\overline{\ran{C_z^{1/2}}}}$.
			\label{item:cross_covariance_representation}
	\end{enumerate}
\end{restatable}
Since $x,y$ in \Cref{problem:signal_processing} are square integrable, we see from \Cref{lemma:covariance} \cref{item:covariance_trace} that $C_x$, $C_y$ and $C_{xy}$ are all trace-class. Thus $C_x^{1/2},C_y^{1/2}$ are Hilbert--Schmidt. We can now solve \Cref{problem:signal_processing}, assuming that $\ran{W_A}$ is closed. We note that in order to prove this result, we apply the adjoint formulation approach outlined in \Cref{rmk:adjoint_formulation}.
\begin{proposition}
	Suppose that $W_A$ has closed range. 
	Then a solution to \Cref{problem:signal_processing} exists if and only if there exists a truncated SVD $(P_{\overline{\ran{C_y^{1/2}W_y^*}}}(C_y^{1/2})^\dagger C_{yx}W_x^*P_{\ran{W_A}})_r$ satisfying
	\begin{align}
		\ran{(P_{\overline{\ran{C_y^{1/2}W_y^*}}}(C_y^{1/2})^\dagger C_{yx}W_x^*P_{\ran{W_A}})_r}\subset\ran{C_y^{1/2}W_y^*}.
		\label{eqn:existence_assumption_signal_processing}
	\end{align}
	In this case, a solution is given by
	\begin{align}
		\label{eqn:signal_processing_solution}
		\hat{A} = W_A^\dagger \left((C_y^{1/2}W_y^*)^\dagger(P_{\overline{\ran{C_y^{1/2}W_y^*}}}(C_y^{1/2})^\dagger C_{yx}W_x^*P_{\ran{W_A}})_r\right)^*,
	\end{align}
	which also satisfies the minimality property
	\begin{align}
		\label{eqn:signal_processing_minimality}
		P_{\ker{W_A}^\perp}\hat{A}P_{\ker{C_y^{1/2}W_y^*}^\perp}=\hat{A}.
	\end{align}
	Additionally, \eqref{eqn:signal_processing_solution} uniquely specifies a solution of \Cref{problem:signal_processing} that also satisfies \eqref{eqn:signal_processing_minimality} if and only if the truncated SVD in \eqref{eqn:signal_processing_solution} is unique.
	\label{prop:signal_processing}
\end{proposition}
\begin{proof}
	Let $A\in\B_{00,r}(\mathcal{H}_3,\mathcal{H}_2)$.
	Let us also put 
	$B \coloneqq C_y^{1/2}W_y^*\in\B(\mathcal{H}_3,G),$
	$C \coloneqq W_A^*\in\B(\mathcal{H}_1,\mathcal{H}_2),$
	$N \coloneqq BA^*C$ 
	and $M \coloneqq (C_{y}^{1/2})^\dagger C_{yx} W_x^*$. 
	Then $\rank{N}\leq \rank{A^*}=\rank{A}\leq r$, so $N$ is Hilbert--Schmidt. 
	By \Cref{lemma:covariance} \cref{item:cross_covariance_representation}, $C_{yx}=C_y^{1/2}UC_{x}^{1/2}$ with $U\in\mathcal{B}(F,G)$ and $\ran{U}\subset\overline{\ran{C}_y^{1/2}}$. Hence, by \eqref{eqn:third_moore_penrose} and \Cref{prop:ker_ran_relation}, 
	\begin{align}
		\label{eqn:M_computation}
		(C_y^{1/2})^\dagger C_{yx}=P_{\ker{C_y^{1/2}}^\perp}UC_x^{1/2} = P_{\overline{\ran{C_y^{1/2}}}}UC_x^{1/2} = UC_x^{1/2},
	\end{align}
	so that $M$ is a well-defined Hilbert--Schmidt operator, as the product of the Hilbert--Schmidt operator $C_x^{1/2}$ and bounded operators.
	We can now use \Cref{lemma:covariance} \cref{item:covariance_trace,item:covariance_of_linear_transform} to compute 
	\begin{align*}
		\mathbb{E}\norm{W_xx-W_AAW_yy}^2 &= \tr{W_AAW_yC_yW_y^*A^*W_A^*} + \tr{W_xC_xW_x^*}  \\
		&- \tr{W_AAW_yC_{yx}W_x^*}
		-\tr{W_xC_{xy}W_y^*A^*W_A^*}.
	\end{align*}
	Multiplying \eqref{eqn:M_computation} by $C_y^{1/2}$ from the left and using $C_{yx}=C_y^{1/2}UC_x^{1/2}$, we see that $C_y^{1/2}(C_y^{1/2})^\dagger C_{yx}=C_{yx}.$
	As all Hilbert spaces are over $\R$, it holds that $\tr{S^*T}=\tr{ST^*}$ for $S,T\in L_2$. 
	Therefore,
	\begin{align}
		\mathbb{E}\norm{W_xx-W_AAW_yy}^2 =&\tr{W_AAW_yC_yW_y^*A^*W_A^*}\nonumber\\
		&+\tr{W_xC_xW_x^*}- 2\tr{W_AAW_yC_{yx}W_x^*} \nonumber \\
		=& \norm{C_y^{1/2}W_y^*A^*W_A^*}_{L_2}^2 + \norm{W_xC_x^{1/2}}_{L_2}^{2} \nonumber \\
		&- 2\tr{W_AAW_yC_y^{1/2}(C_y^{1/2})^{\dagger}C_{yx}W_x^*}\nonumber \\
		=& \norm{W_xC_x^{1/2}}_{L_2}^2 + \norm{N}_{L_2}^2 - 2\tr{N^*M}\nonumber \\
		=& c + \norm{M-N}_{L_2}^2\nonumber \\
		=& c + \norm{M-BA^*C}_{L_2}^{2} \label{eqn:mean_square_calculation}
	\end{align}
	where $c\coloneqq\norm{W_xC_x^{1/2}}_{L_2}^2 - \norm{M}_{L_2}^2$ does not depend on $A$.
	Since taking the adjoint is a bijective operation from $\B_{00,r}(\mathcal{H}_3,\mathcal{H}_2)$ onto $\B_{00,r}(\mathcal{H}_2,\mathcal{H}_3)$ and since $A^{**}=A$, we see from \eqref{eqn:mean_square_calculation} that $A$ solves \Cref{problem:signal_processing} if and only if $A^*$ solves
	\begin{align}
		\min\{\norm{M-BXC}_{L_2}:\ X\in\B_{00,r}(\mathcal{H}_2,\mathcal{H}_3)\}.
		\label{eqn:bounded_adjoint_problem}
	\end{align}
	Recall $\mathcal{X}_r$ from Setting \ref{setting:main}, with $B$ and $C$ as defined above. The problem \eqref{eqn:bounded_adjoint_problem} does not have a solution when the problem 
	\begin{align}
		\min\{\norm{M-BXC}_{L_2}:\ X\in\mathcal{X}_r\}
		\label{eqn:adjoint_problem}
	\end{align}
	does not have a solution, since $B_{00,r}(\mathcal{H}_2,\mathcal{H}_3)\subset\mathcal{X}_r$.
	If \eqref{eqn:adjoint_problem} does have a solution, then it has a solution $\hat{X}$ of the form \eqref{eqn:hat_X} by \Cref{thm:main}, which then also solves \eqref{eqn:bounded_adjoint_problem}. Indeed, by \Cref{prop:closed_range_thm}, $\ran{W_A^*}$ is closed if and only if $\ran{W_A}$ is closed. By the assumption that $\ran{W_A}$ is closed and by \Cref{rmk:closed_range_case}, $\hat{X}$ is bounded. Thus, \Cref{problem:signal_processing}, \eqref{eqn:bounded_adjoint_problem} and \eqref{eqn:adjoint_problem} are all equivalent, and $X$ solves \Cref{problem:signal_processing} if and only if $X^*$ solves \eqref{eqn:bounded_adjoint_problem} and \eqref{eqn:adjoint_problem}.
	
	Recall the definitions of $B$ and $C$ at the start of the proof. Using \Cref{thm:main} we find that an optimal solution to \eqref{eqn:adjoint_problem} exists if and only if 
	\begin{align*}
		\ran{(P_{\overline{\ran{B}}}(C_y^{1/2})^\dagger C_{yx}W_x^*P_{\ker{C}^\perp})_r}\subset\dom{(C_y^{1/2}W_y^*)^\dagger},
	\end{align*}
	for some truncated rank-$r$ SVD, which by \Cref{rmk:equivalent_range_condition} is equivalent to \eqref{eqn:existence_assumption_signal_processing}. In that case, a solution $\hat{X}$ to \eqref{eqn:adjoint_problem} and \eqref{eqn:bounded_adjoint_problem} that also satisfies 
	\begin{align*}
		P_{\ker{B}^\perp}\hat{X}P_{\overline{\ran{C}}} =\hat{X}
	\end{align*}
	is given by
	\begin{align*}
		\hat{X} = (C_y^{1/2}W_y^*)^\dagger(P_{\overline{\ran{C_y^{1/2}W_y^*}}}(C_y^{1/2})^\dagger C_{yx}W_x^*P_{\ker{W_A^*}^\perp})_r(W_A^*)^\dagger.
	\end{align*}
	This solution is unique if and only if the rank-$r$ truncated SVD in its definition is unique.
	Since $\hat{X}^*$ solves \Cref{problem:signal_processing}, the proof is completed after taking the adjoint and applying \Cref{prop:moore_penrose} \cref{item:mp_inv_adjoint} and \Cref{prop:ker_ran_relation} to $W_A^*$.
\end{proof}

Passing to the adjoint in the formulation \eqref{eqn:adjoint_problem} allows us to use \Cref{thm:main}, as in this case \Cref{thm:main} gives bounded solutions, which are precisely the ones we search for in \Cref{problem:signal_processing}. By \Cref{prop:adjoint} and the discussion following it, every bounded solution can thus be obtained by passing to the adjoint formulation and taking adjoints of the resulting adjoint solutions.
If we would not pass to the adjoint formulation and instead use \Cref{thm:main} on
\begin{align*}
	\mathbb{E}\norm{W_xx-W_AAW_yy}^2 = c + \norm{M^*-C^*AB^*}_{L_2}^{2},
\end{align*}
then we would generally be able to find solutions $\tilde{A}$ without further assumptions, but these may be unbounded since $B^*$ does not have closed range. By \Cref{prop:bounded_solutions}\ref{item:bounded}-\ref{item:adjoint}, assuming boundedness is equivalent to assuming existence condition \eqref{eqn:existence_assumption_signal_processing} for the adjoint solution.

\Cref{problem:signal_processing} considers bounded solutions, 
because otherwise the mean square error $\mathbb{E}\norm{W_xx-W_A\tilde{A}W_yy}^2$ is ill-defined in general. Indeed, it requires $W_yy\in\dom{\tilde{A}}$ almost surely, which is a restrictive condition and is not satisfied in general. In particular, if $W_y=I$ and $y$ has a nondegenerate Gaussian distribution, then $\dom{\tilde{A}}=\ran{C_y^{1/2}}$ is the Cameron--Martin space of $y$ which has zero measure in infinite dimensions, see e.g. \cite[Theorem 2.4.7]{bogachev_gaussian_1998}. Thus, $W_yy\not\in\dom{\tilde{A}}$ almost surely.

When no solution to \Cref{problem:signal_processing} exists, it can still be approximated using the technique described in \Cref{subsec:approximation} and \Cref{rmk:bounded_approximations}, applied to the adjoint formulation \eqref{eqn:adjoint_problem}.

We can now study the problem equivalent to the linear operator learning setting, but with a rank restriction.

\begin{corollary}
	\label{cor:identity_weights_signal_processing}
	A solution to the problem
	\begin{align}
		\label{eqn:unweighted_signal_processing_problem}
		\min\{\mathbb{E}\norm{x-Ay}^2:\ A\in\B_{00,r}(G,F)\}
	\end{align}
	exists if and only if $(\ran{(C_y^{1/2})^\dagger C_{yx}})_r\subset \ran{C_y^{1/2}}$ for some rank-$r$ restricted SVD $(\ran{(C_y^{1/2})^\dagger C_{yx}})_r$. In that case, a solution that also satisfies the minimality property $\hat{A}P_{\ker{C_y}^\perp}=\hat{A}$ is given by
	\begin{align*}
		\hat{A}=\left( (C_y^{1/2})^\dagger (P_{\overline{\ran{C_y^{1/2}}}}(C_y^{1/2})^\dagger C_{yx})_r \right)^*.
	\end{align*}
	This solution is the only solution to \eqref{eqn:unweighted_signal_processing_problem} which satisfies $\hat{A}P_{\ker{C_y}^\perp}=\hat{A}$ if and only if the rank-$r$ reduced SVD is unique.
\end{corollary}

\begin{proof}
	We can apply \Cref{prop:signal_processing} for the case $\mathcal{H}_1=\mathcal{H}_2=F$, $\mathcal{H}_3=G$ and $W_A,W_x,W_y$ equal to the identity operators. By \eqref{eqn:M_computation} it holds that $\ran{(C_y^{1/2})^\dagger C_{yx}}\subset \overline{\ran{C_y^{1/2}}}$. The condition \eqref{eqn:existence_assumption_signal_processing} therefore becomes $\ran{((C_y^{1/2})^\dagger C_{yx})_r}\subset\ran{C_y^{1/2}}$. Finally, \eqref{eqn:signal_processing_minimality} becomes $AP_{\ker{C_y^{1/2}}^\perp}=A$. It remains to show that $\ker{C_y^{1/2}}=\ker{C_y}$. If $0=C_yh=C_y^{1/2}C_y^{1/2}h$, then $C_y^{1/2}h\in\ker{C_y^{1/2}}=\ran{C_y^{1/2}}^\perp$ using \Cref{prop:ker_ran_relation}. Since also $C_y^{1/2}h\in\ran{C_y^{1/2}}$, it follows that $\ker{C_y}\subset\ker{C_y^{1/2}}$. The reverse inclusion is immediate.
\end{proof}

With $U$ as in the proof of \Cref{prop:signal_processing}, the equation \eqref{eqn:M_computation} implies that the existence condition $\ran{((C_y^{1/2})^\dagger C_{yx})_r}\subset\ran{C_y^{1/2}}$ is equivalent to $\ran{(UC_x^{1/2})_r}\subset\ran{C_y^{1/2}}$.

\begin{remark}[Condition comparison with \cite{mollenhauer_learning_2023}]	 
	In \cite{mollenhauer_learning_2023} the analogous problem to \eqref{eqn:unweighted_signal_processing_problem} is studied, in which there is no rank restriction on $A$. There, it is shown in \cite[Proposition 3.11]{mollenhauer_learning_2023} that a solution to this analogous problem exists if and only if the existence condition $\ran{C_{yx}}\subset\ran{C_y}$ holds, in which case the solution is given by $(C_y^{\dagger}C_{yx})^*$.
	Under the latter condition, for all $f\in F$ there exists $h\in G$ such that $C_{yx}f = C_yh$. Hence by \eqref{eqn:third_moore_penrose} and \Cref{prop:ker_ran_relation},
	\begin{align*}
		(C_y^{1/2})^\dagger C_{yx}f = P_{\ker{C_y^{1/2}}^\perp}C_y^{1/2}h=P_{\overline{\ran{C_y^{1/2}}}}C_y^{1/2}h=C_y^{1/2}h\in\ran{C_y^{1/2}}.
	\end{align*}
	Hence,
	\begin{align*}
		\ran{((C_y^{1/2})^\dagger C_{yx})_r}\subset \ran{(C_y^{1/2})^\dagger C_{yx}} \subset\ran{C_y^{1/2}}.
	\end{align*}
	We therefore see that the existence condition in \cite{mollenhauer_learning_2023} is stronger than the existence condition of \Cref{cor:identity_weights_signal_processing}. In particular, whenever one can find an optimal continuous linear operator $A$ that solves the problem in \cite{mollenhauer_learning_2023}, i.e. \eqref{eqn:unweighted_signal_processing_problem} without rank restriction on $A$, then one can also find an optimal continuous and rank restricted operator that solves \eqref{eqn:unweighted_signal_processing_problem}. The converse does not hold in general.
\end{remark}

\begin{remark}[Relation between minimality condition and maximal kernel]
	\label{rmk:maximal_kernel}
	The minimality condition $\hat{A}=\hat{A}P_{\ker{C_y}^\perp}$ of \Cref{cor:identity_weights_signal_processing} is equivalent to $\hat{A}(\ker{C}_y)=0$. In \cite[Proposition 3.11 and Remark 3.12(b)]{mollenhauer_learning_2023} it is shown that the solution of the unconstrained problem also satisfies this property, and in fact has maximal kernel among all solutions. The solution of the rank-constrained problem \eqref{eqn:unweighted_signal_processing_problem} can be seen to have maximal kernel among all solutions as well. Indeed, by applying \Cref{cor:all_solutions} to \eqref{eqn:adjoint_problem} and taking the adjoint, and using $\ker{C_y^{1/2}}=\ker{C_y}$ as explained in the proof of \Cref{cor:identity_weights_signal_processing}, we see that an arbitrary solution of \eqref{eqn:unweighted_signal_processing_problem} is of the form $\tilde{A}\coloneqq \hat{A}+T^*P_{\ker{C}_y}$ for some bounded $T:F\rightarrow G$. Now $\ker{(T^*P_{\ker{C}_y})}^\perp\subset \ker{C}_y$. Furthermore, $\ker{\hat{A}}^\perp\subset\ker{C_y}^\perp$ because $\hat{A}(\ker{\mathcal{C}_y})=0$. Hence $\ker{\tilde{A}}$ is maximal if and only if $T^*P_{\ker{C_y}}=0$, which holds if and only if $\tilde{A}=\hat{A}$. Furthermore, $\tilde{A}=\hat{A}$ holds if and only if $\tilde{A}$ satisfies the minimality property $\tilde{A}P_{\ker{C_y}^\perp}=\tilde{A}$. Thus, the minimality property originating from \eqref{eqn:minimality_constraint} is exactly the maximal kernel property mentioned in \cite[Proposition 3.11 and Remark 3.12(b)]{mollenhauer_learning_2023}.
\end{remark}

\section{Conclusion}
Generalised rank-constrained matrix approximation problems in Frobenius norm have had a myriad of applications, for example in signal processing, statistical inverse problems and data transmission systems.
These finite-dimensional approximation problems have been analysed in \cite{sondermann1986best,friedland_generalized_2007,soto2021error}, for example. In this work we have extended the analysis in the following ways.

In \cite{friedland_generalized_2007}, a solution to the generalised rank-constrained approximation problem was proposed, which also satisfies a minimal norm property. We have shown that this property does not hold in general.  We have proposed the modified minimality property \eqref{eqn:corrected_miniminality_constraint} which is satisfied in general. The modified result in the finite-dimensional setting is given in \Cref{cor:friedland_torokthi_case}.

Secondly, we have studied the infinite-dimensional version of the generalised approximation problem, studying existence, uniqueness and continuity of a solution. In particular, we have given necessary and sufficient conditions for existence and uniqueness of a solution. This is done in \Cref{thm:main}. \Cref{thm:main} also gives explicit solutions that satisfy the minimality property. We have also shown that the proposed solutions---and in fact all solutions---may be discontinuous in \Cref{sec:unbounded_solution}, and have given necessary and sufficient conditions for the proposed solutions to be continuous in \Cref{subsec:boundedness_conditions}. In addition, we have constructed approximate solutions when no solution exists in \Cref{subsec:approximation}, and continuous approximations to discontinuous solutions in \Cref{subsec:bounded:approximations}.

Thirdly and finally, we have studied problems from signal processing, reduced rank regression and linear operator learning in infinite-dimensional settings. In particular, we have generalised the signal processing problem of \cite{torokhti2009towards,torokhti2009filtering} in the case where both the unknown signal and the observation are finite-dimensional random vectors, to the case where both can be random functions. We have also studied the reduced rank regression setting of \cite{kacham2021reduced} to infinite dimensions but with Hilbert--Schmidt norm error instead of operator norm error. Finally, we have discussed the linear operator learning setting of \cite{mollenhauer_learning_2023}, where in our case the operator to be learned is not only bounded or Hilbert--Schmidt but has finite rank. In these settings, we used \Cref{thm:main} to find explicit solutions, and to give necessary and sufficient conditions for existence and uniqueness of solutions.

In summary, since finite-dimensional approximation problems often originate from a discretised version of an infinite-dimensional problem, our infinite-dimensional result paves the way for the analysis of various reduced rank approximation problems such as reduced rank linear operator learning in their native infinite-dimensional formulation.

\section{Acknowledgements}
The research of the authors has been partially funded by the Deutsche Forschungsgemeinschaft (DFG) Project-ID 318763901 -- SFB1294. The authors thank Francesco Carere (Euro-Mediterranean Center on Climate Change) for helpful discussions.
\appendix

\section{Elements from Functional Analysis}
\label{sec:elements_fa}

In this section we recall some key concepts and results and introduce notation. We consider specific properties of linear mappings from $\mathcal{H}$ to $\mathcal{K}$ for possibly infinite-dimensional separable Hilbert spaces $\mathcal{H}$ and $\mathcal{K}$ over $\C$ or $\R$.  In particular, we recall the definition of certain operator classes in \Cref{subsec:linear_operators}, the Moore--Penrose inverse in \Cref{subsec:generalised_inverse}, the SVD theorem for compact operators and the Eckart--Young Theorem in \Cref{subsec:svd}. 

\subsection{Linear Operators}
\label{subsec:linear_operators}
To accommodate the treatment of unbounded operators, we introduce the following concept of a linear map, c.f also \cite[Paragraph X.1]{conway_course_2007}.

A linear operator $T:\mathcal{H}\rightarrow\mathcal{K}$ is a function defined on a linear subspace of $\mathcal{H}$, denoted by $\dom{T}$, on which $T$ acts in a linear fashion. We also write $T:\dom{T}\subset\mathcal{H}\rightarrow\mathcal{K}$ to stress that $T$ is only defined on $\dom{T}$ and not necessarily on all of $\mathcal{H}$. When $\dom{T}\subset\mathcal{H}$ densely, we call $T$ `densely defined'. Without loss of generality, one can always assume that $T$ is densely defined when it is linear, as one can put $T$ equal to 0 on $\overline{\dom{T}}^\perp={\dom{T}}^\perp$ and $\dom{T}\oplus\dom{T}^\perp$ is dense in $\mathcal{H}= \overline{\dom{T}}\oplus\dom{T}^\perp$. Here, $U \oplus V$ denotes the direct sum of linear subspaces $U,V$ and $U^\perp$ the orthogonal complement of $U$. If also $S:\mathcal{H}\rightarrow\mathcal{K}$ is linear and $\dom{T}\subset\dom{S}$ and $S=T$ on $\dom{T}$, then we call $S$ an `extension' of $T$. We define the kernel of $T$ by $\ker{T}\coloneqq \{h\in\dom{T}:\ Th=0\}$. If 
\begin{align*}
	\norm{T}\coloneqq \inf\{C>0:\ \norm{Th}\leq C\norm{h},\ h\in\dom{T}\}<\infty
\end{align*}
then we say that $T$ is `bounded', and call $\norm{T}$ its `operator norm'. Otherwise, $T$ is called `unbounded'. As $T$ is linear, it is bounded if and only if it is continuous.
If $T$ is bounded, then it has by continuity a unique bounded extension to $\overline{\dom{T}}$. If $T$ is bounded and $\dim\ran{T}<\infty$, then we say that $T$ is `finite-rank' and define $\rank{T}\coloneqq\dim\ran{T}$. If also $S:\mathcal{Z}\rightarrow\mathcal{H}$ is a linear operator for some separable Hilbert space $\mathcal{Z}$, then we define $T+S$ on $\dom{T+S} \coloneqq \dom{S}\cap\dom{T}$, and $TS$ on $\dom{TS}\coloneqq S^{-1}(\dom{T})$. If $S$ is finite-rank, then so is $TS$. We record this fact below for future reference.
\begin{lemma}
	Let $T:\mathcal{H}\rightarrow\mathcal{K}$ be linear. If $S:\mathcal{Z}\rightarrow\mathcal{H}$ is a finite-rank operator, then so is $TS$ and $\rank{TS}\leq \rank{S}$.
	\label{lemma:unbounded_after_finite_rank}
\end{lemma}
\begin{proof}
	As $\rank{S}$ is finite, $\restr{T}{\ran{S}}$ is bounded, since every linear operator is bounded on a finite-dimensional domain. 	Also,
	\begin{align*}
		\rank{TS}&= \rank{\restr{T}{\ran{S}}}=\dim{\ran{S}}-\dim{\ker{\restr{T}{\ran{S}}}}\\
		&\leq \dim{\ran{S}}= \rank{S},
	\end{align*}
	by the rank-nullity theorem in linear algebra, which can be applied because $\dom{\restr{T}{\ran{S}}}=\ran{S}$ is finite-dimensional.
\end{proof}

The following is a classical example of an unbounded operator.
\begin{example}
	Let $\mathcal{H}$ be an infinite-dimensional separable Hilbert space and $(e_n)_n$ an orthonormal system in $\mathcal{H}$. Let $(\gamma_n)_n\subset\C\backslash\{0\}$ be a sequence converging to zero. Let $\mathcal{D}\coloneqq\{h\in\mathcal{H}:\ \sum_{n}^{}\gamma_n^{-2}\abs{\langle h,e_n\rangle}^2<\infty\}$. We define the mapping,
	\begin{align*}
		T:\mathcal{D}\subset\mathcal{H}\rightarrow\mathcal{H},\quad h\mapsto \sum_{n}^{}\gamma_n^{-1}\langle h,e_n\rangle e_n. 
	\end{align*}
	Then $T$ is linear, but not bounded, since $\norm{Te_n}=\abs{\gamma_n}^{-1}$ diverges as $n\rightarrow\infty$. Furthermore, $T$ is densely defined if and only if $(e_n)_n$ forms an orthonormal basis. We can extend $T$ to a densely defined operator $\bar{T}$ on $\mathcal{H}$, by putting $\bar{T}h=0$ for $h\in\mathcal{D}^\perp$ and $\bar{T}h=Th$ for $h\in\mathcal{D}$. Then $\bar{T}$ is a well-defined linear operator on $\mathcal{D}\oplus \mathcal{D}^\perp$, which is dense in $\mathcal{H}$.
	\label{ex:unbounded_operator}
\end{example}
We endow $\B(\mathcal{H}$, $\mathcal{K})$, $\B_0(\mathcal{H}$, $\mathcal{K})$, $\B_{00}(\mathcal{H},\mathcal{K})$ and $\B_{00,r}(\mathcal{H},\mathcal{K})$ with the operator norm. We recall that an operator between Hilbert spaces is compact when it is a limit of finite-rank operators in operator norm. When $\mathcal{H}=\mathcal{K}$ we simply write $\B(\mathcal{H})$, $\B_0(\mathcal{H})$, $\B_{00}(\mathcal{H})$ and $\mathcal{B}_{00,r}(\mathcal{H})$. For $T\in\B(\mathcal{H},\mathcal{K})$ we denote its Hilbert space adjoint by $T^*\in\B(\mathcal{K},\mathcal{H})$. That is, $T^*$ is the unique operator satisfying
$\langle Th,k\rangle = \langle h,T^*k\rangle$ for $h\in \mathcal{H},k\in \mathcal{K}$.
If $T=T^*$, then $T$ is `self-adjoint'. If $\langle Th,h\rangle \geq 0$ for all $h\in\mathcal{H}$, then $T$ is `positive'. If $T$ is positive and self-adjoint, there exists a unique positive self adjoint square root $T^{1/2}$ with $T^{1/2}T^{1/2}=T$. For $T\in\mathcal{B}(\mathcal{H},\mathcal{K})$ arbitrary, we have that $T^*T$ is positive and self adjoint and can define $\abs{T}\coloneqq (T^*T)^{1/2}$. We refer to \cite[Sections II.2, II.4, VIII.3]{conway_course_2007} for more on the above notions.
The following relation between $T$ and $T^*$ holds. Note that it implies that $\ker{T}=\ran{T^*}^\perp$ and $\ker{T^*}=\ran{T}^\perp$.
\begin{proposition}
	Let $T\in\B(\mathcal{H},\mathcal{K})$. Then $\overline{\ran{T}}=\ker{T^*}^\perp$ and $\overline{\ran{T^*}}=\ker{T}^\perp$.
	\label{prop:ker_ran_relation}
\end{proposition}
This can be found in for example \cite[Theorem VI.1.8]{conway_course_2007}. We also recall part of the closed range theorem, c.f.\ \cite[Theorem VI.1.10]{conway_course_2007}.
\begin{proposition}
	Let $T\in\B(\mathcal{H},\mathcal{K})$. Then $\ran{T}$ is closed if and only if $\ran{T^*}$ is closed.
	\label{prop:closed_range_thm}
\end{proposition}

We also recall the definition of Hilbert--Schmidt operators and the Hilbert--Schmidt norm. If $T\in\B(\mathcal{H},\mathcal{K})$ satisfies $\sum_{i}^{}\norm{Te_i}^2<\infty$ for some ONB $(e_i)_i$ of $\mathcal{H}$, then we recall that $T$ is called a `Hilbert--Schmidt' operator and $\norm{T}_{L_2}\coloneqq\left(\sum_{i}^{}\norm{Te_i}^2\right)^{1/2}$ is independent of the choice of $(e_i)_i$. The linear space consisting of all Hilbert--Schmidt operators $L_2(\mathcal{H},\mathcal{K})\coloneqq\{T\in\B(\mathcal{H},\mathcal{K}):\ \sum_{i}^{}\norm{Te_i}^2<\infty\}$ becomes a Hilbert space, with inner product $\langle T,S\rangle_{L_2}\coloneqq\sum_{i}^{}\langle Te_i,Se_i\rangle$, which is again independent of the choice of ONB. We call $\langle\cdot,\cdot\rangle_{L_2}$ the `trace inner product', and it induces the Hilbert--Schmidt norm $\norm{\cdot}_{L_2}$. 
For $h\in\mathcal{H},k\in\mathcal{K}$, we can interpret the tensor product $k\otimes h\in \mathcal{H}\otimes\mathcal{K}\simeq L_2(\mathcal{H},\mathcal{K})$ as the rank-1 map $\tilde{h}\mapsto \langle \tilde{h},h\rangle k$.
If $T\in\mathcal{B}(\mathcal{H},\mathcal{K})$ is the product of two Hilbert--Schmidt operators, then it is called `trace-class'. For a trace-class operator $T$ we define its trace $\tr{T}\coloneqq \sum_{i}^{}\langle Te_i,f_i\rangle$, where $(e_i)_i$ and $(f_i)_i$ are any ONBs of $\mathcal{H}$ and $\mathcal{K}$ respectively. The value of $\tr{T}$ does not depend on the choice of these ONBs. If $T\in L_2(\mathcal{H},\mathcal{K})$, then $\abs{T}^{1/2}$ is trace-class. We recall that a trace-class operator is Hilbert--Schmidt, and a Hilbert--Schmidt operator is compact. Finally, we notice that by the above definitions, $\langle T,S\rangle_{L_2}=\tr{S^*T}$ for $S,T\in L_2(\mathcal{H},\mathcal{K})$. For a discussion on Hilbert--Schmidt and trace-class operators, we refer to \cite[Section 3.18]{conway_course_2000}.

\subsection{Moore--Penrose inverse}
\label{subsec:generalised_inverse}
We recall the definitions of the Moore--Penrose inverse on Hilbert spaces, see also for example \cite[Section 2.1]{engl_regularization_1996} or \cite[Section 3.5]{hsing_theoretical_2015}.

Let $T\in\B(\mathcal{H},\mathcal{K})$. The Moore--Penrose inverse $T^\dagger$ of $T$ is given by a densely defined linear operator $\mathcal{K}\rightarrow \mathcal{H}$ constructed as follows. The continuous linear map $\tilde{T}\coloneqq\restr{T}{\ker{T}^\perp}:\ker{T}^\perp\rightarrow \mathcal{K}$ is injective and hence invertible on $\ran{T}$. The inverse $\tilde{T}^{-1}$ is thus well-defined on $\ran{T}$. Also, $\tilde{T}^{-1}$ is bounded if and only if $\ran{T}$ is closed. Let $\dom{T^\dagger}\coloneqq\ran{T} \oplus \ran{T}^\perp$, 
which is dense in $\overline{\ran{T}}\oplus \ran{T}^\perp= \mathcal{K}$. Since every $x\in\mathcal{K}$ has a unique decomposition $x=h+k$ for $h\in\ran{T}$ and $k\in\ran{T}^\perp$, we can define the action of $T^\dagger$ on $\dom{T^\dagger}$ unambiguously via $T^\dagger(h+k)\coloneqq\tilde{T}^{-1}h$. Thus, $T^\dagger:\mathcal{K}\rightarrow\mathcal{H}$ is linear.

The following holds, the proof of which can be found in \cite[Propositions 2.3 and 2.4]{engl_regularization_1996} for \cref{item:mp_inv_range,item:mp_inv_equations,item:mp_inv_closed,item:mp_inv_boundedness}. \Cref{item:mp_inv_adjoint} is a special case of \cite[Theorem 9.3.2 (g)]{ben-israel_generalized_2003}, as the conclusion also holds for closed operators and for operators that do not satisfy the closed range condition.
\begin{proposition}
	\label{prop:moore_penrose}
	Let $\mathcal{H},\mathcal{K}$ be Hilbert spaces and $T\in\B(\mathcal{H},\mathcal{K})$. Then the following hold:
	\begin{enumerate}
		\item 
			$\ran{T^\dagger} = \ker{T}^\perp$.
			\label{item:mp_inv_range}
		\item
			$T$ and $T^\dagger$ satisfy the `Moore--Penrose equations': 
			\begin{align}
				TT^\dagger T &= T,\label{eqn:first_moore_penrose}\\
				T^\dagger TT^\dagger &=  T^\dagger,\label{eqn:second_moore_penrose}\\
				T^\dagger T &=  P_{\ker{T}^\perp},\label{eqn:third_moore_penrose}\\
				TT^\dagger &= \restr{P_{\overline{\ran{T}}}}{\dom{T^\dagger}}.\label{eqn:fourth_moore_penrose}
			\end{align}
			\label{item:mp_inv_equations}
		\item
			$T^\dagger$ is a closed operator, i.e. it has closed graph.
			\label{item:mp_inv_closed}
		\item $T^\dagger$ is bounded if and only if $T$ has closed range.
			\label{item:mp_inv_boundedness}
		\item If $T$ has closed range, then $(T^*)^\dagger=(T^\dagger)^*$.
			\label{item:mp_inv_adjoint}
	\end{enumerate}
\end{proposition}

We consider an example in which the Moore--Penrose operator is unbounded. 
\begin{example}
	\label{ex:moore_penrose_inverse}
	Assume the setting of \Cref{ex:unbounded_operator} with $(e_n)_n$ an orthonormal basis. Hence $\mathcal{D}\subset\mathcal{H}$ densely, $\mathcal{D}^\perp=\{0\}$ and $\bar{T}=T$. We define $S\in\mathcal{B}_0(\mathcal{H})$ as
	\begin{align*}
		Sh = \sum_{n}^{}\gamma_n\langle h,e_n\rangle e_n,\quad h\in\mathcal{H}.
	\end{align*}
	As $\gamma_n\rightarrow 0$, $S$ is indeed compact as the limit of the finite-rank operators $h\mapsto \sum_{n=1}^{N}\gamma_n\langle h,e_n\rangle e_n$ in $B(\mathcal{H},\mathcal{K})$ when $N\rightarrow\infty$.
	Also, $\ran{S}=\mathcal{D}$ and $\ker{S}=\{0\}$. If $k\in\mathcal{D}$, then $h\coloneqq \sum_{n}^{}\gamma_n^{-1}\langle k,e_n\rangle e_n=Tk\in\mathcal{H}$ and $k=Sh$. Thus $S^\dagger:\mathcal{D}= \mathcal{D}\oplus\mathcal{D}^\perp\rightarrow\mathcal{H}$ is given by ${T}$. The condition $k\in\mathcal{D}$ is called the ``Picard criterion'', see e.g. \cite[Theorem 2.8]{engl_regularization_1996}, and under this condition $S^\dagger k$ is well-defined.
\end{example}

\subsection{Singular Value Decomposition and Eckart--Young Theorem}
\label{subsec:svd}
Let $T\in\B_0(\mathcal{H},\mathcal{K})$. We define $\mathcal{I}\coloneqq\{1,\ldots,\dim{\ran{T}}\}$. Thus, if $T$ has infinite-dimensional range, then $\mathcal{I}$ contains infinitely many elements. If $T$ has finite-dimensional range, then $\mathcal{I}$ contains $\rank{T}$ elements. Then there exists a nonincreasing sequence $(\sigma_i)_{i\in\mathcal{I}}\subset[0,\infty)$, an orthonormal system $(e_i)_{i\in\mathcal{I}}$ of $\mathcal{H}$ and an orthonormal system $(f_i)_{i\in\mathcal{I}}$ of $\mathcal{K}$ such that 
\begin{align*}
	T = \sum_{i\in\mathcal{I}}^{}\sigma_i f_i\otimes e_i,\quad\text{i.e. } Th = \sum_{i\in\mathcal{I}}^{}\sigma_i \langle h,e_i\rangle f_i,\quad h\in\mathcal{H},
\end{align*}
see e.g. \cite[Section 4.3]{hsing_theoretical_2015}.
We call such a series expansion of $T$ an SVD of $T$. The $i$-th singular value $\sigma_i$ is also denoted by $\sigma_i(T)$ to express its dependence on $T$. The summation can be finite or countably infinite, depending on whether $T$ is finite-rank or not. For any $r\in\N$ we can truncate the sum
\begin{align}
	(T)_r \coloneqq \sum_{i=1}^{r}\sigma_i f_i\otimes e_i.
	\label{eqn:definition_svd}
\end{align}
This truncation is uniquely defined if and only if one of the following conditions holds: $r\geq\dim{\ran{T}}$  or $\sigma_r>\sigma_{r+1}$. 
The first condition expresses the scenario in which $T$ is a finite-rank operator with rank at most $r$, so that $(T)_r=T$. The second condition, $\sigma_r>\sigma_{r+1}$, implies that $T$ is not finite-rank or that $T$ is finite-rank but its rank is at least $r$. Both conditions hold simultaneously if and only if $\rank{T}=r$.

The following result is known as the Eckart--Young Theorem in the case of finite-dimensional matrices. See for example, \cite[Theorem 4.4.7]{hsing_theoretical_2015}.
\begin{proposition}[Eckart--Young Theorem]
	Let $T\in L_2(\mathcal{H},\mathcal{K})$. Then $(T)_r$ solves the problem,
	\begin{align*}
		\min\{\norm{T-S}_{L_2}:\ S\in\B_{00,r}(\mathcal{H},\mathcal{K})\}.
	\end{align*}
	\label{prop:eckart_young}
\end{proposition}
The SVD of self-adjoint compact operators reduces to $T=\sum_{i\in\mathcal{I}}^{}\lambda_i e_i\otimes e_i$ with real eigenvalues $\lambda_i\rightarrow 0$ and orthonormal basis $(e_i)_i$, c.f.\ the spectral theorem for self-adjoint and compact operators \cite[Theorem II.5.1]{conway_course_2007}.
The Hilbert--Schmidt norm can now be calculated in the following way.
\begin{proposition}
	Let $T\in L_2(\mathcal{H},\mathcal{K})$, then $\norm{T}_{L_2}^2 = \sum_{i}^{}\sigma_i(T^*T)=\sum_{i}^{}\sigma_i(TT^*)$.
	\label{prop:hs_calc_from_svd}
\end{proposition}
\begin{proof}
	We have $T^*T=\sum_{i}^{}\sigma_i(T^*T)e_i\otimes e_i$, with ONB $(e_i)_i$ of eigenvectors and real eigenvalues $(\sigma(T^*T)_i)_i$. Hence, $\norm{T}_{L_2}^2=\tr{T^*T}=\sum_{i.j}^{}\langle T^*Te_i,e_j\rangle=\sum_{i}^{}\sigma_i(T^*T)$. The proof for $TT^*$ is analogous.
\end{proof}

\section{Proofs of results}
\label{sec:proofs}
\begin{lemma}
	Let $\mathcal{H},\mathcal{K}$ be normed spaces and $T:\dom{T}\subset\mathcal{H}\rightarrow \mathcal{K}$. Suppose that $r\coloneqq\dim{\ran{T}}<\infty$. Then $T$ is continuous if and only if $\ker{T}$ is closed in $\dom{T}$.
	\label{lemma:closed_kernel_and_finite_dim_range}
\end{lemma}
\begin{proof}
	If $T$ is continuous, then $\ker{T}=T^{-1}(\{0\})$ is closed as $\{0\}$ is closed. For the converse, note that $\ker{T}\subset\dom{T}$ closed implies that $\dom{T}/\ker{T}$ is a normed space and the quotient map $Q:\dom{T}\rightarrow\dom{T}/\ker{T}$, $h\mapsto h+\ker{T}$ is continuous, see \cite[Section III.4]{conway_course_2007}. Let us define $S:\dom{T}/\ker{T}\rightarrow\mathcal{K}$, $S(h+\ker{T})=Th$. Then $S$ is a well-defined and linear map satisfying $T=S\circ Q$. Thus $\ran{S}=\ran{T}$, and if $S(h+\ker{T})=0$, then $Th=0$, so that $h\in\ker{T}$ and $S$ is injective. This implies that $\dim{\dom{T}/\ker{T}}=r$. Thus $S$ is continuous as a linear map defined on a finite-dimensional space, c.f.\ \cite[Proposition III.3.4]{conway_course_2007}. We conclude that $T$ is continuous as a composition of continuous maps.
\end{proof}

\covlemma*
\begin{proof}
	Let $(e_i)_i$ be an ONB of $E_1$. Applying the Cauchy-Schwarz inequality twice, we get
	\begin{align*}
		\mathbb{E}\sum_{i}^{}\langle w,e_i\rangle\langle e_i,z\rangle = \mathbb{E}\langle w,z\rangle \leq \mathbb{E}\norm{z}\norm{w}\leq (\mathbb{E}\norm{z}^2)^{1/2}(\mathbb{E}\norm{w}^2)^{1/2}<\infty.
	\end{align*}
	Hence by Fubini's Theorem,
	\begin{align*}
		\tr{C_{wz}}=\sum_{i}^{}\langle C_{wz}e_i,e_i\rangle = \sum_{i}^{}\mathbb{E}\langle w,e_i\rangle \langle e_i,z\rangle =\mathbb{E}\sum_{i}^{}\langle w,e_i\rangle\langle e_i,z\rangle = \mathbb{E}\langle z,w\rangle.
	\end{align*}
	This proves \cref{item:covariance_trace}. For \cref{item:covariance_of_linear_transform}, we note that for all $e\in E_1$ and $f\in E_2$,
	\begin{align*}
		\langle C_{Qw,z}e,f\rangle = \mathbb{E}\langle f,Qw\rangle\langle e,z\rangle=\mathbb{E}\langle Q^*f,w\rangle\langle e,z\rangle=\langle C_{wz}e,Q^*f\rangle=\langle QC_{wz}e,f\rangle.
	\end{align*}
	The statement of \cref{item:cross_covariance_representation} is proven in \cite[Theorem 1]{baker_joint_1973} in the case of centered covariances. However, as noted in \cite[Proof of Theorem 5.4]{klebanov_rigorous_2020}, the proof also works for uncentered covariances, showing that \cref{item:cross_covariance_representation} holds.  
\end{proof}

\bibliographystyle{plain}
\bibliography{references}

@article{baker_joint_1973,
   AUTHOR = {Baker, Charles R.},
     TITLE = {Joint measures and cross-covariance operators},
   JOURNAL = {Trans. Amer. Math. Soc.},
  FJOURNAL = {Transactions of the American Mathematical Society},
    VOLUME = {186},
      YEAR = {1973},
     PAGES = {273--289},
       DOI = {10.2307/1996566},
}

@inproceedings{bakshi2020robust,
  author       = {Ainesh Bakshi and Nadiia Chepurko and David P. Woodruff},
  editor       = {Sandy Irani},
  title        = {Robust and Sample Optimal Algorithms for {PSD} Low Rank Approximation},
  booktitle    = {61st {IEEE} Annual Symposium on Foundations of Computer Science, {FOCS} 2020, Durham, NC, USA, November 16-19, 2020},
  pages        = {506--516},
  publisher    = {{IEEE}},
  year         = {2020},
  url          = {https://doi.org/10.1109/FOCS46700.2020.00054},
  doi          = {10.1109/FOCS46700.2020.00054},
}

@article {begovic2022hybrid,
    AUTHOR = {Begovi\'{c} Kova\v{c}, Erna},
     TITLE = {Hybrid {CUR}-type decomposition of tensors in the {T}ucker
              format},
   JOURNAL = {BIT},
  FJOURNAL = {BIT. Numerical Mathematics},
    VOLUME = {62},
      YEAR = {2022},
    NUMBER = {1},
     PAGES = {125--138},
       DOI = {10.1007/s10543-021-00876-x},
}

@book{ben-israel_generalized_2003,
    AUTHOR = {Ben-Israel, Adi and Greville, Thomas N. E.},
     TITLE = {Generalized inverses: {Theory and applications}},
    SERIES = {CMS Books in Mathematics},
    VOLUME = {15},
   EDITION = {Second},
 PUBLISHER = {Springer-Verlag, New York},
      YEAR = {2003},
     PAGES = {xvi+420},
       doi = {10.1007/b97366},
}

@book{bogachev_gaussian_1998,
    AUTHOR = {Bogachev, Vladimir I.},
     TITLE = {Gaussian measures},
    SERIES = {Mathematical Surveys and Monographs},
    VOLUME = {62},
 PUBLISHER = {American Mathematical Society, Providence, RI},
      YEAR = {1998},
     PAGES = {xii+433},
       DOI = {10.1090/surv/062},
}

@article{bousserez2018optimal,
  title={Optimal and scalable methods to approximate the solutions of large-scale {Bayesian} problems: theory and application to atmospheric inversion and data assimilation},
  author={Bousserez, Nicolas and Henze, Daven K},
  fjournal={Quarterly Journal of the Royal Meteorological Society},
  journal={Quart. J. Roy. Meteorolog. Soc.},
  volume={144},
  number={711},
  pages={365--390},
  year={2018},
  doi = {10.1002/qj.3209}
}

@inproceedings{boutsidis2016optimal,
  title={Optimal principal component analysis in distributed and streaming models},
  author={Boutsidis, Christos and Woodruff, David P and Zhong, Peilin},
  booktitle    = {Proceedings of the 48th Annual {ACM} {SIGACT} Symposium on Theory of Computing, {STOC} 2016, Cambridge, MA, USA, June 18-21, 2016},
  pages={236--249},
  year={2016},
  editor       = {Daniel Wichs and Yishay Mansour},
  doi          = {10.1145/2897518.2897646},
}

@article{chavarria2022effective,
  title={Effective implementation to reduce execution time of a low-rank matrix approximation problem},
  author={Chavarr{\'\i}a-Molina, Jeffry and Fallas-Monge, Juan Jos{\'e} and Soto-Quiros, Pablo},
  journal={J. Comput. Appl. Math.},
  fjournal={Journal of Computational and Applied Mathematics},
  volume={401},
  pages={113763},
  year={2022},
  doi = {10.1016/j.cam.2021.113763},
}

@inproceedings{chen2018visual,
  title={Visual-{SLIM}: Integrated sparse linear model with visual features for personalized recommendation},
  author={Chen, Siyang and Xue, Feng and Zhang, Haobo},
  booktitle    = {Advances in Multimedia Information Processing - {PCM} 2018 - 19th Pacific-Rim Conference on Multimedia, Hefei, China, September 21-22, 2018, Proceedings, Part {I}},
  pages={126--135},
  year={2018},
  editor       = {Richang Hong and Wen{-}Huang Cheng and Toshihiko Yamasaki and Meng Wang and Chong{-}Wah Ngo},
  series       = {Lecture Notes in Computer Science},
  volume       = {11164},
  publisher    = {Springer},
  doi          = {10.1007/978-3-030-00776-8\_12},  
}

@article{chung2015optimal,
  title={Optimal regularized low rank inverse approximation},
  author={Chung, Julianne and Chung, Matthias and O'Leary, Dianne P},
  journal = {Linear Algebra Appl.}, 
  fjournal={Linear Algebra and its Applications},
  volume={468},
  pages={260--269},
  year={2015},
  doi = {10.1016/j.laa.2014.07.024},
}

@article{chung2014efficient,
  title={An efficient approach for computing optimal low-rank regularized inverse matrices},
  author={Chung, Julianne and Chung, Matthias},
  journal={Inverse Problems},
  volume={30},
  number={11},
  pages={114009},
  year={2014},
  doi={10.1088/0266-5611/30/11/114009},
}

@inproceedings{chung2013computing,
  title={Computing optimal low-rank matrix approximations for image processing},
  author={Chung, Julianne and Chung, Matthias},
  booktitle={2013 Asilomar Conference on Signals, Systems and Computers},
  pages={670--674},
  year={2013},
  editor       = {Michael B. Matthews},
  booktitle    = {2013 Asilomar Conference on Signals, Systems and Computers, Pacific Grove, CA, USA, November 3-6, 2013},
  pages        = {670--674},
  publisher    = {{IEEE}},
  year         = {2013},
  doi          = {10.1109/ACSSC.2013.6810366},
}

@article{chung2017optimal,
  title={Optimal regularized inverse matrices for inverse problems},
  author={Chung, Julianne and Chung, Matthias},
  journal = {SIAM J. Matrix Anal. Appl.},
  fjournal={SIAM Journal on Matrix Analysis and Applications},
  volume={38},
  number={2},
  pages={458--477},
  year={2017},
   doi = {10.1137/16M1066531},
}

@book{cohn_measure_2013,
    AUTHOR = {Cohn, Donald L.},
     TITLE = {Measure theory},
    SERIES = {Birkh\"auser Advanced Texts},
   EDITION = {Second},
 PUBLISHER = {Birkh\"auser/Springer, New York},
      YEAR = {2013},
     PAGES = {xxi+457},
       DOI = {10.1007/978-1-4614-6956-8},
}

@book{conway_course_2000,
    AUTHOR = {Conway, John B.},
     TITLE = {A course in operator theory},
    SERIES = {Graduate Studies in Mathematics},
    VOLUME = {21},
 PUBLISHER = {American Mathematical Society, Providence, RI},
      YEAR = {2000},
     PAGES = {xvi+372},
       DOI = {10.1090/gsm/021},
}

@book{conway_course_2007,
    AUTHOR = {Conway, John B.},
     TITLE = {A course in functional analysis},
    SERIES = {Graduate Texts in Mathematics},
    VOLUME = {96},
   EDITION = {Second},
 PUBLISHER = {Springer-Verlag, New York},
      YEAR = {1990},
     PAGES = {xvi+400},
       DOI = {10.1007/978-1-4757-4383-8},
}

@article{de_hoop_convergence_2023,
	title = {Convergence Rates for Learning Linear Operators from Noisy Data},
    journal = {SIAM/ASA J. Uncertainty Quantif.},
    fjournal = {SIAM/ASA Journal on Uncertainty Quantification},
    volume = {11},
    number = {2},
    pages = {480-513},
    year = {2023},
    doi = {10.1137/21M1442942},
	author = {de Hoop, Maarten V. and Kovachki, Nikola B. and Nelsen, Nicholas H. and Stuart, Andrew M.},
}

@book{engl_regularization_1996,
    AUTHOR = {Engl, Heinz W. and Hanke, Martin and Neubauer, Andreas},
     TITLE = {Regularization of inverse problems},
    SERIES = {Mathematics and its Applications},
    VOLUME = {375},
 PUBLISHER = {Springer Dordrecht},
      YEAR = {1996},
     PAGES = {viii+321},
       url = {https://link.springer.com/book/9780792341574},
    EDITION = {First}
}

@book{friedland2016matrices,
  title = {Matrices---algebra, analysis and applications},
  author={Friedland, Shmuel},
  publisher={World Scientific Publishing},
  address={New Jersey},
  year={2015},
  pages = {xii+582},
  doi = {10.1142/9567},
}

@article{friedland_generalized_2007,
	title = {Generalized {Rank}-{Constrained} {Matrix} {Approximations}},
	volume = {29},
	doi = {10.1137/06065551},
	number = {2},
	urldate = {2023-10-20},
	fjournal = {SIAM Journal on Matrix Analysis and Applications},
	journal = {SIAM J. Matrix Anal. Appl.},
	author = {Friedland, Shmuel and Torokhti, Anatoli},
	year = {2007},
	pages = {656--659},	
}

@article{grant2014efficient,
  title={Efficient Compression of Distributed Information in Estimation Fusion},
  author={Grant, Alex and Torokhti, Anatoli and Miklavcic, Stan},
  fjournal={Electronic Notes in Discrete Mathematics},
  journal={Electron. Notes Discrete Math.},
  doi={10.1016/j.endm.2014.08.039},
  volume={46},
  pages={297--304},
  year={2014},
}

@phdthesis{holodnak2015topics,
	address = {Raleigh, North Carolina},
	title = {Topics in {Randomized} {Algorithms} for {Numerical} {Linear} {Algebra}},
	school = {North Carolina State University},
	author = {Holodnak, John T.},
	year = {2015},
	url = {https://repository.lib.ncsu.edu/bitstream/handle/1840.16/10327/etd.pdf}
}

@article{howlett2022multilinear,
  title={Multilinear {Karhunen}-{Lo{\`e}ve} Transforms},
  author={Howlett, Phil and Torokhti, Anatoli and Pudney, Peter and Soto-Quiros, Pablo},
  journal      = {{IEEE} Trans. Signal Process.},
  fjournal={IEEE Transactions on Signal Processing},
  volume={70},
  pages={5148--5163},
  year={2022},
  doi          = {10.1109/TSP.2022.3214684},
}

@book{hsing_theoretical_2015,
	series = {Wiley {Series} in {Probability} and {Statistics}},
	title = {Theoretical {Foundations} of {Functional} {Data} {Analysis}, with an {Introduction} to {Linear} {Operators}},
	publisher = {John Wiley \& Sons, Ltd., Chichester},
	author = {Hsing, Tailen and Eubank, Randall},
	year = {2015},
	doi = {10.1002/9781118762547},
    pages = {xiv+334},
}

@article{huang_approximation_2006,
	title = {Approximation theorems of {Moore}–{Penrose} inverse by outer inverses},
	volume = {15},
	journal = {Numer. Math. J. Chinese Univ. (Engl. Ser.)},
  fjournal = {Numerical Mathematics. A Journal of Chinese Universities.
              English Series},
	author = {Huang, Qianglian and Fang, Zheng},
	year = {2006},
    number = {2},
     pages = {113--119},
}

@book{hytonen_analysis_2016,
     TITLE = {Analysis in {B}anach spaces. {V}ol. {I}. {M}artingales and
              {L}ittlewood-{P}aley theory},
	AUTHOR = {Hytönen, Tuomas and Van Neerven, Jan and Veraar, Mark and Weis, Lutz},
	YEAR = {2016},
	DOI = {10.1007/978-3-319-48520-1},
    SERIES = {Series of Modern Surveys in Mathematics},
    VOLUME = {63},
 PUBLISHER = {Springer, Cham},
      YEAR = {2016},
     PAGES = {xvi+614},
}

@article{izenman_reduced-rank_1975,
	title = {Reduced-rank regression for the multivariate linear model},
	volume = {5},
	doi = {10.1016/0047-259X(75)90042-1},
	number = {2},
	journal = {J. Multivariate Anal.},
	fjournal = {Journal of Multivariate Analysis},
	author = {Izenman, Alan Julian},
	year = {1975},
	pages = {248--264},
}

@inproceedings{kacham2021reduced,
  title={Reduced-Rank Regression with Operator Norm Error},
  author={Kacham, Praneeth and Woodruff, David},
  booktitle    = {Conference on Learning Theory, {COLT} 2021, 15-19 August 2021, Boulder, Colorado, {USA}},  
  pages={2679--2716},
  year={2021},
  editor       = {Mikhail Belkin and Samory Kpotufe},
  series       = {Proceedings of Machine Learning Research},
  volume       = {134},
  publisher    = {{PMLR}},
  url          = {http://proceedings.mlr.press/v134/kacham21a.html},
}

@article{klebanov_rigorous_2020,
	title = {A {Rigorous} {Theory} of {Conditional} {Mean} {Embeddings}},
	volume = {2},
	doi = {10.1137/19M1305069},
	number = {3},
	fjournal = {SIAM Journal on Mathematics of Data Science},
	journal = {SIAM J. Math. Data Science},
    author = {Klebanov, Ilja and Schuster, Ingmar and Sullivan, T. J.},
	year = {2020},
	pages = {583--606},
}

@article{kovachki_neural_2024,
	title = {Neural operator: learning maps between function spaces with applications to {PDEs}},
	volume = {24},
	number  = {89},
	fjournal = {Journal of Machine Learning Research},
	journal = {J. Mach. Learn. Res.},
	author = {Kovachki, Nikola and Li, Zongyi and Liu, Burigede and Azizzadenesheli, Kamyar and Bhattacharya, Kaushik and Stuart, Andrew and Anandkumar, Anima},
	year = {2023},
	pages   = {1--97},
    url     = {http://jmlr.org/papers/v24/21-1524.html}
}

@inproceedings{langs2011learning,
  title={Learning an atlas of a cognitive process in its functional geometry},
  author={Langs, Georg and Lashkari, Danial and Sweet, Andrew and Tie, Yanmei and Rigolo, Laura and Golby, Alexandra J and Golland, Polina},
  booktitle={Information Processing in Medical Imaging},
  pages={135--146},
  year={2011},
  publisher={Springer Berlin Heidelberg},
  doi = {10.1007/978-3-642-22092-0_12}
}

@article{langs2014decoupling,
  title={Decoupling function and anatomy in atlases of functional connectivity patterns: {L}anguage mapping in tumor patients},
  author={Langs, Georg and Sweet, Andrew and Lashkari, Danial and Tie, Yanmei and Rigolo, Laura and Golby, Alexandra J and Golland, Polina},
  journal={Neuroimage},
  volume={103},
  pages={462--475},
  year={2014},  
  doi = {10.1016/j.neuroimage.2014.08.029},
}

@inproceedings{larsson2015simple,
  title={A Simple Method for Subspace Estimation with Corrupted Columns},
  author={Larsson, Viktor and Olsson, Carl and Kahl, Fredrik},
  booktitle={2015 IEEE International Conference on Computer Vision Workshop (ICCVW)}, 
  pages={841-849},
  year={2015},
  doi={10.1109/ICCVW.2015.113}
}

@article{li2023generalized,
author = {Li, Zihao and Lim, Lek-Heng},
title = {Generalized Matrix Nearness Problems},
journal = {SIAM J. Matrix Anal. Appl.},
fjournal = {SIAM Journal on Matrix Analysis and Applications},
volume = {44},
number = {4},
pages = {1709-1730},
year = {2023},
doi = {10.1137/22M1526034},
}

@inproceedings{mineiro2015fast,
  title={Fast label embeddings via randomized linear algebra},
  author={Mineiro, Paul and Karampatziakis, Nikos},
  booktitle={Machine Learning and Knowledge Discovery in Databases},
  pages={37--51},
  year={2015},
  editor={Appice, Annalisa and Rodrigues, Pedro Pereira and Santos Costa, V{\'i}tor and Soares, Carlos and Gama, Jo{\~a}o and Jorge, Al{\'i}pio},
 publisher={Springer, Cham},
  doi = {10.1007/978-3-319-23528-8_3}
}

@unpublished{mollenhauer_learning_2023,
    title = {Learning linear operators: {Infinite}-dimensional regression as a well-behaved non-compact inverse problem},
    year={2024},
    doi = {10.48550/arXiv.2211.08875},
    author = {Mollenhauer, Mattes and M\"{u}cke, Nicole and Sullivan, T. J.},
    note={arXiv:2211.08875}
}

@article{saibaba2016hoid,
  title={{HOID}: higher order interpolatory decomposition for tensors based on {Tucker} representation},
  author={Saibaba, Arvind K},
  fjournal={SIAM Journal on Matrix Analysis and Applications},
  journal={SIAM J. Matrix Anal. Appl.},
  volume={37},
  number={3},
  pages={1223--1249},
  year={2016},
  volume = {37},
    doi = {10.1137/15M1048628},
}

@article{sondermann1986best,
  title={Best approximate solutions to matrix equations under rank restrictions},
  author={Sondermann, Dieter},
  journal={Statistische Hefte},
  volume={27},
  number={1},
  pages={57--66},
  year={1986},
  doi={10.1007/BF02932555},
}

@inproceedings{song2017low,
  title        = {Low rank approximation with entrywise \(\ell{}_{\mbox{1}}\)-norm error},
  editor       = {Hamed Hatami and Pierre McKenzie and Valerie King},
  author={Song, Zhao and Woodruff, David P and Zhong, Peilin},
  booktitle    = {Proceedings of the 49th Annual {ACM} {SIGACT} Symposium on Theory of Computing, {STOC} 2017, Montreal, QC, Canada, June 19-23, 2017},
  publisher    = {{ACM}},
  pages={688--701},
  doi          = {10.1145/3055399.3055431},
  year={2017}
}

@article{soto2021error,
  title={Error analysis of the generalized low-rank matrix approximation},
  author={Soto-Quiros, Pablo},
  fjournal={Electronic Journal of Linear Algebra},
  journal={Electron. J. Linear Algebra},
  volume={37},
  pages={544--548},
  year={2021},
  doi={10.13001/ela.2021.5961}
}

@article{soto2023fast,
  title={Fast multiple rank-constrained matrix approximation},
  author={Soto-Quiros, Pablo and Chavarr{\'\i}a-Molina, Jeffry and Fallas-Monge, Juan Jos{\'e} and Torokhti, Anatoli},
  journal={SeMA Journal},
  pages={1--23},
  year={2023},
  doi={10.1007/s40324-023-00340-6},
}

@article{soto2016generalized,
   title={Generalized {Brillinger}-like transforms},
  author={Torokhti, Anatoli and Soto-Quiros, Pablo},
  journal={IEEE Signal Process. Lett.},
  fjournal={IEEE Signal Processing Letters},
  volume={23},
  number={6},
  pages={843--847},
  year={2016},
  doi={10.1109/LSP.2016.2556714}
}

@article{soto2022regularized,
  title={A Regularized Alternating Least-Squares Method for Minimizing a Sum of Squared {Euclidean} Norms with Rank Constraint},
  author={Soto-Quiros, Pablo and others},
  fjournal={Journal of Applied Mathematics},
  journal={J. Appl. Math.},
  pages={4838182},
  volume={2022},
  number={1},
  year={2022},
  doi = {10.1155/2022/4838182},
}

@article{soto2017optimal,
  title={Optimal transforms of random vectors: The case of successive optimizations},
  author={Soto-Quiros, Pablo and Torokhti, Anatoli},
  fjournal={Signal Processing},
  journal={Signal Process.},
  volume={132},
  pages={183--196},
  year={2017},
  doi={10.1016/j.sigpro.2016.09.020}
}

@article{soto2019improvement,
  title={Improvement in accuracy for dimensionality reduction and reconstruction of noisy signals. {Part} {II}: the case of signal samples},
  author={Soto-Quiros, Pablo and Torokhti, Anatoli},
  fjournal={Signal Processing},
  journal={Signal Process.},
  volume={154},
  pages={272--279},
  year={2019},
    doi = {10.1016/j.sigpro.2018.09.020},
}

@unpublished{soto2021data,
  title={Data Compression: Multi-Term Approach},
  author={Soto-Quiros, Pablo and Torokhti, Anatoli},
  doi={10.48550/arXiv.2111.05775},
  year={2021},
  note={arXiv:2111.05775}
}

@unpublished{soto2021second,
  title={Second Degree Model for Multi-Compression and Recovery of Distributed Signals},
  author={Soto-Quiros, Pablo and Torokhti, Anatoli and Miklavcic, Stanley J},
  doi={10.48550/arXiv.2111.03614},
  year={2021},
  note={arXiv:2111.03614},
}

@article{sou2012generalized,
  title={On generalized matrix approximation problem in the spectral norm},
  author={Sou, Kin Cheong and Rantzer, Anders},
  fjournal={Linear algebra and its applications},
  journal={Linear Algebra Appl.},
  volume={436},
  number={7},
  pages={2331--2341},
  year={2012},
  doi = {10.1016/j.laa.2011.10.009},
}

@article{spantini_optimal_2015,
	title = {Optimal {Low}-rank {Approximations} of {Bayesian} {Linear} {Inverse} {Problems}},
	volume = {37},
    number = {6},
    pages = {A2451-A2487},
    year = {2015},
    doi = {10.1137/140977308},
	fjournal = {SIAM Journal on Scientific Computing},
	journal = {SIAM J. Sci. Comput.},
    publisher={SIAM},
	author = {Spantini, Alessio and Solonen, Antti and Cui, Tiangang and Martin, James and Tenorio, Luis and Marzouk, Youssef},
}

@article{spantini2017goal,
  title={Goal-oriented optimal approximations of {Bayesian} linear inverse problems},
  author={Spantini, Alessio and Cui, Tiangang and Willcox, Karen and Tenorio, Luis and Marzouk, Youssef},
  fjournal={SIAM Journal on Scientific Computing},
  journal = {SIAM J. Sci. Comput.},
  volume={39},
  number={5},
  pages={S167--S196},
  year={2017},
  volume = {39},
  doi = {10.1137/16M1082123},
}

@unpublished{torokhti2021new,
  title={A new technique for compression of data sets},
  author={Torokhti, Anatoli},
  doi={10.48550/arXiv.2111.06572},
  year={2021},
  note={arXiv:2111.06572},
}

@article{torokhti2009towards,
  title={Towards theory of generic principal component analysis},
  author={Torokhti, Anatoli and Friedland, Shmuel},
  fjournal={Journal of Multivariate Analysis},
  journal={J. Multivariate Anal.},
  volume={100},
  number={4},
  pages={661--669},
  year={2009},
  doi={10.1016/j.jmva.2008.07.005},
}

@inproceedings{torokhti2007towards,
  title={Towards generic theory of data compression},
  author={Torokhti, Anatoli and Friedland, Shmuel and Howlett, Phil},
  booktitle={2007 IEEE International Symposium on Information Theory},
  pages={291--295},
  year={2007},
  doi={10.1109/ISIT.2007.4557241}
}

@book{torokhti2007computational,
  title={Computational Methods for Modelling of Nonlinear Systems},
  author={Torokhti, Anatoli and Howlett, Phil},
  year={2007},
  address={Amsterdam},
  publisher={Elsevier},
  series = {Mathematics in Science and Engineering},
  volume = {212},
  pages = {xii+397},
  url={https://www.sciencedirect.com/bookseries/mathematics-in-science-and-engineering/vol/212/suppl/C}
}

@article{torokhti2009filtering,
  title={Filtering and compression for infinite sets of stochastic signals},
  author={Torokhti, Anatoli and Howlett, Phil},
  fjournal={Signal Processing},
  journal={Signal Process.},
  volume={89},
  number={3},
  pages={291--304},
  year={2009},
  doi={10.1016/j.sigpro.2008.08.011},
}

@article{torokhti2023optimal,
  title={Optimal estimation of distributed highly noisy signals within {KLT}-{Wiener} archetype},
  author={Torokhti, Anatoli and Howlett, Phil},
  fjournal={Digital Signal Processing},
  journal={Digit. Signal Process.},
  volume={143},
  pages={104225},
  year={2023},
  doi={10.1016/j.dsp.2023.104225}
}

@article{torokhti2009generic,
  title={Generic weighted filtering of stochastic signals},
  author={Torokhti, Anatoli and Manton, Jonathan H},
  fjournal={IEEE Transactions on Signal Processing},
  journal={IEEE Trans. Signal Process.},
  volume={57},
  number={12},
  pages={4675--4685},
  year={2009},
  doi={10.1109/TSP.2009.2027463}
}

@inproceedings{torokhti2012optimal,
  title={Optimal Multidimensional Signal Processing in Wireless Sensor Networks.},
  author={Torokhti, Anatoli and Miklavcic, Stan},
  booktitle={Proceedings of the International Conference on Signal Processing and Multimedia Applications and Wireless Information Networks and Systems
(SIGMAP-2012)},
  pages={126--129},
  year={2012},
  doi={10.5220/0004056101260129},
  publisher={Scitepress}
}

@unpublished{torokhti2018best,
  title={Best approximations of non-linear mappings: Method of optimal injections},
  author={Torokhti, Anatoli and Soto-Quiros, Pablo},
  doi={10.48550/arXiv.1811.03125},
  year={2018},
  note={arXiv:1811.03125},
}

@article{torokhti2019improvement,
  title={Improvement in accuracy for dimensionality reduction and reconstruction of noisy signals. {Part I}: the case of random signals},
  author={Torokhti, Anatoli and Soto-Quiros, Pablo},
  fjournal={Signal Processing},
  journal={Signal Process.},
  volume={154},
  pages={338--349},
  year={2019},
  doi = {10.1016/j.sigpro.2018.09.021},
}

@article{torokhti2023matrix,
  title={Matrix Approximation by a Sum of Matrix Products},
  author={Torokhti, Anatoli and Soto-Quiros, Pablo and Ejov, Vladimir},
  fjournal={International Journal of Applied and Computational Mathematics},
  journal={Int. J. Appl. Comput. Math.},
  volume={9},
  number={6},
  pages={129},
  year={2023},
  doi={10.1007/s40819-023-01612-5}
}

@unpublished{turri_randomized_2023,
	title = {A randomized algorithm to solve reduced rank operator regression},
	doi = {10.48550/arXiv.2312.17348},
	author = {Turri, Giacomo and Kostic, Vladimir and Novelli, Pietro and Pontil, Massimiliano},
	note={arXiv:2312.17348},
	year = {2023},
}

@inproceedings{xiang2012optimal,
  title={Optimal exact least squares rank minimization},
  author={Xiang, Shuo and Zhu, Yunzhang and Shen, Xiaotong and Ye, Jieping},
  booktitle={Proceedings of the 18th ACM SIGKDD international conference on Knowledge discovery and data mining},
  pages={480--488},
  year={2012},
  publisher = {Association for Computing Machinery},
  doi = {10.1145/2339530.2339609},
}

@phdthesis{yang2020towards,
  title={Towards Better Understanding of Algorithms and Complexity of Some Learning Problems},
  author={Yang, Xin},
  year={2020},
  school={University of Washington},
  url={https://digital.lib.washington.edu/server/api/core/bitstreams/8adb71f1-477b-4b89-85b4-79249d9954b3/content}
}

@inproceedings{yu2012rank,
  title={Rank/norm regularization with closed-form solutions: Application to subspace clustering},
  author={Yu, Yao-Liang and Schuurmans, Dale},
  editor       = {F{\'{a}}bio Gagliardi Cozman and Avi Pfeffer},
  booktitle    = {{UAI} 2011, Proceedings of the Twenty-Seventh Conference on Uncertainty in Artificial Intelligence, Barcelona, Spain, July 14-17, 2011},
  pages        = {778--785},
  publisher    = {{AUAI} Press},
  year         = {2011},
  doi          = {10.48550/arXiv.1202.3772},
}

@inproceedings{zhao2016predictive,
  title={Predictive Collaborative Filtering with Side Information.},
  author={Zhao, Feipeng and Xiao, Min and Guo, Yuhong},
  booktitle={Proceedings of the Twenty-Fifth International Joint Conference on Artificial Intelligence, New York, USA, 9–15 July 2016},
  editor = {Subbarao Kambhampati},
  publisher    = {{AAAI} Press},
  pages={2385--2391},
  year={2016},
  url={https://www.ijcai.org/Proceedings/16/Papers/340.pdf}
}
\end{document}